\documentclass[12pt,reqno]{amsart}
\usepackage{amsmath,a4wide}
\usepackage{xfrac}
\usepackage{stmaryrd,mathrsfs,bm,amsthm,mathtools,yfonts,amssymb,color}
\usepackage{xcolor}
\usepackage{courier}

\begingroup
\newtheorem{theorem}{Theorem}[section]
\newtheorem{lemma}[theorem]{Lemma}
\newtheorem{proposition}[theorem]{Proposition}
\newtheorem{corollary}[theorem]{Corollary}
\endgroup

\theoremstyle{definition}
\newtheorem{definition}[theorem]{Definition}
\newtheorem{remark}[theorem]{Remark}
\newtheorem{ipotesi}[theorem]{Assumption}

\numberwithin{equation}{section}
\numberwithin{subsection}{section}

\newcommand{\dv}{{\textup {div}}}
\newcommand{\ph}{\varphi}

\newcommand{\sing}{{\rm Sing}}
\newcommand{\reg}{{\rm Reg}}
\newcommand\supp{{\rm spt}}
\newcommand\res{\mathop{\hbox{\vrule height 7pt width .3pt depth 0pt
\vrule height .3pt width 5pt depth 0pt}}\nolimits}
\newcommand\ser{\mathop{\hbox{\vrule height .3pt width 5pt depth 0pt
\vrule height 7pt width .3pt depth 0pt}}\nolimits}

\newcommand{\mass}{{\mathbf{M}}}
\newcommand{\cone}{{\times\hspace{-0.6em}\times\,}}
\newcommand{\bOmega}{{\mathbf{\Omega}}}
\newcommand{\de}{\partial}
\newcommand{\bI}{{\mathbf{I}}}
\def\a#1{\left\llbracket{#1}\right\rrbracket}
\newcommand{\bE}{{\mathbf{E}}}
\newcommand{\gr}{{\rm Gr}}

\newcommand{\im}{{\rm Im}}
\newcommand{\bT}{\mathbf{T}}
\newcommand{\cH}{{\mathcal{H}}}

\def\I#1{{\mathcal{A}}_{#1}}
\newcommand{\Iq}{{\mathcal{A}}_Q}
\newcommand{\Iqs}{{\mathcal{A}}_Q(\R^{n})}
\newcommand{\bG}{{\mathbf{G}}}
\newcommand{\cG}{{\mathcal{G}}}
\newcommand{\etaa}{{\bm{\eta}}}
\newcommand{\D}{\textup{Dir}}
\newcommand{\card}{{\textrm{card}}}

\newcommand{\bB}{{\mathbf{B}}}
\newcommand{\bC}{{\mathbf{C}}}

\def\Xint#1{\mathchoice
{\XXint\displaystyle\textstyle{#1}}%
{\XXint\textstyle\scriptstyle{#1}}%
{\XXint\scriptstyle\scriptscriptstyle{#1}}%
{\XXint\scriptscriptstyle\scriptscriptstyle{#1}}%
\!\int}
\def\XXint#1#2#3{{\setbox0=\hbox{$#1{#2#3}{\int}$ }
\vcenter{\hbox{$#2#3$ }}\kern-.6\wd0}}
\def\mint{\Xint-}

\newcommand\N{{\mathbb N}}

\newcommand\C{{\mathbb C}}
\newcommand\R{{\mathbb R}}

\newcommand{\eps}{{\varepsilon}}
\newcommand{\bA}{\mathbf{A}}

\newcommand{\Lip}{{\rm {Lip}}}

\newcommand{\p}{{\mathbf{p}}}

\renewcommand{\P}{{\mathbf{P}}}

\newcommand{\gira}{{\mathfrak{B}}}
\newcommand{\disco}{{\mathfrak{D}}}
\newcommand{\bV}{{\mathbf{V}}}

\newcommand\bmo{{\bm m}_0}

\newcommand{\cM}{{\mathcal{M}}}

\newcommand{\bL}{{\mathbf{L}}}
\newcommand{\phii}{{\bm{\varphi}}}
\newcommand{\Phii}{{\bm{\Phi}}}
\newcommand{\bc}{{\bm{c}}}

\newcommand{\cC}{{\mathcal{C}}}

\newcommand{\bW}{{\mathbf{W}}}

\newcommand{\Psii}{{\bm{\Psi}}}
\newcommand{\bK}{{\mathbf{K}}}
\newcommand{\bD}{{\mathbf{D}}}
\newcommand{\bH}{{\mathbf{H}}}
\newcommand{\bbI}{{\mathbf{\bar I}}}
\newcommand{\bF}{{\mathbf{F}}}
\newcommand{\bSigma}{{\mathbf{\Sigma}}}

\newcommand{\bLambda}{{\mathbf{\Lambda}}}

\DeclareMathAlphabet{\mathpzc}{OT1}{pzc}{m}{it}
\newcommand{\NN}{{\mathpzc{N}}}
\newcommand{\FF}{{\mathpzc{F}}}
\newcommand{\LL}{{\mathpzc{L}}}
\newcommand{\GG}{{\mathpzc{G}}}
\newcommand{\HH}{{\mathpzc{H}}}
\newcommand{\UU}{{\mathpzc{U}}}
\newcommand{\VV}{{\mathpzc{V}}}

\newcommand{\Thetaa}{{\bm{\Theta}}}

\title[$2$-dimensional regularity theory]{Regularity theory for $2$-dimensional almost minimal currents III: blowup}
\author{Camillo De Lellis, Emanuele Spadaro and Luca Spolaor}

\begin{document}

\begin{abstract}
We analyze the asymptotic behavior of a $2$-dimensional integral current which is almost minimizing in a suitable sense at a singular point. Our analysis is the second half of an argument which shows the discreteness of
the singular set for the following three classes of $2$-dimensional currents: area minimizing in Riemannian manifolds, semicalibrated and spherical cross sections of $3$-dimensional
area minimizing cones.
\end{abstract}

\maketitle

This paper is the fourth and last in a series of works aimed at establishing an optimal regularity theory
for $2$-dimensional integral currents which are almost minimizing in a suitable sense. Building upon the monumental
work of Almgren \cite{Alm}, Chang in \cite{Chang} established that $2$-dimensional area-minimizing currents in
Riemannian manifolds are classical minimal surfaces, namely they are regular (in the interior) except for a discrete
set of branching singularities. The argument of Chang is however not entirely complete since a key starting point 
of his analysis, the existence of the so-called ``branched center manifold'', is only sketched in the appendix of
\cite{Chang} and requires the understanding (and a suitable modification) of the most involved portion of the monograph \cite{Alm}. 

An alternative proof of Chang's theorem has been found by Rivi\`ere and Tian in \cite{RT1} for the special case
of $J$-holomorphic curves. Later on the approach of Rivi\`ere and Tian has been generalized by Bellettini and Rivi\`ere
in \cite{BeRi} to handle a case which is not covered by \cite{Chang}, namely that of special Legendrian cycles in
$\mathbb S^5$.

Meanwhile the first and second author revisited Almgren's theory giving a much shorter version of his program
for proving that area-minimizing currents are regular up to a set of Hausdorff codimension $2$, cf. \cite{DS1,DS2,DS3,DS4,DS5}.
In this note and its companion papers \cite{DSS2,DSS3} we build upon the latter works in order to give a
complete regularity theory which includes both the theorems of Chang and Bellettini-Rivi\`ere as special cases. 
In order to be more precise, we introduce the following terminology (cf. \cite[Definition 0.3]{DSS1}).

\begin{definition}\label{d:semicalibrated}
Let $\Sigma \subset R^{m+n}$ be a $C^2$ submanifold and $U\subset \R^{m+n}$ an open set.
\begin{itemize}
  \item[(a)] An $m$-dimensional integral current $T$ with finite mass and $\supp (T)\subset \Sigma\cap U$ is area-minimizing in $\Sigma\cap U$
if $\mass(T + \partial S)\geq \mass(T)$ for any $m+1$-dimensional integral current $S$ with $\supp (S) \subset \subset \Sigma\cap U$.
 \item[(b)] A semicalibration (in $\Sigma$) is a $C^1$ $m$-form $\omega$ on $\Sigma$ such that 
  $\|\omega_x\|_c \leq 1$ at every $x\in \Sigma$, where $\|\cdot \|_c$ denotes the comass norm on $\Lambda^m T_x \Sigma$. 
  An $m$-dimensional integral current $T$ with $\supp (T)\subset \Sigma$ is {\em semicalibrated} by $\omega$ if $\omega_x (\vec{T}) = 1$ for $\|T\|$-a.e. $x$.
  \item[(c)]  An $m$-dimensional integral current $T$ supported in $\partial \bB_{\bar R} (p) \subset \R^{m+n}$ is a {\em spherical cross-section of an area-minimizing cone} if ${p\cone T}$ is area-minimizing. 
\end{itemize}
\end{definition}

In what follows, given an integer rectifiable current $T$, we denote by $\reg (T)$ the subset of $\supp (T)\setminus \supp (\partial T)$ consisting of those points $x$ for which there is a neighborhood $U$ such that $T\res U$ is a (costant multiple of) a regular submanifold. Correspondingly, $\sing (T)$ is
the set $\supp (T)\setminus (\supp (\partial T)\cup \reg (T))$. Observe that $\reg (T)$ is relatively open in $\supp (T) \setminus \supp (\partial T)$
and thus $\sing (T)$ is relatively closed.
The main result of this and the works \cite{DSS2,DSS3} is then the following

\begin{theorem}\label{t:finale}
Let $m=2$ and $T$ be as in (a), (b) or (c) of Definition \ref{d:semicalibrated}. Assume in addition that
$\Sigma$ is of class $C^{3,\eps_0}$ (in case (a) and (b)) and $\omega$ of class $C^{2,\eps_0}$ (in case (b)) for some positive $\eps_0$. Then $\sing (T)$
is discrete.
\end{theorem}

Clearly Chang's result is covered by case (a). As for the case of special Lagrangian cycles considered by Bellettini and Rivi\`ere in
\cite{BeRi} observe that they form a special subclass of both (b) and (c). Indeed these cycles arise as spherical cross-sections of $3$-dimensional special lagrangian cones: as such they are then spherical cross sections of area-minimizing cones but they are also semicalibrated by a specific smooth form
on $\mathbb S^5$.

Following the Almgren-Chang program, Theorem \ref{t:finale} will be established through a suitable ``blow-up argument'': this argument is presented here but requires
several tools. The first important tool is the theory of multiple-valued functions, for which we
will use the results and terminology of the papers \cite{DS1,DS2}. The second tool is a suitable approximation result for area-minimizing currents with graphs of multiple valued functions, which for the case at hand has been established in the preceding
note \cite{DSS2}. 
The last tool is the so-called ``branched center manifold'': this has been constructed in the paper \cite{DSS3}. We note in passing that all our arguments use heavily the uniqueness
of tangent cones for $T$. This result is a, by now classical, theorem of White for area-minimizing $2$-dimensional currents in the euclidean
space, cf. \cite{Wh}. Chang extended it to case (a) in the appendix of \cite{Chang}, whereas Pumberger and Rivi\`ere covered case (b) in \cite{PuRi}. 
A general derivation of these results for a wide class of almost minimizers has been given in \cite{DSS1}: the theorems in there
cover, in particular, all the cases of Definition \ref{d:semicalibrated}.

The proof of Theorem \ref{t:finale} is based, as in \cite{Chang}, on an induction statement, cf. Theorem \ref{t:induttivo} below. Although Theorem \ref{t:induttivo} is already stated in \cite{DSS3}, we will recall it in the next section, where we also show how it implies
Theorem \ref{t:finale}. This and the previous paper \cite{DSS3} can be hence thought as the two halves in the proof of Theorem \ref{t:induttivo}.  After introducing some terminology, in Section \ref{s:2_pezzi} we will state the first half in Theorem \ref{t:cm} (which is proved in \cite{DSS3}) and the second half in Theorem \ref{t:bu}, and we will then show how they fit together to prove Theorem \ref{t:induttivo}. The remaining sections (the biggest portion of this note) are then dedicated to prove Theorem \ref{t:bu}.

\smallskip

{\bf Acknowledgments}  The research of Camillo De Lellis and Luca Spolaor has been supported by the ERC grant RAM (Regularity for Area Minimizing currents), ERC 306247.

\section{Preliminaries and the main induction statement}

\subsection{Basic notation and first main assumptions}
For the notation concerning submanifolds
$\Sigma\subset \R^{2+n}$ we refer to \cite[Section 1]{DS3}. With $\bB_r (p)$ and $B_r (x)$ we denote, respectively, the open ball with radius $r$ and center $p$ in $\mathbb R^{2+n}$ and the open ball with radius $r$ and center $x$ in $\mathbb R^2$.  $\bC_r (p)$ and $\bC_r (x)$ will always denote the cylinder $B_r (x)\times \mathbb R^n$, where $p=(x,y)\in \R^2\times \R^n$. We will often need to consider cylinders whose bases are parallel to other $2$-dimensional planes, as well as balls in $m$-dimensional affine planes. We then introduce the notation $B_r (p, \pi)$ for $\bB_r (p) \cap (p+\pi)$ and $\bC_r (p, \pi)$ for $B_r (p, \pi) + \pi^\perp$. 
$e_i$ will denote the unit vectors in the standard basis,
$\pi_0$ the (oriented) plane $\R^2\times \{0\}$ and $\vec \pi_0$
the $2$-vector $e_1\wedge e_2$ orienting it.
Given an $m$-dimensional plane $\pi$, we denote by $\p_\pi$ and $\p_\pi^\perp$ the
orthogonal projections onto, respectively, $\pi$ and its orthogonal complement $\pi^\perp$. 
For what concerns integral currents we use the definitions and the notation of \cite{Sim}. Since $\pi$ is used recurrently for $2$-dimensional planes, the $2$-dimensional area of the unit circle in $\R^2$ will be denoted by $\omega_2$. 

By \cite[Lemma 1.1]{DSS2} in case (b) we can assume, without loss of generality, that the ambient manifold $\Sigma$ coincides with
the euclidean space $\R^{2+n}$. In the rest of the paper we will therefore always make the following

\begin{ipotesi}\label{ipotesi_base}
$T$ is an integral current of dimension $2$ with bounded support and it satisfies one of the three conditions (a), (b) or (c) in Definition \ref{d:semicalibrated}. Moreover
\begin{itemize}
\item In case (a),
$\Sigma\subset\R^{2+n}$ is a $C^{3, \eps_0}$
submanifold of dimension $2 + \bar n = 2 + n - l$, which is the graph of an entire
function $\Psi: \R^{2+\bar n}\to \R^l$ and satisfies the bounds
\begin{equation}\label{e:Sigma}
\|D \Psi\|_0 \leq c_0 \quad \mbox{and}\quad \bA := \|A_\Sigma\|_0
\leq c_0,
\end{equation}
where $c_0$ is a positive (small) dimensional constant and $\eps_0\in ]0,1[$. 
\item In case (b) we assume that $\Sigma = \R^{2+n}$ and that the semicalibrating form $\omega$ is $C^{2, \eps_0}$. 
\item In case (c) we assume that $T$ is supported in $\Sigma = \partial \bB_R (p_0)$ for some $p_0$ with $|p_0|=R$, so that $0\in \partial \bB_R (p_0)$. We assume also that $T_0 \partial \bB_R (p_0)$ is $\R^{2+n-1}$ (namely $p_0= (0, \ldots, 0, \pm |p_0|)$ and we let $\Psi: \R^{2+n-1}\to \R$ be a smooth extension to the whole space of the function which describes $\Sigma$ in $\bB_2 (0)$. We assume then that \eqref{e:Sigma} holds, which is equivalent to the requirement that $R^{-1}$ be sufficiently small.
\end{itemize}
\end{ipotesi}

In addition, since the conclusion of Theorem~\ref{t:finale} is local, by \cite[Proposition 0.4]{DSS1}
we can also assume to be always in the following situation.

\begin{ipotesi}\label{ipotesi_base_2}
In addition to Assumption \ref{ipotesi_base} we assume the following:
\begin{itemize}
\item[(i)] $\partial T \res \bC_2 (0, \pi_0)=0$;
\item[(ii)] $0\in \supp (T)$ and the tangent cone
at $0$ is given by $\Theta (T, 0) \a{\pi_0}$ where $\Theta (T, 0)\in \N\setminus \{0\}$;
\item[(iii)] $T$ is irreducible in any neighborhood $U$ of $0$ in the following sense:
it is not possible to find $S$, $Z$ non-zero integer rectifiable
currents in $U$ with $\partial S = \partial Z =0$ (in $U$), $T= S+Z$ and $\supp (S)\cap \supp (Z)=\{0\}$.
\end{itemize}
\end{ipotesi}

In order to justify point (iii), observe that we can argue as in the proof of \cite[Theorem 3.1]{DSS1}: assuming that in a certain neighborhood $U$ 
there is a decomposition $T= S+Z$ as above, it follows from \cite[Proposition 2.2]{DSS1} that both $S$ and $Z$ fall in one of the classes of Definition \ref{d:semicalibrated}. In turn this implies that $\Theta (S, 0), \Theta (Z, 0)\in \N\setminus \{0\}$ and thus $\Theta (S, 0) < \Theta (T, 0)$.
We can then replace $T$ with either $S$ or $Z$. Let $T_1=S$ and argue similarly if it is not irreducibile: obviously we can apply one more time the argument above and find a $T_2$ which satisfies all the requirements and has $0<\Theta (T_2, 0) < \Theta (T_1, 0)$. This process must stop after at most $N = \Theta (T, 0)$ steps: the final current is then necessarily irreducible.

\subsection{Branching model} We next introduce an object which will play a key role in the rest of our work, because it is the basic local model of the
singular behavior of a $2$-dimensional area-minimizing current: for each positive natural number $Q$ we will denote by $\gira_{Q,\rho}$ the flat Riemann surface which is a disk with a conical singularity, in the origin, of angle $2\pi Q$ and radius $\rho>0$. More precisely we have

\begin{definition}\label{d:Riemann_surface}
$\gira_{Q,\rho}$ is topologically an open $2$-dimensional disk, which we identify with the topological space $\{(z,w)\in \mathbb C^2 : w^Q=z, |z|<\rho\}$. For each $(z_0, w_0)\neq 0$ in $\gira_{Q,\rho}$ we consider the connected component $\disco (z_0, w_0)$ of $\gira_{Q,\rho}\cap \{(z,w):|z-z_0| < |z_0|/2\}$ which contains $(z_0, w_0)$. We then consider the smooth manifold given by the atlas 
\[
\{(\disco (z,w)), (x_1, x_2)): (z,w)\in \gira_{Q,\rho} \setminus \{0\}\}\, ,
\] 
where $(x_1, x_2)$ is the function which gives the real and imaginary part of the first complex coordinate of a generic point of $\gira_{Q,\rho}$. On such smooth manifold we consider the following flat Riemannian metric: on each $\disco (z,w)$ with the chart $(x_1, x_2)$ the metric tensor is the usual euclidean one $dx_1^2+dx_2^2$. Such metric will be called the {\em canonical flat metric} and denoted by $e_Q$.
The coordinates $(x_1, x_2) = z$ will be called {\em standard flat coordinates}.
\end{definition}

When $Q=1$ we can extend smoothly the metric tensor to the origin and we obtain the usual euclidean $2$-dimensional disk.
For $Q>1$ the metric tensor does not extend smoothly to $0$, but we can nonetheless complete the induced geodesic distance on $\gira_{Q,\rho}$ in a neighborhood of $0$: for $(z,w)\neq 0$ the distance to the origin will then correspond to $|z|$. The resulting metric space is a well-known object in the literature, namely a flat Riemann surface with an isolated conical singularity at the origin (see for instance \cite{Zorich}). Note that for each $z_0$ and $0 <r\leq \min \{|z_0|, \rho -|z_0|\} $ the set $\gira_{Q, \rho}\cap \{|z-z_0|<r\}$ consists then of $Q$ nonintersecting $2$-dimensional disks, each of which is a geodesic ball of $\gira_{Q, \rho}$ with radius $r$ and center $(z_0, w_i)$ for some $w_i\in \C$ with $w_i^Q = z_0$. We then denote each of them by $B_r (z_0,w_i)$ and treat it as a standard disk in the euclidean $2$-dimensional plane (which is correct from the metric point of view). We use however the same notation for the distance disk $B_r (0)$, namely for the set $\{(z,w): |z|<r\}$, although the latter is {\em not isometric} to the standard euclidean disk. Since this might be create some ambiguity, we will use the specification $\R^2 \supset B_r (0)$ when referring to the standard disk in $\R^2$. 

\subsection{Admissible $Q$-branchings}
When one of (or both) the parameters $Q$ and $\rho$ are clear from the context, the  corresponding subscript (or both) will be omitted. We will consider repeatedly functions $u$ defined on $\gira$. We will always treat each point of $\gira$ as an element of $\C^2$, mostly using $z$ and $w$ for the horizontal and vertical complex coordinates. Often $\C$ will be identified with $\R^2$ and thus the coordinate $z$ will be treated as a two-dimensional real vector, avoiding the more cumbersome notation $(x_1, x_2)$.

\begin{definition}[$Q$-branchings]\label{d:admissible}
Let $\alpha\in ]0,1[$, $b>1$, $Q\in \N\setminus \{0\}$ and $n\in \N\setminus \{0\}$. 
An admissible $\alpha$-smooth and $b$-separated $Q$-branching in $\R^{2+n}$ (shortly a $Q$-branching) is the graph
\begin{equation}
\gr (u) := \{(z, u(z,w)): (z,w) \in \gira_{Q, 2\rho}\} \subset \R^{2+n}\, 
\end{equation}
of a map $u: \gira_{Q,2\rho} \to \mathbb R^n$ satisfying the following assumptions.
For some constants $C_i>0$ we have
\begin{itemize}
\item $u$ is continuous, $u\in C^{3,\alpha}$ on $\gira_{Q, \rho} \setminus \{0\}$ and $u(0)=0$;
\item $|D^j u (z,w)|\leq C_i |z|^{1-j+\alpha}$ $\forall (z,w)\neq 0$ and $j\in \{0,1,2,3\}$;
\item $[D^3 u]_{\alpha, B_r (z,w)} \leq C_i |z|^{-2}$ for every $(z,w)\neq 0$ with $|z|=2r$;
\item If $Q>1$, then there is a positive constant $c_s\in ]0,1[$ such that
\begin{equation}\label{e:separation}
\min \{|u (z,w)-u(z,w')|: w\neq w'\} \geq 4 c_s |z|^b \qquad \mbox{for all $(z,w)\neq 0$.}
\end{equation}
\end{itemize}
The map $\Phii (z,w):= (z, u (z,w))$ will be called the {\em graphical parametrization} of the $Q$-branching.
\end{definition}

Any $Q$-branching as in the Definition above is an immersed disk in $\R^{2+n}$ and can be given a natural structure as integer rectifiable current, which will be denoted by $\bG_u$. For $Q=1$ a map $u$ as in Definition \ref{d:admissible} is a (single valued) $C^{1,\alpha}$ map $u: \R^2\supset B_{2\rho} (0)\to \R^n$. Although the term branching is not appropriate in this case, the advantage of our setup is that $Q=1$ will not be a special case in the induction statement of Theorem \ref{t:induttivo} below.  Observe that for $Q>1$
the map $u$ can be thought as a $Q$-valued map $u: \R^2\supset B_\rho (0) \to \Iqs$, setting $u(z)= \sum_{(z,w_i)\in \gira} \a{u(z,w_i)}$ for $z\neq 0$ and $u (0) = Q\a{0}$. The notation $\gr (u)$ and $\bG_u$ is then coherent with the corresponding objects defined in \cite{DS2} for general $Q$-valued maps.

\subsection{The inductive statement} Before coming to the key inductive statement, we need to introduce some more terminology.

\begin{definition}[Horned Neighborhood]\label{d:horned} Let $\gr (u)$ be a $b$-separated $Q$-branching.
For every $a>b$ we define the {\em horned neighborhood} $\bV_{u,a}$ of $\gr (u)$ to be
\begin{equation}
\bV_{u, a} := \{(x,y)\in \R^2\times \R^n: \exists (x,w) \in \gira_{Q, 2\rho} \mbox{ with } |y-u(x,w)|< c_s |x|^a\}\, ,
\end{equation}
where $c_s$ is the constant in \eqref{e:separation}.
\end{definition}

\begin{definition}[Excess]\label{d:excess}
Given an $m$-dimensional current $T$ in $\mathbb R^{m+n}$ with finite mass, its {\em excess} in the ball $\bB_r (x)$ and in the cylinder $\bC_r (p, \pi')$ with respect to the $m$-plane $\pi$ are
\begin{align}
\bE(T,\bB_r (p),\pi) &:= \left(2\omega_m\,r^m\right)^{-1}\int_{\bB_r (p)} |\vec T - \vec \pi|^2 \, d\|T\|\\
\bE (T, \bC_r(p, \pi'), \pi) &:= \left(2\omega_m\,r^m\right)^{-1}\int_{\bC_r (p, \pi')} |\vec T - \vec \pi|^2 \, d\|T\|\, .
\end{align}
For cylinders we omit the third entry when $\pi=\pi'$, i.e. $\bE (T, \bC_r (p, \pi)) := \bE (T, \bC_r (p, \pi), \pi)$.
In order to define the spherical excess we consider $T$ as in Assumption \ref{ipotesi_base} and we say that
$\pi$ {\em optimizes the excess} of $T$ in a ball $\bB_r (x)$ if
\begin{itemize}
\item In case (b)
\begin{equation}\label{e:optimal_pi}
\bE(T,\bB_r (x)):=\min_\tau \bE (T, \bB_r (x), \tau) = \bE(T,\bB_r (x),\pi);
\end{equation} 
\item In case (a) and (c) $\pi\subset T_x \Sigma$ and 
\begin{equation}\label{e:optimal_pi_2}
\bE(T,\bB_r (x)):=\min_{\tau\subset T_x \Sigma} \bE (T, \bB_r (x), \tau) = \bE(T,\bB_r (x),\pi)\, .
\end{equation} 
\end{itemize}
\end{definition}
Note in particular that, in case (a) and (c), $\bE (T, \bB_r (x))$ differs from the quantity defined in \cite[Definition 1.1]{DS5}, where, although $\Sigma$ does not coincide with the ambient euclidean space, $\tau$ is allowed to vary among {\em all} planes, as in case (b). Thus a notation more consistent with that of \cite{DS5} would be, in case (a) and (c), $\bE^\Sigma (T, \bB_r (x))$. However, the difference is a minor one and we prefer to keep our notation simpler.

Our main induction assumption is then the following

\begin{ipotesi}[Inductive Assumption]\label{induttiva}
$T$ is as in Assumption \ref{ipotesi_base} and \ref{ipotesi_base_2}. 
For some constants $\bar Q\in \mathbb N\setminus \{0\}$ and $0<\alpha< \frac{1}{2\bar Q}$ there is an $\alpha$-admissible $\bar Q$-branching
$\gr (u)$ with $u: \gira_{\bar{Q}, 2}\to \R^n$ such that
\begin{itemize}
\item[(Sep)] If $\bar Q>1$, $u$ is $b$-separated for some $b>1$; a choice of some $b>1$ is fixed also in the case
$\bar Q =1$, although in this case the separation condition is empty.
\item[(Hor)] $\supp (T) \subset \bV_{u,a}\cup\{0\}$ for some $a>b$;
\item[(Dec)] There exist $\gamma >0$ and a $C_i>0$ with the following property. Let $p = (x_0,y_0)\in \supp (T)\cap \bC_{\sqrt{2}} (0)$ and $4 d := |x_0|>0$, let $V$ be the connected component of 
$\bV_{u,a}\cap \{(x,y): |x-x_0| < d\}$ containing $p$ and let $\pi (p)$ be the plane tangent to $\gr (u)$ at the only point of the form $(x_0, u(x_0, w_i))$ which is contained in $V$. Then
\begin{equation}\label{e:effetto_energia}
\bE (T\res V,\bB_\sigma (p), \pi (p)) \leq C_i^2 d^{2\gamma -2} \sigma^2 \qquad \forall \sigma\in \left[{\textstyle{\frac{1}{2}}}d^{(b+1)/2}, d\right]\, .
\end{equation}
\end{itemize}
\end{ipotesi}

The main inductive step is then the following theorem, where we denote by $T_{p,r}$ the rescaled current $(\iota_{p,r})_\sharp T$, through the map $\iota_{p,r} (q) := (q-p)/r$.

\begin{theorem}[Inductive statement]\label{t:induttivo}
Let $T$ be as in Assumption \ref{induttiva} for some $\bar Q = Q_0$. Then,
\begin{itemize}
\item[(a)] either $T$ is, in a neighborhood of $0$, a $Q$ multiple of a $\bar{Q}$-branching $\gr (v)$;
\item[(b)] or there are $r>0$ and $Q_1>Q$ such that $T_{0, r}$ satisfies Assumption \ref{induttiva} with $\bar Q = Q_1$.
\end{itemize}
\end{theorem}

Theorem \ref{t:finale} follows then easily from Theorem \ref{t:induttivo} and \cite{DSS1}. 

\subsection{Proof of Theorem \ref{t:finale}} As already mentioned, without loss of generality we can assume that Assumption \ref{ipotesi_base} holds, cf. \cite[Lemma 1.1]{DSS1} (the bounds on $\bA$ and $\Psi$ can be achieved by a simple scaling argument). Fix now a point $p$
in $\supp (T)\setminus \supp (\partial T)$. Our aim is to show that $T$ is regular in a punctured neighborhood of $p$. Without loss of generality we can assume that $p$ is the origin. Upon suitably decomposing $T$ in some neighborhood of $0$ we can easily assume that (I) in Assumption \ref{induttiva} holds, cf. the argument of Step 4 in the proof of \cite[Theorem 3.1]{DSS1}. Thus, upon suitably rescaling and rotating $T$ we can assume that $\pi_0$ is the unique tangent cone to $T$ at $0$, cf. \cite[Theorem 3.1]{DSS1}. In fact, by \cite[Theorem 3.1]{DSS1} $T$ satisfies Assumption \ref{induttiva} with $\bar{Q}=1$: it suffices to chose $u\equiv 0$ as admissible smooth branching. If $T$ were not regular in any punctured neighborhood of $0$, we could then apply Theorem \ref{t:induttivo} inductively to find a sequence of rescalings $T_{0, \rho_j}$ with $\rho_j\downarrow 0$ which satisfy Assumption \ref{induttiva} with $\bar Q=Q_j$ for some strictly increasing sequence of integers.
  It is however elementary that the density $\Theta (0, T)$ bounds $Q_j$ from above, which is a contradiction.

\section{The branched center manifold and the blow-up theorem}\label{s:2_pezzi}

From now on we fix $T$ satisfying Assumption \ref{induttiva}. Observe that, without loss of generality, we are always free to rescale homothetically our current $T$ with a factor larger than $1$ and ignore whatever portion falls outside $\bC_2 (0)$. We will do this several times, with factors which will be assumed to be sufficiently large. Hence, if we can prove that something holds in a sufficiently small neighborhood of $0$, then we can assume, withouth loss of generality, that it holds on $\bC_2$.
For this reason we can assume that the constants $C_i$ in Definition \ref{d:admissible} and Assumption \ref{induttiva} is as small as we want.
In turns this implies that there is a well-defined orthogonal projection $\P: \bV_{u,a}\cap \bC_1 \to \gr (u)\cap \bC_2$, which is a $C^{2,\alpha}$ map. 
We next recall \cite[Lemma 2.1]{DSS3}:

\begin{lemma}\label{l:density}
Let $T$ and $u$ be as in Assumption \ref{induttiva} for some $\bar  Q$. Then the nearest point projection $\P: \bV_{u,a} \cap \bC_1 \to \gr (u)$ is a well-defined $C^{2, \alpha}$ map. In addition there is $Q\in \N\setminus \{0\}$ such that $\Theta (0, T) = Q \bar{Q}$ and the unique tangent cone to $T$ at $0$ is $Q\bar{Q} \a{\pi_0}$. Finally,
after possibly rescaling $T$, $\Theta (p, T) \leq Q$ for every $p\in \bC_2\setminus \{0\}$ and, for every $x\in B_2 (0)$, each
connected component of $(x\times \R^n) \cap \bV_{u,a}$ contains at least one point of $\supp (T)$.
\end{lemma}

Since we will assume during the rest of the paper that the above discussion applies, we summarize the relevant conclusions in the following

\begin{ipotesi}\label{piu_fogli}
$T$ satisfies Assumption \ref{induttiva} for some $\bar{Q}$ and with $C_i$ sufficiently small. $Q\geq 1$
is an integer, $\Theta (0, T) = Q \bar{Q}$ and $\Theta (p, T) \leq Q$ for all $p\in \bC_2\setminus \{0\}$. 
\end{ipotesi}

The overall plan to prove Theorem \ref{t:induttivo} is then the following:
\begin{itemize}
\item[(CM)] We construct first a branched center manifold, i.e. a second admissible smooth branching $\phii$ on $\gira_{\bar{Q}}$, and a corresponding $Q$-valued map $N$ defined on the normal bundle of $\gr (\phii)$, which approximates $T$ with a very high degree of accuracy (in particular more accurately than $u$) and whose average $\etaa\circ N$ is very small;
\item[(BU)] Assuming that alternative (a) in Theorem \ref{t:induttivo} does not hold, we study the asymptotic behavior of $N$ around $0$ and use it to build a new admissible smooth branching $v$ on some $\gira_{k \bar{Q}}$ where $k\geq 2$ is a factor of $Q$: this map will then be the one sought in alternative (b) of Theorem \ref{t:induttivo} and a suitable rescaling of $T$ will lie in a horned neighborhood of its graph.
\end{itemize}
The first part of the program is the one achieved in \cite{DSS3}, whereas the second part will be completed in this note. Note that, when $Q=1$, from (BU) we will conclude that alternative (a) necessarily holds: this will be a simple corollary of the general case, but we observe that it could also be proved resorting to the classical Allard's regularity theorem.

\subsection{Smallness condition} In several occasions we will need that the ambient manifold $\Sigma$ is suitably flat and that the excess of the current $T$ is suitably small. This can, however, be easily achieved after scaling. More precisely we recall \cite[Lemma 2.3]{DSS3}. 

\begin{lemma}\label{l:piccolezza}
Let $T$  be as in the Assumptions \ref{induttiva} and \ref{piu_fogli}.
After possibly rescaling, rotating and modifying $\Sigma$ outside $\bC_2 (0)$ we can assume that, in case (a) and (c) of Definition \ref{d:semicalibrated},
\begin{itemize}
\item[(i)] $\Sigma$ is a complete submanifold of $\R^{2+n}$;
\item[(ii)] $T_0 \Sigma = \R^{2+\bar{n}}\times \{0\}$ and, $\forall p\in \Sigma$, $\Sigma$ is the graph of a $C^{3,\eps_0}$ map $\Psi_p: T_p \Sigma \to (T_p\Sigma)^\perp$.
\end{itemize}
Under these assumptions, we denote by $\bc$ and $\bmo$ the following quantities
\begin{align}
&\bc := \sup \{\|D\Psi_p\|_{C^{2,\eps_0}}:p\in \Sigma\}\qquad \mbox{in the cases (a) and (c) of Definition \ref{d:semicalibrated}}\\
&\bc :=\|d\omega\|_{C^{1,\eps_0}} \qquad \qquad\;\;\qquad\qquad\mbox{in case (b) of Definition \ref{d:semicalibrated}}\\
&\bmo:= \max \left\{ \bc^2, \bE (T, \bC_2, \pi_0), C_i^2, c_s^2\right\}\, ,
\end{align}
where $C_i$ and $c_s$ are the constants appearing in Definition \ref{d:admissible} and Assumption \ref{induttiva}. Then, for any $\eps_2>0$,
after possibly rescaling the current by a large factor, we can assume
\begin{equation}\label{e:smallness}
\bmo \leq \eps_2\, .
\end{equation}
\end{lemma}

In order to carry on the plan outlined in the previous subsection, it is convenient to use a different parametrization of $Q$-branchings.
If we remove the origin, any admissible $Q$-branching is a Riemannian submanifold of $\R^{2+n}\setminus \{0\}$: this gives a Riemannian tensor $g := \Phii^\sharp e$ (where $e$ denotes the euclidean metric on $\R^{2+n}$) on the puntured disk $\gira_{Q, 2\rho}\setminus \{0\}$. Note that in $(z,w)$ the difference between the metric tensor $g$ and the canonical flat metric $e_Q$ can be estimated by (a constant times) $|z|^{2\alpha}$: thus, as it happens for the canonical flat metric $e_Q$, when $Q>1$ it is not possible to extend the metric $g$ to the origin. However, using well-known arguments in differential geometry, we can find a conformal map from $\gira_{Q, r}$ onto a neighborhood of $0$ which maps the conical singularity of $\gira_{Q,r}$ in the conical singularity of the $Q$-branching. 
In fact, we need the following accurate estimates for such a map, cf. \cite[Proposition 2.4]{DSS3}:

\begin{proposition}[Conformal parametrization]\label{p:conformal}
Given an admissible $b$-separated $\alpha$-smooth $Q$-branching $\gr (u)$ with $\alpha <1/(2Q)$
there exist a constant $C_0 (Q, \alpha)>0$, a radius $r>0$ and functions $\Psii\colon \gira_{Q,r} \to \gr (u)$
and  $\lambda\colon \gira_{Q,r} \to \R_+$
such that
\begin{itemize}
\item[(i)] $\Psii$ is a homeomorphism of $\gira_{Q,r}$ with a neighborhood of $0$ in $\gr (u)$;
\item[(ii)] $\Psii\in C^{3,\alpha} (\gira_{Q,r}\setminus \{0\})$, with the estimates
\begin{align}
|D^l \big(\Psii(z,w) - (z,0)\big)| \leq & C_0 C_i |z|^{1+\alpha-l} \qquad \mbox{for $l=0,\dots,3$, $z\neq 0$}\, ,\label{e:conformal2}\\
[D^3 \Psii]_{\alpha, B_r (z,w)} \leq & C_0 C_i |z|^{-2} \qquad \mbox{for $z\neq 0$ and $r = |z|/2$}\, ;\label{e:conformal3}
\end{align}
\item[(iii)] $\Psii$ is a conformal map with conformal factor $\lambda$, namely, if we denote by $e$ the ambient euclidean metric
in $\R^{2+n}$ and by $e_Q$ the canonical euclidean metric of $\gira_{Q,r}$, 
\begin{equation}\label{e:conformal}
g:= \Psii^\sharp e = \lambda\, e_Q\, \qquad \mbox{on $\gira_{Q,r}\setminus \{0\}$.}
\end{equation}
\item[(iv)] The conformal factor $\lambda$ satisfies
\begin{align}
|D^l (\lambda -1) (z,w)| \leq &C_0 C_i |z|^{2\alpha-l} \qquad \mbox{for $l=0,1,\ldots, 2$}\label{e:conformal4}\\
[D^2 \lambda]_{\alpha, B_r (z,w)} \leq & C_0 C_i |z|^{\alpha-2}\qquad \mbox{for $z\neq 0$ and $r = |z|/2$}\, .\label{e:conformal5}
\end{align}
\end{itemize}
\end{proposition}

\begin{definition}
A map $\Psii$ as in Proposition \ref{p:conformal} will be called a {\em conformal parametrization} of an admissible $Q$-branching.  
\end{definition}

\subsection{The center manifold and the approximation} We are now ready to state the two ``halves'' of Theorem \ref{t:induttivo}. The first one is 
\cite[Theorem 2.6]{DSS3}, which we recall here for the reader's convenienve.

\begin{theorem}[Center Manifold Approximation]\label{t:cm}
Let $T$ be as in Assumptions~\ref{induttiva} and \ref{piu_fogli}.
Then there exist $\eta_0, \gamma_0, r_0,C>0$, $b>1$, an admissible $b$-separated $\gamma_0$-smooth $\bar Q$-branching $\cM$, a corresponding conformal parametrization $\Psii: \gira_{\bar Q,2} \to \cM$ 
and a $Q$-valued map $\NN: \gira_{\bar Q,2} \to \I{Q}(\R^{2+n})$ with the following properties:
\begin{itemize}
\item[(i)] $\bar Q Q = \Theta (T, 0)$ and 
\begin{align}
|D(\Psii (z,w) -(z,0))|\leq &\, C \bmo^{1/2} |z|^{\gamma_0}\\ 
|D^2 \Psii (z,w)|+ |z|^{-1} |D^3 \Psii (z,w)| \leq &\, C \bmo^{\sfrac{1}{2}} |z|^{\gamma_0-1}\, ;
\end{align}
in particular, if we denote by $A_\cM$ the second fundamental form of $\cM\setminus \{0\}$,
\[
|A_\cM (\Psii (z,w))| + |z|^{-1} |D_{\cM} A_\cM (\Psii (z,w))| \leq C \bmo^{\sfrac{1}{2}} |z|^{\gamma_0-1}\, .
\]
\item[(ii)] $\NN\,_i (z,w)$ is orthogonal to the tangent plane, at $\Psii (z,w)$, to $\cM$.
\item[(iii)] If we define $S:= T_{0,r_0}$, then
$\supp (S)\cap \bC_1 \setminus \{0\}$ is contained in a suitable horned neighborhood of the $\bar Q$-branching, where the orthogonal projection
$\P$ onto it is well-defined. Moreover, for every $r\in]0,1[$ we have
\begin{align}\label{e:Ndecay}
\|\NN\vert_{B_r}\|_0 + \sup_{p\in \supp (S) \cap \P^{-1} (\Psii(B_r))} |p- \P (p)|\leq C \bmo^{\sfrac{1}{4}} r^{1+\sfrac{\gamma_0}{2}}\, .
\end{align}
\item[(iv)] If we define
\[
\bD (r) := \int_{B_r} |D\NN|^2
\quad \text{and}\quad
\bH (r) := \int_{\de B_r} |\NN|^2\, ,
\]
\[
\bF(r) := \int_0^r\frac{\bH(t)}{t^{2-\,\gamma_0}}\,dt
\quad\text{and}\quad \bLambda(r):= \bD(r)+\bF(r)\, ,
\]
then the following estimates hold for every $r\in ]0,1[$:
\begin{align}
\Lip (\NN\vert_{B_r}) \leq & C \min\{\bLambda^{\eta_0}(r), \bmo^{\eta_0}r^{\eta_0}\} \label{e:Lip_N}\\
\bmo^{\eta_0} \,\int_{B_r}  |z|^{\gamma_0 -1} |\etaa \circ \NN(z,w)| 
\leq & C\, \bLambda^{\eta_0}(r)\,\bD(r)+C\,\bF(r)\, .\label{e:media_pesata}
\end{align}
\item[(v)] Finally, if we set
\[
\FF(z,w):= \sum_i\a{\Psii(z,w) + \NN\,_i(z,w)}\, ,
\]
then
\begin{align}
\|S-\bT_\FF\|\big(\P^{-1}(\Psii(B_r))\big) \leq  & C\,\bLambda^{\eta_0}(r)\,\bD(r)+ C\,\bF(r)\, . \label{e:diff masse} 
\end{align}
\end{itemize}
\end{theorem}

The second main step is the analysis of the asymptotic behaviour of $\NN$ around the origin, which is the main focus of this
paper. 

\begin{remark}\label{r:W12}
In order to state it, we agree to define $W^{1,2}$ functions on $\gira$ in the following fashion: removing the origin $0$ from
$\gira$ we have a $C^3_{loc}$ (flat) Riemannian manifold embedded in $\R^{4}$ and we can define $W^{1,2}$ maps on it following \cite{DS1}. Alternatively we can use the conformal parametrization $\bW: \R^2 = \C \to \gira_{\bar Q}$ given by $\bW (z)= (z^{\bar Q}, z)$ and agree that $u\in W^{1,2} (\gira_{\bar Q})$ if $u\circ \bW$ is in $W^{1,2} (\R^2)$. Since discrete sets have zero $2$-capacity, it is immediate to verify that these two definitions are equivalent.

In a similar fashion, we will ignore the origin when integrating by parts Lipschitz vector fields, treating $\gira_{\bar Q}$ as a $C^1$ Riemannian manifold. It is straightforward to show that our assumption is correct, for instance removing a disk of radius $\varepsilon$ centered at the origin, integrating by parts and then letting $\varepsilon\downarrow 0$.
\end{remark}

\begin{theorem}[Blowup Analysis]\label{t:bu}
Under the assumptions of Theorem~\ref{t:cm}, the following dichotomy holds:
\begin{itemize}
\item[(i)] either there exists $s>0$ such that $\NN\vert_{B_s} \equiv Q \a{0}$;
\item[(ii)] or there exist constants $I_0>1$, $a_0, \bar r, C>0$ and an
$I_0$-homogeneous nontrivial Dir-minimizing
function $g: \gira_{\bar{Q}} \to \I{Q}(\R^{2+n})$ such that
\begin{itemize}
\item $\etaa \circ g \equiv 0$, 
\item $g = \sum_i \a{(0, \bar{g}_i , 0)}$, where $\bar{g}_i (x)\in \R^{\bar{n}}$ and $(0, \bar g_i (x), 0)\in \R^2\times \R^{\bar n} \times \R^l$,
\item and the following estimates hold:
\begin{align}
\cG\big(\NN(z,w),g(z,w)\big) \leq &\; C |z|^{I_0+a_0}\quad\qquad\qquad\qquad \forall\;(z,w) \in \gira_{Q},
\;|z|<\bar r,\label{e:uniform convergence}\\
\int_{B_{r+\rho}\setminus B_{r-\rho}}|D\NN|^2 \leq &\; C\,r^{2I_0+a_0} + C\,r^{2I_0-1}\,\rho
\qquad \forall\;  4\,\rho\leq r < 1,\label{e:energia N}\\
\bH(r) \leq &\; C\,r\,\bD(r)\quad\qquad\qquad\qquad\; \forall\; r < 1.\label{e:Poincare0}
\end{align}
\end{itemize}
\end{itemize}
\end{theorem}

\begin{remark}\label{r:ultimo_step}
Note that, when $\bar{Q} = \Theta (T, 0)$, we necessarily have $Q=1$ and the second 
alternative is excluded. In particular we conclude that $T$ coincides with $\a{\cM}$ in a neighborhood of $0$ and
thus it is a regular submanifold in a {\em punctured neighborhood} of $0$.
\end{remark}

\begin{remark}\label{r:bound_energia}
By a simple dyadic argument it follows from \eqref{e:energia N} and \eqref{e:Poincare0} that
\begin{equation}\label{e:consequences}
\int_{B_r}|D\NN|^2 \leq C\,r^{2I_0}
\quad\text{and}\quad
\bF(r) \leq C\,r^{2I_0+\gamma_0}
\quad  \forall\;r < 1.
\end{equation}
\end{remark}

Below we show how to conclude Theorem~\ref{t:induttivo} from Theorem \ref{t:bu}.
The remaining part of the paper is dedicated to the proof of the latter, which will be split in six sections each corresponding to one of the following steps. 
\begin{itemize}
\item[(i)] In Section \ref{s:AM} we will deduce an almost minimizing property for the map $\NN$ in terms of its Dirichlet energy.
\item[(ii)] In Section \ref{s:harmonic} we will apply the almost minimizing property and compare the Dirichlet energy of $\NN$ with that of a suitable harmonic extension of its boundary value on any given ball.
\item[(iii)] In Section \ref{s:poincare} we use the comparison above and a first variation argument to derive a suitable Poincar\'e-type inequality for $\NN$.
\item[(iv)] In Section \ref{s:refined} we compute again the first variations of the Dirichlet energy of $\NN$ and use the Poincar\'e inequality to bound efficiently several error terms.
\item[(v)] Using the latter bounds, in Section \ref{s:decay} we will prove an almost monotonicity property for the frequency function and the existence and boundedness of its limit, which is indeed the number $I_0$ of Theorem \ref{t:bu}. The almost minimality of $\NN$ will then allow us to conclude an exponential rate of decay to this limit.
\item[(vi)] From the decay of the previous step we will capture in Section \ref{s:bu} the asymptotic behaviour of $\NN$ and show the existence of the map $g$ of Theorem \ref{t:bu}.  
 \end{itemize}
The overall strategy follows the ideas and some of the computations in \cite{Chang}. However several adjustments are needed
to carry on the proof in the cases (b) and (c) of Definition \ref{d:semicalibrated}. In particular in Section \ref{s:decay} we need to introduce a suitable modification of the usual frequency function to handle case (b).

\subsection{Proof of the inductive step} In the next sections we will prove Theorem \ref{t:bu}.
We start observing that if case (a) of Theorem~\ref{t:induttivo}
does not hold, then we are necessarily in case (ii) of
Theorem~\ref{t:bu}. Therefore we only need to prove that
Theorem~\ref{t:bu}(ii) implies Theorem~\ref{t:induttivo}(b).

We divide the proof in different steps.

\medskip

\noindent{\bf Step 1.} For  a reason which will become clear later, it is convenient to slightly modify the map $g$ to a multivalued map
$n (z,w) = \sum_i \a{n_i (z,w)}$ in such a way that $n_i (z,w)$ is orthogonal to $\cM$ at $\Psii (z,w)$. To achieve this it suffices to project
$g_i (z,w) = (0, \bar{g}_i (z,w),0)$ on the normal bundle. Observe that, by the estimates on $|A_\cM|$ and $\Psii$, we have 
\begin{align}
|g_i (z,w) - n_i (z,w)|\leq & C C_i |z|^{\gamma_0} |g_i (z,w)|\, ,\label{e:proietta_1}\\
|Dn| (z,w) \leq & |Dg| (z,w) + C C_i |z|^{\gamma_0 -1} |g| (z,w)\,\label{e:proietta_2}.
\end{align}
We introduce the function
$H: \gira_{\bar Q} \to \I{Q}(\R^{2+n})$ given by
\[
H(z,w) = \sum_{i=1}^{Q}\a{H_i(z,w)} := \sum_{i=1}^{Q}
\a{\Psii(z,w) + n_i (z,w)}\, . 
\]

Note that, since $g$ is $I_0$-homogeneous and $\D$-minimizer, by \cite[Proposition 5.1]{DS1} there is a constant $C$ such that
\begin{equation}\label{e:separazione omogenea 1}
|g_i(z,w) - g_j(z,w)| \geq 2 C\,|z|^{I_0} \quad\text{whenever }g_i(z,w) \neq g_j(z,w)\, .
\end{equation}
In fact \cite[Proposition 5.1]{DS1} is stated for maps with domain $\bC = \R^2$. However, if we define the map $\bW: \C \to \gira_{\bar Q, \infty}$ as $\bW (z) = (z^{\bar Q}, z)$, by the conformality of $\bW$ it is easy to check that $g\circ \bW$ is $\D$-minimizer and $I_0 Q$ homogeneous. 

By \eqref{e:proietta_1} and \eqref{e:separazione omogenea 1}, provided $z$ is small enough we have
\begin{equation}\label{e:separazione omogenea}
|H_i(z,w) - H_j(z,w)| \geq C\,|z|^{I_0} \quad\text{whenever }H_i(z,w) \neq H_j(z,w).
\end{equation}
Let $\bar a\in ]0, a_0[$ be a constant to be fixed momentarily and 
$\zeta := I_0 + \sfrac{\bar a}{2}>1$. Set
\[
\bV_{H, \zeta} := \big\{ H_i(z,w) + p \in \R^{2+n} :
|p|< |z|^\zeta,\; i=1,\ldots, Q\big\}.
\]
We claim that there exists $s>0$ such that $\supp(T) \cap \bB_s\setminus \{0\} \subset \bV_{H,\zeta}$.

In order to prove this claim, we distinguish two cases.
First we consider any point $p \in \supp(T) \cap \supp(\bT_\FF)\setminus \{0\}$. In this case $p = \Psii(z,w) + \NN\, _i(z,w)$ for some 
$(z,w) \in \gira_{\bar Q}\setminus \{0\}$ and 
for some $i=1,\ldots, Q$.
Without loss of generality, by \eqref{e:uniform convergence} we can assume
$|\NN\,_i(z,w) - g_i(z,w)| \leq   C |z|^{I_0+\bar a}$, i.e.
\begin{align}\label{e:uniform graph}
|p-H_i(z,w)| & = |\NN\,_i (z,w) - n_i (z,w)|\leq |\NN\,_i (z,w)-g_i (z,w)|+ |g_i (z,w)- n_i (z,w)|\nonumber\\
&\leq
C |z|^{I_0+\bar a} + C|z|^{\gamma_0+I_0},
\end{align}
which in particular implies
$\supp(T) \cap \supp(\bT_\FF) \cap \bB_s \subset \bV_{H,\zeta}$
if $s$ is sufficiently small and we impose $\frac{\bar a}{2}< \gamma_0$.

For the second case we consider a point $p \in \supp(T) \setminus \supp(\bT_\FF)$
and assume by contradiction that $p \not\in \bV_{H,\zeta}$.
In particular, in view of \eqref{e:uniform graph} we have that 
\begin{equation}\label{e:grafico_escluso}
B:=\bB_{\frac{|z|^{\zeta}}{2}}(p) \cap \supp(\bT_\FF) = \emptyset
\end{equation}
if $|z|$ is sufficiently small.
By the monotonicity formula we know that $\|T\|(B) \geq C\, |z|^{2\zeta}$;
nevertheless since $B \subset \P^{-1}(B_{2|z|}\setminus B_{\sfrac{|z|}{2}})$, \eqref{e:grafico_escluso}
implies $\|T\| (B) \leq \|T - \bT_{\FF}\| (B)$ and 
from \eqref{e:diff masse} and \eqref{e:consequences} we conclude
$\|T\|(B) \leq C\, |z|^{2I_0+2\,\kappa}$ with $\kappa = \min\{2\,\eta_0\,I_0, \gamma_0\}$.
This gives a contradiction if $\bar a < 2\kappa$.

\medskip

\noindent{\bf Step 2.}
From the previous step we can infer that $g$
is a constant multiple of an irreducible
function, namely there exists $Q'>0$ such
that $\card(g(z,w)) = Q'$ for every $(z,w) \neq (0,0)$ and 
there exists a continuous map $h: \gira_{\bar Q Q'} \to \R^{2+n}$ such that
\begin{equation}\label{e:irreducibility}
g(z,w) = \frac{Q}{Q'} \sum_{\tilde z = z, \; \tilde w^{Q'} = w} \a{h(\tilde z, \tilde w)}.
\end{equation}
If this is not the case, by  
\cite[Proposition~5.1]{DS1} we can decompose $g$
in the superposition of irreducible functions, i.e.
there exists a unique decomposition $g=\sum_{j=1}^J k_j g_j$
where $g_j:\gira_{Q} \to \I{q_j}(\R^{n})$ are Dir-minimizing
$I_0$-homogeneous functions, for some choice of positive integers $J,k_j, q_j$ such that $\sum_{j=1}^Jk_jq_j=Q$.

Denoting by $H^j$ the corresponding maps
\[
H^j(z,w) := \sum_{l=1}^{q_j}\a{\Psii(z,w) + (n^j)_l(z,w)}
\]
and by $\bV_{H^j,\zeta}$ the corresponding horned neighborhoods
\[
\bV_{H^j, \zeta} := \big\{ (H^j)_l(z,w) + p \in \R^{2+n} :
|p|< |z|^\zeta,\; l=1,\ldots, q_j\big\},
\]
it follows from \eqref{e:separazione omogenea} that the closures of the
$V_{\zeta, H_i}$ intersect only at the origin.
Setting $T_i := T\res V_{\zeta, H_i}$, we infer that $T= \sum_i T_i$
with $\supp(T_i) \cap \supp(T_j) = \{0\}$, against the irreducibility of $T$.
Note that, since $\etaa\circ g=0$ it also follows that $Q'>1$.

Having established \eqref{e:irreducibility}, let us define $\Thetaa:\gira_{\bar QQ'} \to \R^{2+n}$ as
\[
\Thetaa(\tilde z,\tilde w) := \Psii (\tilde z, \tilde w^{Q'}) + h^n (\tilde z, \tilde w)
\quad \forall\;(\tilde z, \tilde w) \in \gira_{\bar QQ'}\, ,
\]
where $h^n (\tilde{z}, \tilde w)$ is the projection of $h (\tilde{z}, \tilde{w})$ on the space normal to
$\cM$ at the point $\Psii (\tilde{z}, \tilde w^{Q'})$. 
It follows that
$\im(H) = \im(\Thetaa)$ is an admissible $\bar Q Q'$-branching (the H\"older regularity for the graphical parametrization follows from the fact that $I_0 >1$). 
Moreover, from the homogeneity of $g$ we easily infer that $\im (\Thetaa)$
is $I_0$-separated (for a suitable constant $c_s$).
Note that for $\zeta':=I_0+\bar a/4$ and $s$ sufficiently small
$\bV_{H,\zeta}\cap\bB_s\subset \bV_{\Thetaa,\zeta'}\cap \bB_s$.

\medskip

\noindent{\bf Step 3.}
Finally we prove the condition (Dec) of Assumption~\ref{induttiva}.
Let $(z,w) \in \gira_{\bar Q}$ with $0<|z|<\sqrt{2}$,
let $V$ be the connected component of $\bV_{\Thetaa, \zeta'}\cap \{(x,y): |x-z| < d\}$
with $d:=\sfrac{|z|}{4}$ containing $\Thetaa(z,w)$,
and $p\in \supp(T) \cap V$ with coordinates $p=(z,y)$.
Denote by $\pi$ the oriented two-vector for $\im(\Thetaa)$ at $\Thetaa(z,w)$,
and consider $\rho \in [\frac12 d^{\sfrac{(I_0+1)}{2}},d]$.

Since $\bB_{\rho}(p)\cap \supp (T) \subset \P^{-1}(\Psii (B_{|z|+2\rho}\setminus B_{|z|-2\rho}))$,
we start estimating as follows
\begin{align}
\int_{\bB_\rho(p)} |\vec T - \vec\pi|^2 \, d\|T\res V\| &\leq 
\int_{\bB_\rho(p)\cap V} |\vec \bT_\FF - \vec\pi|^2 \, d\|\bT_\FF\| + \|T-\bT_\FF\|(\p^{-1}(B_{|x_0|+2\rho}))\notag\\
& \stackrel{\eqref{e:diff masse}}{\leq} \int_{\bB_\rho(p)\cap V} |\vec \bT_\FF - \vec\pi|^2 \, d\|\bT_\FF\| + C\,|z|^{2I_0 
+ 
2 \kappa}.
\end{align}
Next, note that for $|z|$ small enough $\P(\bB_\rho(p)\cap \bV_{\Thetaa, \zeta'}) \subset \Psii (B_{2\rho}(z,w))$.

We can consider the set of indices $A \subset \{1,\ldots,Q\}$ such
that $\FF_i(z,w) \in V$ for $i \in A$ and estimate as follows
\begin{align}
\int_{\bB_\rho(p)\cap V} |\vec \bT_\FF - \vec\pi|^2 \, d\|\bT_\FF\| & \leq
C\sum_{i\in A}\int_{B_{2\rho}(z,w)} |\vec \bT_{\FF_i (\zeta, \omega)} - \vec \bT_{\Thetaa (\zeta, \omega)}|^2\, d\zeta
+C\,\rho^2\,\Lip(D\Thetaa\vert_{B_{2\rho}(z,w)})^2\notag\\
& \leq C\sum_{i\in A}\int_{B_{2\rho}(z,w)} |\vec \bT_{\FF_i (\zeta, \omega)} - \vec \bT_{\Psii (\zeta, \omega)}|^2\, d\zeta\notag\\
&\quad+ C\,\int_{B_{2\rho}(z,w)} |\vec \bT_{\Psii (\zeta, \omega)} - \vec \bT_{\Thetaa (\zeta, \omega)}|^2\, d\zeta + C\,\rho^4\,|z|^{2\theta-2}, 
\end{align}
where $\theta:=\min\{\gamma_0, I_0-1\}$ and 
we used that $|D^2\Thetaa|(z,w) \leq C\,|z|^{\theta-1}$.

We can finally use the computation of 
the excess in curvilinear coordinates in \cite[Proposition~3.4]{DS2} to get
\begin{align}
\sum_i\int_{B_{2\rho}(z,w)} |\vec \bT_{\FF_i (\zeta, \omega)} - \vec \bT_{\Psii (\zeta, \omega)}|^2 &
\leq C\int_{B_{2\rho}(z,w)} \left(|D\NN|^2 + |\zeta|^{2\gamma_0-2} |\NN|^2\right)
\notag\\
&\stackrel{\eqref{e:consequences}}{\leq} C\int_{B_{|z|+2\rho}\setminus B_{|z|-2\rho}} |D\NN|^2 + C\,|z|^{2I_0+2\gamma_0}\\
&\stackrel{\eqref{e:energia N}}{\leq} C\,|z|^{2I_0+a_0}+C\,|z|^{2I_0-1}\,\rho\,,
\end{align}
and similarly
\begin{align}
\int_{B_{2\rho}(z,w)} |\vec \bT_{\Thetaa (\zeta, \omega)} - \vec \bT_{\Psii (\zeta, \omega)}|^2&
\leq C\int_{B_{2\rho}(z,w)} \left(|Dn|^2 + |\zeta|^{2\gamma_0-2} |n|^2\right)\notag\\
& \leq C\int_{B_{2\rho}(z,w)} \left(|Dg|^2 + |\zeta|^{2\gamma_0-2} |g|^2\right)\notag\\
&\leq C\,|z|^{2I_0-2}\rho^2+C\,|z|^{2I_0+2\gamma_0}\, 
\end{align}
(observe that, in order to apply \cite[Proposition~3.4]{DS2} we need that $n$ takes value into the
normal bundle). 

Collecting all the estimates together, we have that there exists a suitable
constant $\varpi$ such that
\begin{equation}\label{e:dec finale}
\int_{\bB_\rho(p)} |\vec T - \vec\pi|^2 \, d\|T\res V\| \leq
C\,|z|^{2I_0 + 2 \varpi} + C\,\rho\,|z|^{2I_0-1} + C\,\rho^4\,|z|^{2\varpi-2}
\leq |z|^{\gamma-2}\rho^4,
\end{equation}
where the last inequality is easily verified
for $\gamma>0$ and $|z|$ small enough. This shows (Dec) in Assumption \ref{induttiva} and completes the proof.

\section{Dirichlet almost minimizing property}\label{s:AM}

The normal approximation $\NN$ inherits from $T$ an almost minimizing property for the Dirichlet energy, where the errors involved are in
fact expressed in terms of some specific norms of $\NN$
itself and of its competitors.

For techincal reasons we introduce the maps $F  := \sum_{i=1}^{Q}\a{p + N_i(p)}$, where $N := \NN \circ \Psii^{-1}$. In order to state the almost minimizing property of $\NN$ we introduce an appropriate notion of competitor.

\begin{definition}\label{d:competitors}
A Lipschitz map $\LL\colon B_r\to \I{Q}(\R^{n+2})$ is called
a competitor for $\NN$ in the ball $B_r$ if 
\begin{itemize}
\item[(a)] $\LL \vert_{\partial B_r}= \NN\vert_{\de B_r}$;
\item[(b)] $\supp(\GG(z, w)) \subset \Sigma$ for all $(z,w)\in B_r$, where $\GG(z,w) := \sum_{j=1}^{Q}\a{\Psii(z,w) +\LL_j(z,w)}$.
\end{itemize}
\end{definition}

We are now ready to state the almost minimizing property for $\NN$. We use the notation $\p_{T_p \Sigma}$ for the orthogonal projection on the tangent space to $\Sigma$ at $p$. We recall that, given our choice of coordinates, $\p_{T_0 \Sigma}$ is the projection on $\R^{2+\bar{n}}\times \{0\}$. Since this projection will be used several times, we will denote it by $\p_0$. By the $C^{3,\eps}$ regularity of $\Sigma$, there exists a map $\Psi_0\in C^{3,\eps}(\R^{2+\bar{n}},\R^l)$ such that
\[
\Psi_0 (0)= 0\, ,\; D\Psi_0 (0)=0\quad \textup{and}\quad \gr (\Psi_0) = \Sigma\,.
\]
Next, for each function $\LL$ satisfying Condition (b) in Definition \ref{d:competitors} we consider the map $\bar\LL := \p_0 \circ \LL$, which is a multivalued $\bar\LL: \gira \to \Iq (\R^{2+\bar{n}})$. We observe that it is possible to determine $\LL$ from $\bar\LL$. In particular, fix coordinates
$(\xi, \eta)\in \R^{2+\bar{n}}\times \R^{l}$ and let $\LL = \sum \a{\LL_i}$, $\bar\LL = \sum \a{\bar \LL_i}$, where $\bar \LL_i = \p_0 \circ \LL_i$. Then the formula relating $\LL_i$ and $\bar \LL_i$ is 
\begin{equation}\label{e:relazione}
\LL_i (z,w) = \big(\bar \LL_i (z,w), \Psi_0 \big(\p_0 (\Psii (z,w)) + \bar\LL_i (z,w)\big) - \Psi_0 (\p_0 (\Psii (z,w))\big)\, .
\end{equation}

\begin{proposition}\label{p:amin}
There exists a constant $C_{\ref{p:amin}}>0$ such that the following holds.
If $r\in (0,1)$ and $\LL\colon B_r\to \I{Q}(\R^{2+n})$ is a
Lipschitz competitor for $\NN$ with $\|\LL\|_\infty \leq r$ and $\Lip (\LL)\leq C_{\ref{p:amin}}^{-1}$, then
\begin{align}\label{e:amin}
\int_{B_r}|D\NN|^2 &\leq
(1+C_{\ref{p:amin}}\,r)\,\int_{B_r} |D\bar \LL|^2 + 
C_{\ref{p:amin}}\,\textup{Err}_1(\NN,B_r) + C_{\ref{p:amin}}\,\textup{Err}_2(\LL,B_r) + C_{\ref{p:amin}}\, r^2 \bD' (r)\, ,
\end{align}
where $\bar \LL:=\p_0 \circ \LL$ and the the error terms
$\textup{Err}_1(\NN,B_r)$, 
$\textup{Err}_2(\LL,B_r)$ are given by the following expressions:
\begin{align}\label{e:error1}
\textup{Err}_1(\NN,B_r) & =
\bLambda^{\eta_0}(r)\,\bD(r)+\bF(r)+\bH(r)
+\bmo^{\sfrac12}\, r^{1+\gamma_0}\int_{\de B_r} |\etaa\circ \NN|\, 
\end{align}
and
\begin{align}\label{e:error2}
\textup{Err}_2(\LL,B_r) & =
\bmo^{\sfrac12}\,\int_{B_r} |z|^{\gamma_0 -1} |\etaa\circ \LL|\,  \,.
\end{align}
\end{proposition}

For the proof of Proposition~\ref{p:amin} we
consider separately the three cases: 
\begin{itemize}
\item[(a)] $T$ is mass minimizing;
\item[(b)] $T$ is semicalibrated;
\item[(c)] $T$ is the cross-section of a mass minimizing three-dimensional cone.
\end{itemize}
For notational convenience we set $L := \LL \circ \Psii^{-1}$, $G :=  \GG \circ \Psii^{-1}$. 

Observe also that, by Lemma \ref{l:poinc} and \ref{l:proiezione}, it is enough to prove that
\begin{equation}\label{e:amin2}
\int_{B_r}|D\NN|^2 \leq
(1+C_{\ref{p:amin}}\,r)\,\int_{B_r} |D \LL|^2 + 
C\,\textup{Err}_1(\NN,B_r) + C \,\textup{Err}_2(\LL,B_r)+\frac{C}{r}\int_{B_r}|\LL|^2 + C r^2\bD'(r)\,.
\end{equation}
Indeed Lemma \ref{l:proiezione} implies that
\begin{align*}
\int_{B_r} |D \LL|^2 \leq & (1+Cr) \int_{B_r} |D \bar\LL|^2 + C r \int_{\partial B_r} |\bar \LL|^2 \leq (1+Cr) \int_{B_r} |D\bar\LL|^2 + Cr \int _{\partial B_r}|\LL|^2\\
= &  (1+Cr) \int_{B_r} |D\bar\LL|^2 + Cr \int_{\partial B_r} |\NN|^2 \leq (1+Cr)\int_{B_r} |D\bar \LL|^2 + C\, \textup{Err}_1 (\NN, B_r)\, ,
\end{align*}
whereas Lemma \ref{l:poinc} implies
\begin{align*}
\frac{1}{r} \int_{B_r} |\LL|^2 \leq & Cr \int_{B_r} |D\LL|^2 + C \int_{\partial B_r} |\LL|^2 \leq C r \int_{B_r} |D\bar\LL|^2 + 
C \, \textup{Err}_1 (\NN,
 B_r)\, .
\end{align*}

\subsection{Proof of Proposition~\ref{p:amin} case (a): $T$ mass minimizing.} We fix $\LL$, $\bar \LL$, $L$, $\GG$, $\bar \GG$ and $G$ as above. Let us set
\begin{equation}\label{e:competitore}
Z:= T - \bT_{\FF|_{B_r}} + \bT_{\GG}\, .
\end{equation}
Since $\FF|_{\partial B_r} = \GG|_{\partial B_r}$, from \cite{DS2} it follows that $\partial (\bT_{\GG} - \bT_{\FF|_{B_r}}) =0$. Moreover $\supp (Z) \subset \Sigma$ and therefore we must have $\mass (T) \leq \mass (Z)$. Taking into account \eqref{e:diff masse}, we conclude that
\begin{align}\label{e:prima}
\mass (\bT_{\FF|_{B_r}}) 
&\leq \mass (T\res \p^{-1} (\Psii (B_r))) + \|T-\bT_{\FF|_{B_r}}\|( \p^{-1} (\Psii (B_r)))\notag\\ 
&\leq  \mass (\bT_{\GG}) +2 \,\|T-\bT_{\FF|_{B_r}}\|( \p^{-1} (\Psii (B_r))) \notag\\
&\leq  \mass (\bT_{\GG})+C\, \textup{Err}_1 (\NN, B_r) \, .
\end{align}
Observe now that $\bT_{\FF|_{B_r}} = \bT_{F|_{\Psii (B_r)}}$ and we can use the Taylor expansion in \cite{DS2} to compute:
\begin{align}\label{e:mass1}
\mass(\bT_{\FF|_{B_r}}) &\geq Q \cH^2(\Psii (B_r))+\frac{1}{2}\int_{\Psii (B_r)}|D N|^2 -
Q\int_{\Psii (B_r)} \langle \etaa\circ N,H_\cM\rangle\notag\\
&\quad -C\int_{\Psii (B_r)} \Big(|A_{\cM}|^2 | N|^2+|DN|^4\Big)\, ,
\end{align}
where $H_{\cM}$ denotes the mean curvature vector of $\cM$. Note that in order to apply the Taylor expansion in \cite{DS2} we need the manifold $\cM$ to be $C^2$, with an apriori bound on the $C^2$ norm. However, if we take $T_F \res \bB_r \setminus \bB_{r/2}$ and rescale by a factor $1/r$, the corresponding rescaled current, map and manifold fall under the assumptions of the Taylor expansion in \cite{DS2}. We can then scale back to find the corresponding inequalities for $T\res \bB_r \setminus \bB_{r/2}$ and sum over dyadic annuli to conclude \eqref{e:mass1}.

Using the conformality of $\Psii$ we conclude
\[
\int_{\Psii (B_r)}|D N|^2 = \int_{B_r} |D\NN|^2\, ,
\]
As for the other terms, we recall
\begin{align}
&\int_{\Psii (B_r)} |\langle \etaa\circ N, H_{\cM}\rangle| \leq  C \bmo^{\sfrac{1}{2}} \int_{B_r} |\etaa\circ \NN|
\stackrel{\eqref{e:media_pesata}}{\leq} C \textup{Err}_1 (\NN, B_r)\, , \\
&\int_{\Psii (B_r)} |DN|^4 \leq  C \Lip (\NN|_{B_r})^2 \int_{B_r} |D\NN|^2 \stackrel{\eqref{e:Lip_N}}{\leq} C \textup{Err}_1 (\NN, B_r)\, ,\\
&\int_{\Psii (B_r)} |A_\cM|^2 |N|^2 \leq  C \bmo \int_{B_r} |z|^{2\gamma_0-2} |\NN|^2 =
C \bmo \int_0^r \frac{\bH (s)}{s^{2-2\gamma_0}}\, ds \leq C \textup{Err}_1 (\NN, B_r)\, .
\end{align}
Combining the latter estimates with \eqref{e:competitore} and \eqref{e:prima} we achieve
\begin{equation}\label{e:competitore2}
\frac{1}{2} \int_{B_r} |D\NN|^2 \leq C \textup{Err}_1 (\NN, B_r) + \mass (\bT_G) - Q \cH^2 (\Psii (B_r (x))\, . 
\end{equation}
Next, fix an orthonormal frame $\xi_1, \xi_2$ on $B_r$ and, using the area formula from \cite{DS2}, compute
\begin{align*}
\mass (\bT_G)= & \int_{\Psii (B_r)} \sum_i |(\xi_1 + DL_i \cdot \xi_1) \wedge (\xi_2 + DL_i \cdot \xi_2)|\\
 \leq &\frac{1}{2} \int_{\Psii (B_r)} \sum_i \left(|\xi_1 + DL_i \cdot \xi_1)|^2 + |\xi_2 + DL_i \cdot \xi_2|^2\right)\\
= & Q \cH^2 (\Psii (B_r)) + \frac{1}{2} \int_{\Psii (B_r)} |DL|^2\\
&\quad + Q \int_{\Psii (B_r)} \left( \langle D  (\etaa\circ L) \cdot \xi_1, \xi_1\rangle +
\langle D (\etaa\circ L) \cdot \xi_2, \xi_2 \rangle \right)\, .
\end{align*}
By conformality the second summand in the last inequality equals $\frac{1}{2}\int_{B_r} |D\LL|^2$. We integrate by parts the third summand. Recall that $\etaa \circ L=\etaa \circ N$ on $\Psii (\partial B_r) = \partial (\Psii (B_r))$: since $\etaa\circ N$ is orthogonal to $\xi_i$ the boundary term vanishes. Moreover, since the origin is a singularity, we must in fact integrate by parts in $B_r\setminus B_{\varepsilon}$ and then let $\varepsilon \to 0$. 
A specific choice of $\xi_i$ is $\xi_i = \lambda^{-\sfrac{1}{2}} D \Psii \cdot e_i$, where $e_1, e_2$ is the parallel frame on $\gira_Q$ naturally induced by the standard flat coordinates. It then turns out that
\[
|D_{\xi_1} \xi_1 + D_{\xi_2} \xi_2|(\Psii (z,w)) \leq C \bmo^{\sfrac{1}{2}} |z|^{\gamma_0-1}\, .
 \]
In particular $|D_{\xi_1} \xi_1 + D_{\xi_2} \xi_2|$ is integrable on $B_r$ and we can therefore conclude
\begin{align}
\mass (\bT_G) - Q \cH^2 (\Psii (B_r))  \leq  & \frac{1}{2} \int_{\Psii (B_r)} |DL|^2 + Q \int_{\Psii (B_r)} \langle \etaa\circ L, D_{\xi_1} \xi_1 + D_{\xi_2} \xi_2 \rangle\nonumber\\
\leq &  \frac{1}{2} \int_{B_r} |D\LL|^2 + C \textup{Err}_2 (\LL, B_r)\, .\label{e:competitore3}
\end{align}
Combining \eqref{e:competitore2} and \eqref{e:competitore3} we conclude \eqref{e:amin2}.

\subsection{Proof of Proposition~\ref{p:amin} case (b): $T$ semicalibrated} We proceed as in the previous step and define the current $Z$ as in \eqref{e:competitore}. If $S$ is any current such that
\[
\partial S = T-Z = \bT_{\FF|_{B_r}} - \bT_\GG = \bT_{F|_{\Psii (B_r)}} - \bT_G\, ,
\]
then the semicalibrated condition gives
\[
\mass (T) \leq \mass (Z) + S (d\omega)\, ,
\]
where $\omega$ is the calibrating form. In particular, in order to conclude the proof it suffices to find an $S$ such that
\begin{equation}\label{e:errore_semical}
|S (d\omega)| \leq C\,\textup{Err}_1(\NN,B_r) + C\,\textup{Err}_2(\LL,B_r)+\frac{C}{r}\int_{B_r}|\LL|^2\, :
\end{equation}
combining the latter inequality with the estimates of the previous subsection we reach the desired inequality.

We first define $H_i: [0,1] \times \Psii (B_r) \to \I{Q}(\R^{2+n})$ for $i=1,2$ by
\begin{gather}
[0,1]\times \Psii (B_r) \ni (t,p) \mapsto H_1(t,p):= \sum_{i=1}^{Q}\a{p+t\, N_i(p)}\in \I{Q}(\R^{2+n})\notag\\
[0,1]\times \Psii (B_r) \ni (t,p) \mapsto H_2(t,p):= \sum_{i=1}^{Q}\a{p+(1-t) \, L_i(p)}\in \I{Q}(\R^{2+n})\notag\,.
\end{gather}
We choose $S:=S_1+S_2$, where $S_i:=\bT_{H_i}$ for $i=1,2$.
Thanks to the homotopy formula in \cite{DS2}, we get
\begin{align*}
\de S_1 & = \bT_{F|_{\Psii (B_r)}}-Q \a{\cM} -\bT_{H_1\vert_{[0,1]\times \Psii (\partial B_r)} },\\
\de S_2 & =  Q\a{\cM} - \bT_{G|_{\Psii (B_r)}} + \bT_{H_2\vert_{[0,1]\times \Psii (\partial B_r)}} .
\end{align*}
On the other hand since $N=L$ on $\Psii (\partial B_r)$, we conclude $\partial S = \partial (S_1+S_2) = T-Z$. 

We next estimate $|S_1 (d\omega)|$ and $|S_2 (d\omega)|$. Since the estimates are analogous, we give the details only
for the first.
We start from the formula
\[
S_1(d\omega) =\int_{\Psii (B_r)}\int_0^1  \sum_{i=1}^{Q}
\big\langle \vec\zeta_i(t,p), d\omega((H_1)_i(t,p)) \big\rangle\, d\cH^2(p)\,dt,
\]
with
\begin{align*}
\vec{\zeta}_i(t,p) & = 
\big(\xi_1+t\,\nabla_{\xi_1}N_i(p)\big)\wedge\big(\xi_2+t\,\nabla_{\xi_2}  N_i(p)\big)
\wedge  N_i(p)\\
&=:\xi_1\wedge \xi_2 \wedge N_i(p) + \vec{E}_i(t, p)\,,
\end{align*}
and
\begin{align}
|\vec E_i(t,p)| \leq C\,(|D N|(p) + |DN|^2 (p))\,|N|(p).
\end{align}
Next we note that
\begin{align}
d\omega((H_1)_i(t,p)) = d\omega(p) + I(t,p),
\end{align}
where $I(t,p)$ can be estimated by 
\begin{align}
|I(t,p)| & = |d\omega((H_1)_i(t,p)) - d\omega(p)|
\leq C\, \|D^2\omega\|_{L^\infty}\, | N|(p).
\end{align}
Therefore, we have
\begin{align}
&\Big\vert\sum_{i=1}^{Q}\big\langle \vec\zeta_i(t,p), d\omega((H_1)_i(t,p)) \big\rangle
\Big\vert
\leq \sum_{i=1}^{Q}\langle\xi_1\wedge \xi_2 \wedge N_i(p), d\omega(p) \rangle
+ \|d\omega\|_{L^\infty} \sum_{i=1}^{Q} |\vec E_i(t,p)|\notag\\
&\quad + C\sum_{i=1}^{Q}\Big((|N_i| +|\vec E_i|)\, |I|\Big)(t,p)\notag\\
& \leq C \bmo^{\sfrac{1}{2}} \,|\etaa\circ N| + C |N|^2(p) + C |D N|(p)\,| N|(p) + Cr |DN|^2 (p)\, ,\notag
\end{align}
where we have only used the bound $|N| (p) \leq C r$ on $\Psii (B_r)$. 
Arguing similarly for $S_2$ (observe that we have the bound $|L| (p) \leq C r$) 
and estimating $|N||DN| + |L| |DL| \leq r^{-1} (|N|^2 + |L|^2) + Cr (|DN|^2 + |DL|^2)$, we achieve
\begin{align*}
|S_1(d\omega)| + |S_2(d\omega)| &\leq 
C\,\bmo^{\sfrac{1}{2}} \int_{\Psii (B_r)}\big(|\etaa\circ N| +|\etaa\circ L|\big)+
C\, r^{-1}\int_{\Psii (B_r)}\big(| N|^2 +| L|^2\big)\\
&\quad
+ C r\int_{\Psii (B_r)}\big(|D N|^2+|D L|^2\big),
\end{align*}
and we conclude as above by a change of variable and Theorem \ref{t:cm}. 

\subsection{Proof of Proposition~\ref{p:amin} in case (c): $T$ is the cross-section
of a three dimensional area minimizing cone} Recall that in this case $\supp (T) \subset \partial \bB_R (p_0)$, 
where $p_0 = (0, \ldots , 0, R) = R e_{n+2}$ and $R^{-1} \leq \bmo^{\sfrac{1}{2}}$. For the computations of this subsection it is indeed convenient
to change coordinates so that $p_0$ is in fact the origin, whereas $\Psi (0,0)$ is the point $(0, \ldots , 0, -R)$. In these new coordinates we then have $\cM, \supp (T), \im (\FF) \subset \partial \bB_R (0)$.
These coordinates will however be used only in here, whereas in the next sections we will return to the usual ones.

We introduce the following notation:
$\cC (r)$ is the cone over $\Psii (B_r)$ with vertex $0$, i.e.
\[
\cC (r):=\big\{\rho p \in \R^{n+2}\,:\, \rho\in [0,1], \,p\in \Psii (B_r)\big\},
\]
with the orientation compatible with that of $0 \cone \a{\cM}$.
We extend $F$ to $\tilde{F}: \cC (r)\to \Iq (\R^{n+2})$ by setting 
$\tilde F (\rho p) := \rho\, F(p)$ for every $p\in \Psii (B_r)$.

In order to estimate the Dirichlet energy of $N$ in terms of that of $L$,
we construct a suitable function $K: \cC (r) \to \I{Q}(\R^{n+2})$
(depending on $L$ and $N$)
such that $K \vert_{\de \cC (r)} = \tilde{F} \vert_{\de \cC (r)}$:
we can then test the minimizing property of 
$0 \cone T$ comparing its mass with that of 
the current 
\[
Z:= 0\cone T - \bT_{\tilde{F}} + \bT_K = 0 \cone (T- \bT_{F|_{\Psii (B_r)}}) + \bT_K
\]
which is easily recognized to satisfy $\de Z = \de (0\cone T)$. In particular, using the minimality of $0\cone T$, we conclude
\begin{equation}\label{e:minimalita'_cono}
R^{-1} \mass (0\cone \bT_{F|_{\Psii (B_r)}}) \leq R^{-1} \mass (\bT_K) + C \textup{Err}_1 (\NN, B_r)\, .
\end{equation}

We consider the space of parameters $[0,1] \times B_r$ and
recall that the points in $\gira_{Q}$ are identified by two complex
coordinates $(z,w) \in \C \times \C$.
For the definition of $K$ we need to introduce the following sets
\begin{align}
A_1 := & \left\{(\rho, z, w) \in [0,1] \times B_r : 1-r\leq \rho \leq 1, 
\; |z| \leq \frac{\rho + 2\,r -1}{2}\right\},\\
A_2 := &\left\{(\rho, z, w) \in [0,1] \times B_r : 1-2\,r\leq \rho \leq 1 -r, 
\; |z| \leq \frac{1-\rho}{2}\right\},\\
B:=  & [1-2\,r, 1] \times B_r \setminus \big( A_1 \cup A_2 \big),
\end{align}
We then define the function $\HH:[0,1] \times B_r \to \I{Q}(\R^{n+2})$ given by
\begin{equation}\label{e:competitor}
\HH(\rho,z,w) :=
\begin{cases}
\rho\, \LL(z,w) & \text{if } \rho \leq 1 - 2\, r,\\
\rho\,l_1(\rho)\, \NN \left(\frac{2\,r\,z}{\rho+2r -1},\frac{(2\,r)^{\sfrac1{Q}}}{(\rho + 
2\,r -1)^{\sfrac1{Q}}}w\right)
& \text{if } (\rho,z,w) \in A_1,\\
- \rho\,l_1(\rho)\,  \LL \left(\frac{2\,r\,z}{1-\rho},\frac{(2\,r)^{\sfrac1{Q}}}{(1-\rho)^{\sfrac1{Q}}}w\right)
& \text{if } (\rho,z,w) \in A_2,\\
\rho\,l_2(|z|)\,  \NN \left(\frac{r\,z}{|z|},\frac{r^{\sfrac1{Q}}}{|z|^{\sfrac1{Q}}}w\right)
& \text{if } (\rho,z,w) \in B,
\end{cases}
\end{equation}
where $l_1, l_2 : \R \to \R$ are the affine functions
\begin{equation}
l_1(t) := \frac{t+r-1}{r}\quad \text{and}\quad 
l_2(t) := \frac{2\,t-r}{r}.
\end{equation}
The following are simple properties of $\HH$ which can be
easily verified:
\begin{itemize}
\item[(1)] $\HH(1,z,w) = \NN (z,w)$ for every $(z,w) \in B_r$, as 
$(1, z, w) \in A_1$ and $l_1(1)=1$;

\item[(2)] $\HH(\rho, z, w) = \rho\,\NN(z,w)$ for every $\rho \in [0,1]$ and
for every $z$ with $|z| = r$, as
$\LL \vert_{\de B_r} = \NN \vert_{\de B_r}$
and $l_2(r) =1$;

\item[(3)] $\HH$ is well-defined and continuous, as
$\HH \equiv 0$ in $A_1 \cap A_2$ from $l_1(1-r) =0$,
\[
\HH(\rho, z,w) =
\textstyle{
\rho \,\frac{\rho + r -1}{r}\,
\NN \left(\frac{r\,z}{|z|},\frac{r^{\sfrac1{Q}}}{|z|^{\sfrac1{Q}}}z\right)
\quad \text{in } A_1\cap \de B,
}
\]
and 
\[
\HH(\rho, z,w) =
\textstyle{
\rho \,\frac{\rho + r -1}{r}\,
\NN \left(\frac{r\,z}{|z|},\frac{r^{\sfrac1{Q}}}{|z|^{\sfrac1{Q}}}w\right)
\quad \text{in } A_2\cap \de B.
}
\]
\end{itemize}
The competitor map $K: \cC (r) \to \I{Q}(\R^{n+2})$ is now given by
\[
K(\rho\,p) := \sum_{i=1}^Q \a{\rho\,p +  H_i(\rho\,p)} \quad
\text{with }  H(\rho\, p) := \HH(\rho, \Psii^{-1}(p)).
\]
Note that by (1) and (2) above it follows that 
$K\vert_{\de \cC (r)} = \tilde F\vert_{\de \cC (r)}$.

We start now estimating the masses of the various currents introduced above.
Since $\supp(\bT_F) \subset \partial \bB_R (0)$,
it follows that $\mass(0\cone \bT_F) = R \mass(\bT_F) /3$ and,
by the expansion of the mass of $\bT_F$, we have that
\begin{align}\label{e:mass F}
\mass(\bT_{F|_{\Psii (B_r)}})
& \geq Q\, \cH^2(\Psii (B_r)) + \frac{1}{2} \int_{B_r} |D\NN|^2
- C \textup{Err}_1 (\NN, B_r)\, .
\end{align}
Combining the latter estimate with \eqref{e:minimalita'_cono} we conclude
\begin{equation}\label{e:minimalita'_cono_2}
\int_{B_r} |D\NN|^2 \leq 6R^{-1} \mass (\bT_K) - 2Q \cH^2 (\Psii (B_r)) + C\,\textup{Err}_1 (\NN, B_r)\, .
\end{equation}

\medskip

For what concerns the mass of $\bT_K$, recalling that $p+\supp(L(p))\in \partial \bB_R (0)$ for every $p \in \Psii (B_r)$, we deduce that
\[
\mass(\bT_K\res \bB_{R (1-2r})) = \mass(0 \cone \bT_G \res \bB_{R (1-2r)}) =
R \frac{(1-2r)^3\mass(\bT_G)}{3}
\]
and
\[
\mass(\bT_G) \leq Q\, \cH^2(\Psii (B_r)) + \frac{1}{2} \int_{B_r} |D\LL|^2 + \textup{Err}_2 (\LL, B_r)\, .
\]
In particular we conclude
\begin{equation}\label{e:stima_cono_1}
6 R^{-1} \mass (\bT_K\res \bB_{R(1-2r)}) \leq 2Q (1-2r)^3 \cH^2 (\Psii (B_r)) + \int_{B_r} |D\LL|^2 + \textup{Err}_2 (\LL, B_r)\, .
\end{equation}
Next we pass to estimating $\mass(\bT_K \res \bB_R\setminus \bB_{R(1-2r)})$. In order to carry on our estimates we use the area formula for multifunctions, cf. \cite{DS2}. In particular we fix an orthonormal frame $\xi_1, \xi_2$ for $\cM$ as in the proof of case (a) and we let $\xi_3 = R^{-1} \partial_t$ be normal to them in $T \cC (r)$, i.e. pointing in the radial direction of the cone. We then have
\[
\mass (\bT_K \res (\bB_R\setminus \bB_{R(1-2r)}) = \int_{\cC (r)} \sum_i \underbrace{|(\xi_1 + DH_i \cdot \xi_1)\wedge (\xi_2 + DH_i \cdot \xi_2)\wedge (\xi_3 + DH_i\cdot \xi_3)|}_{(A)}\, .
\]
Using the Taylor expansion for $(A)$, cf. \cite{DS2}, we can bound
\begin{align}\label{e:stima K}
& R^{-1} \mass(\bT_K \res (\bB_R\setminus \bB_{R(1-2r)}))
\leq Q R^{-1}\, \cH^3\big(\cC (r) \cap \bB_1 \setminus \bB_{1-2r}\big)\nonumber\\
&\qquad+ Q\, R^{-1} \int_{1-2r}^1 \int_{\Psii (B_r)} \frac{d}{dt} [(\etaa\circ H)(t p)] t^2 dt \notag\\
&\qquad + Q\, R^{-1} \int_{1-2r}^1
\int_{\Psii (B_r)} \sum_{i=1}^2 \langle \nabla_{\xi_i} (\etaa\circ H), \xi_i\rangle \,
t^2 dt + CR^{-1} \int_{1-2r}^1 \int_{\Psii (B_r)} |D H|^2\,t^2 dt\, .
\end{align}
The linear terms can be integrated by parts: since
$\nabla_p (\etaa \circ H)(tp) = \frac{d}{dt}(\etaa \circ H)(tp)$, we have
\begin{align}\label{e:linear 1}
 \int_{1-2r}^1 \int_{\Psii (B_r)} \frac{d}{dt} [(\etaa\circ H)(t p)] t^2 dt
& =  \int_{\Psii (B_r)} \left\langle (\etaa\circ H)(p) -
(1-2r)^2(\etaa\circ H)\big((1-2r)p\big), p \right\rangle\notag\\
&\quad - 2\int_{1-2r}^1\int_{\Psii (B_r)} \langle (\etaa\circ H)(t p), p\rangle\,t dt
\end{align}
\begin{align}\label{e:linear 2}
\int_{1-2r}^1 \int_{\Psii (B_r)} \sum_{i=1}^2 \langle \nabla_{\xi_i} (\etaa\circ  H), \xi_i\rangle \,t^2 dt
& = -
\int_{1-2r}^1 \int_{\Psii (B_r)} \langle (\etaa\circ H), H_{\cM}\rangle\, t^2 dt.
\end{align}
Therefore, by a simple change of coordinates we can estimate
\begin{align}\label{e:stima K bis}
&R^{-1} \mass(\bT_K \res (\bB_R\setminus \bB_{R(1-r)}))
\leq \frac{Q\, \big(1 - (1-2r)^{3}\big)}{3}\, \cH^2\big(\Psii (B_r)\big)\\
&\qquad + C \bmo^{\sfrac{1}{2}} \int_{B_r} \big(|\etaa\circ \NN| + |\etaa\circ \LL|\big) + C \bmo^{\sfrac{1}{2}} \int_{1-2r}^1 \int_{B_r} |D \HH|^2(t,z,w) \,dz\, dt\\
&\qquad + C\bmo^{\sfrac{1}{2}}\int_{1-2r}^1 \int_{B_r} |z|^{\gamma_0 -1} |\etaa \circ \HH|(t, z, w)dz\, dt\,  .
\end{align}
In order to bound the various integrands of \eqref{e:stima K bis},
we start with the following general remark.
Assume that $\chi: [1-2r,1] \times B_r \to [0,+\infty)$ has the structure
\begin{equation}\label{e:chi}
\chi(\rho,x,y) =
\begin{cases}
\chi_1\left(\frac{2\,r\,z}{\rho+2r -1},\frac{(2\,r)^{\sfrac1{Q}}}{(\rho + 
2\,r -1)^{\sfrac1{Q}}}w\right)
& \text{if } (\rho,z,w) \in A_1,\\
\chi_2\left(\frac{2\,r\,z}{1-\rho},\frac{(2\,r)^{\sfrac1{Q}}}{(1-\rho)^{\sfrac1{Q}}}w\right)
& \text{if } (\rho,z,w) \in A_2,\\
\chi_3\left(\frac{r\,z}{|z|},\frac{r^{\sfrac1{Q}}}{|z|^{\sfrac1{Q}}}w\right)
& \text{if } (\rho,z,w) \in B,
\end{cases}
\end{equation}
for some $\chi_1, \chi_2, \chi_3:B_r \to [0,+\infty)$.
Then one can compute the integral of $\chi$ in the following way:
\[
\int_{1-2r}^1 \int_{B_{r}} \chi(t,z,w)\,dz\,dt =
\int_{A_1}\chi(t,z,w)\,dz\,dt + \int_{A_2}\chi(t,z,w)\,dz\,dt +
\int_{B}\chi(t,z,w)\,dz\,dt,
\]
and one can easily compute that
\begin{align}
&\int_{A_1} \chi (t,z,w)\, dz\, dt
 = \int_{1-r}^1 \int_{B_{\frac{t+2r-1}{2}}} \chi_1(t,z,w)\, dz\, dt\notag\\
=& \int_{1-r}^1 \int_{B_{\frac{t+2r-1}{2}}}
\chi_1\left(\frac{2\,r\,z}{t+2r -1},\frac{(2\,r)^{\sfrac1{Q}}}{(t + 2\,r -1)^{\sfrac1{Q}}}w\right)dz\, dt\notag\\
=& \int_{1-r}^1\left(\frac{t+2r-1}{2r}\right)^2 \int_{B_{r}}
\chi_1(z,w) dz\, dt \leq r \int_{B_{r}} \chi_1(z,w) dz\, dt\, .\label{e:int A1}
\end{align}
Similarly 
\begin{align}\label{e:int A2}
\int_{A_2} \chi (t,z,w) dz\, dt \leq r \int_{B_{r}} \chi_2(z,w)\, dt,
\end{align}
and 
\begin{align}\label{e:int B}
\int_{B} \chi(t,z,w) dz\, dt= &  \int_{1-r}^1 dt \int^{r}_{\frac{t+2r-1}{2}} \frac{s}{r}\, ds
\int_{\de B_r} \chi_3(z,w)\,dz\nonumber\\
&\quad + \int_{1-2r}^{1-r} \int_{\frac{1-t}{2}}r \frac{s}{r}\, ds\int_{\de B_r} \chi_3(z,w)\,dz
 \leq  r^2\int_{\de B_{r}} \chi_3(z,w)\,dz\, .
\end{align}
By direct computations one verifies that the integrands in \eqref{e:stima K bis}
are all bounded from above by functions $\chi$ with the structure \eqref{e:chi}:
in particular,
\begin{itemize}
\item[(i)] $|z|^{\gamma_0 -1} |\etaa \circ  \HH|(t,z,w) \leq \chi(t,z,w)$
if we choose
\[
\chi_1(z,w) = \chi_3(z,w) = |z|^{\gamma_0 -1} |\etaa \circ \NN|(z,w)
\quad \text{and}\quad
\chi_2(z,w) = |z|^{\gamma_0 -1} |\etaa \circ  \LL|(x,y);
\]

\item[(ii)] $|D\HH|^2(t,z,w)\leq \chi(t,z,w)$
if we choose
\begin{gather*}
\chi_1(z,w) = \chi_3(z,w) = \frac{C}{r^2} |\NN|^2(z,w) + C\, |D\NN|^2(z,w)\\
\chi_2(z,w) = \frac{C}{r^2} |\LL|^2(z,w) + C\,|DL|^2(z,w).
\end{gather*}
for some dimensional constant $C>0$.
\end{itemize}
It then turns out from \eqref{e:int A1}, \eqref{e:int A2}, \eqref{e:int B}
and (i), (ii), (iii) that 
\begin{align}\label{e:stima K tris}
6 R^{-1} \mass(\bT_K \res (\bB_R\setminus \bB_{R(1-r)}))&
\leq Q\, \big(1 - (1-2r)^{3}\big)\, \cH^2\big(\Psii (B_r)\big)\notag\\
&\qquad + C\, \textup{Err}_1(\NN, B_r) + C\, \textup{Err}_2 (\LL, B_r)\, .
\end{align}
Summing \eqref{e:stima K tris} and \eqref{e:stima_cono_1} we conclude
\[
6 R^{-1} \mass (\bT_K) \leq 2 Q \cH^2 (\Psii (B_r)) + \int_{B_r} |D\LL|^2 + C\, \textup{Err}_1 (\NN, B_r) + C\, \textup{Err}_2 (\LL, B_r)\, .
\]
Combining the latter estimate with \eqref{e:minimalita'_cono_2} we conclude the proof.

\section{Harmonic competitor}\label{s:harmonic}

The most natural choice for the competitor $\LL$ is a suitable ``harmonic'' extension of the boundary value $\NN|_{\partial B_r}$. Following the ideas of \cite{Chang} we estimate carefully the energy of such competitor. To this purpose it is useful to introduce ``polar'' coordinates with center $0$ in $\gira$ and split accordingly the Dirichlet integrand in radial and angular parts. More precisely, consider $(z_0,w_0) = ((\xi_0, \zeta_0), w_0) \in \partial B_r$ and take, locally, the standard flat coordinates $z = (x_1, x_2)$ of Definition \ref{d:Riemann_surface}. We then denote by $\nu$ the exterior unit vector normal to $\partial B_r$ at $(z_0, w_0)$ and by $\tau$ the corresponding tangent unit vector obtained by rotating $\nu$ of an angle $\pi/2$ in the counterclockwise direction, namely
\[
\nu := |z_0|^{-1} \left( \xi_0 \frac{\partial}{\partial x_1} + \zeta_0 \frac{\partial}{\partial x_2}\right)
\qquad \mbox{and}\qquad \tau := |z_0|^{-1} \left( - \zeta_0 \frac{\partial}{\partial x_1} + \xi_0 \frac{\partial}{\partial x_2}\right)\, .
\]
 The directional derivatives of any (multi)function $f$ on $\gira$ gives then two (multi)functions
\[
D_\nu f = \sum_i \a{Df_i\cdot \nu}\qquad \mbox{and} \qquad D_\tau f = \sum_i \a{Df_i \cdot \tau}\, .
\]
The Dirichlet integrand $|Df|^2$ enjoys then the splitting
\[
|Df|^2 = |D_\nu f|^2 + |D_\tau f|^2\, .
\]
For the rigorous justification of these identities see \cite{DS1}.

\begin{proposition}\label{p:harmonic}
There are constants $C>0$, $\sigma >0$ such that,  for every $r\in (0,1)$ there exists a competitor
$\LL \colon B_r\to \I{Q} (\R^{2+{n}})$ for $\NN$ with the following additional properties:
\begin{itemize}
\item[(i)] $\Lip (\LL) \leq C_{\ref{p:amin}}^{-1}
$, $\|\LL\|_0 \leq C r$.
\item[(ii)] The following estimates hold:
\begin{gather}
\int_{B_r}|D\bar \LL|^2\leq C\, r \int_{\de B_r} |D\bar \NN|^2 \leq C r \bD' (r) \, , \label{e:D<rD'}\\
\int_{B_r}|z|^{\gamma_0-1}|\etaa\circ \LL| \leq C\,r^{\gamma_0}  \int_{\de B_r} |\etaa\circ \NN| + C\, \bH (r) \, . \label{e:comp_mean_H}
\end{gather}
\item[(iii)] For every $a>0$ there exists $b_0>0$ such that, for all $b \in (0, b_0)$, the following estimate holds:
\begin{align}\label{e:harmonic energy}
(2\,a+b)\int_{B_r} |D\bar \LL|^2
& \leq r \int_{\de B_r} |D_\tau \NN|^2
+  \frac{a\,(a+b)}{r}\int_{\de B_r} |\NN|^2 + C r^{1+\sigma} \bD' (r)\, .
\end{align}
\end{itemize}
\end{proposition}

Using this competitor in Proposition \ref{p:amin}, we then infer the following corollary.

\begin{corollary}\label{c:amin_ineq} For every $r\in (0,1)$ the following inequality holds 
\begin{align}\label{e:AM1 bis}
\bD(r) & \leq C\,r\,\bD'(r) +  C\, \bH(r)+ C\,\bF(r)
+C\,\bmo^{\sfrac12}\, r^{\gamma_0} \int_{\de B_r}|\etaa\circ \NN|\,.
\end{align}
For every $a>0$ there exists $b_0>0$ such that, for all $b \in (0, b_0)$ and all $r\in ]0,1[$
\begin{align}\label{e:AM2 bis}
\bD(r) &\leq 
(1+Cr) \left[\frac{r}{\,(2\,a+b)}\,
\int_{\de B_r} |D_\tau \NN|^2
+ \frac{a\,(a+b)}{r\,(2\,a+b)}\bH(r)\right]+C\, \mathcal{E}_{QM}(r)+ C r^{1+\sigma} \bD' (r)\,,
\end{align}
with
\[
\mathcal{E}_{QM}(r)\leq \bLambda (r)^{\eta_0} \bD(r) + \bF(r) + \bH(r) 
+ \bmo^{\sfrac12}\, r^{\gamma_0} \int_{\de B_r}|\etaa\circ \NN|\, .
\]
\end{corollary}

\begin{proof}[Proof of Corollary \ref{c:amin_ineq}]
Recalling that $\bH (r) \leq C r \|\NN\|^2_{\partial B_r} \leq C r^{3+\gamma_0}$ we easily infer that $\bLambda (r) \leq C r^2$ and thus the inequalities follow readily from Proposition \ref{p:amin} and Proposition \ref{p:harmonic}.
\end{proof}

\subsection{Proof of Proposition \ref{p:harmonic}: Step 1} First of all we observe that it suffices to exhibit $\bar \LL$, as $\LL$ can be recovered from it via the formula \eqref{e:relazione}. Moreover, it suffices to show the estimates with
$\bar \NN$ in place of $\NN$ in the right hand side, because we obviously have $|\bar \NN|\leq |\NN|$ and $|D\bar \NN|\leq |D\NN|$. Next we wish to relate $\etaa\circ \LL$ and $\etaa\circ \bar \LL$ for two maps satisfying the relation \eqref{e:relazione}. Note that by a simple Taylor expansion we have 
\[
|\etaa \circ \LL|\leq C |\etaa\circ \bar \LL| + C \cG (\bar \LL, \etaa\circ \bar \LL)^2\, ,
\]
where the constant $C$ depends on the $C^2$ norm of $\Psi_0$. In particular we record the following conclusion:
\begin{equation}\label{e:dettagliuccio}
\int_{B_r} |z|^{\gamma_0-1} |\etaa\circ \LL| \leq C \int_{B_r} |z|^{\gamma_0-1} |\etaa\circ \bar \LL| + C \int_{B_r} |z|^{\gamma_0-1} |\bar \LL|^2\, .
\end{equation}

In this step we exhibit an ``harmonic''\footnote{We remark that the competitor used here does not coincide, in general, with the Dirichlet minimizer with boundary value $\bar \NN|_{\partial B_r}$.} competitor $\HH$ which satisfies all the requirements of the proposition except for the Lipschitz estimate. In fact we will show that there is a $W^{1,2}$ map $\HH: B_r \to \Iq (\R^{2+\bar n})$ such that
\begin{align}
&\HH|_{\partial B_r} = \bar \NN|_{\partial B_r} \qquad \mbox{and}\qquad \|\HH\|_{L^\infty (B_r)} \leq Q \|\bar \NN\|_{L^\infty (\partial B_r)}\label{e:armonica1}\\
&\int_{B_r} |D\HH|^2 \leq C r \int_{\partial B_r} |D\bar \NN|^2 \label{e:armonica2}\\
&\int_{B_r} |z|^{\gamma_0-1} |\etaa\circ \HH| \leq C r^{\gamma_0} \int_{\partial B_r} |\etaa\circ \bar \NN|\label{e:armonica3}\\
&\int_{B_r} |z|^{\gamma_0-1} |\HH|^2 \leq C r^{\gamma_0} \int_{\partial B_r} |\bar \NN|^2\label{e:armonica5}\\
&(2\,a+b)\int_{B_r} |D\bar \HH|^2 \leq r \int_{\de B_r} |D_\tau \bar \NN|^2
+  \frac{a\,(a+b)}{r}\int_{\de B_r} |\bar \NN|^2\, .\label{e:armonica4}
\end{align}
In these estimates we do not use any of the particular properties of $\bar\NN$ and indeed for any Lipschitz multivalued map $\bar\NN: B_r \to \Iq (\R^{2+\bar{n}})$ there is such an ``harmonic'' competitor. Therefore, given the scaling invariance  of the estimates, we will assume without loss of generality that $r=1$.  

Let $D_r :=\{|z|<r\}$ denote the disk of radius $r$ in $\R^2$, which we identify with the complex plane.
We start by defining the ``winding map'' $\bW: \bar D_1 \to \gira$ given (in complex notation) by 
\[
\bW (z) := (z^{\bar Q}, z)\, .
\]
We then consider the multivalued map $\UU :=\bar \NN \circ \bW$. Let $\theta \mapsto u (\theta)$ be its trace on $\partial D_1 (0)$, which we parametrize with the angle $\theta\in [0,2\pi]$. According to \cite[Proposition 1.5]{DS1} we can decompose $u$ in a superposition of simple functions $u (\theta)=\sum_{j=1}^J u_j(\theta)$ such that, for every $j=1,\dots,J$,
\[
u_j(\theta)=\sum_{i=1}^{Q_j}\a{\gamma_j\left(\frac{\theta+2\pi i}{Q_j}\right)}\, ,
\]
where the $\gamma_j:[0, 2\pi]\to \R^{2+\bar n}$ are periodic Lipschitz functions.
Next consider the Fourier's expansion of each $\gamma_j$
\[
 \gamma_j(\theta)=\frac{a_{j,0}}{2}+\sum_{l=1}^\infty \left(a_{j,l}\cos(l\theta)+b_{j,l}\sin (l\theta)\right)\,,
\]
and its harmonic extension, which in polar coordinates $(\rho, \theta)$ reads as 
\begin{equation}\label{e:zeta_j}
 \zeta_j(\rho,\theta):=\frac{a_{j,0}}{2}+\sum_{l=1}^\infty \rho^l \big(a_{j,l}\cos(l\theta )+b_{j,l}\sin (l\theta)\big)\, .
\end{equation}
We then can define the ``harmonic'' competitor for $\UU$, which is the $Q$-valued map
\[
\VV (\rho, \theta) := \sum_{j=1}^J\sum_{i=1}^{Q_j}\a{\zeta_j\left(\rho^{\sfrac{1}{Q_j}}, \frac{\theta+2\pi i}{Q_j}\right)}
\]
and the ``harmonic'' competitor for $\bar\NN$, which is  $\HH = \VV\circ \bW^{-1}$.
Observe that the first claim in \eqref{e:armonica1} is obvious, whereas the second claim follows from the maximum principle 
for classical harmonic functions.

Simple computations and the conformality of $\bW$, see for instance \cite[Proof of Proposition 5.2]{DS1}, yield
\begin{align}
\int_{B_1} |D \HH|^2 = \int_{D_1} |D\VV|^2 = & \pi \sum_{j=1}^{J} \sum_{l=1}^\infty l \big(|a_{j,l}|^2+|b_{j,l}|^2\big)\,, \label{e:D}\\
\int_{\de B_1} |D_\tau \HH|^2= &\frac{\pi}{\bar Q}\sum_{j=1}^J\sum_{l=1}^\infty \frac{l^2}{Q_j}\big(|a_{j,l}|^2+|b_{j,l}|^2\big)\,,\label{e:D'}
\end{align}
\begin{align}
\int_{\de B_1} |\HH|^2=& \pi \bar Q \sum_{j=1}^JQ_j\Big(\frac{|a_{j,0}|^2}{2}+\sum_{l=1}^\infty \big(|a_{j,l}|^2+|b_{j,l}|^2\big)\Big)\label{e:h'}\, .
\end{align}
Clearly, \eqref{e:armonica2} follows from the first and second inequality, with the constant $C = \bar{Q} Q_1 \leq \bar{Q} Q$, assuming
that $Q_1 = \max \{Q_1, \ldots, Q_j\}$. 
\eqref{e:armonica4} follows from the fact that, for any chosen $a>0$, if $b_0$ is sufficiently small and $0<b< b_0$, then
\[
(2a+b) \ell \leq \frac{\ell^2}{\bar Q Q_j} + \bar{Q} Q_j \ell a (a+b) \qquad \forall \ell \in \N\, .
\]
The latter claim is elementary and the reader can consult, for instance, Step 2 in the proof of \cite[Proposition 5.2]{DS1}.

Observe next that $\etaa\circ \VV$ is the classical harmonic extension of the single-valued function $\etaa\circ \UU|_{\partial D_1}$. 
We then have the classical estimates 
\[
\|\etaa\circ \VV\|_{L^\infty (D_{2^{ \sfrac{1}{\bar Q} }})}+ \|\etaa\circ \VV\|_{L^1 (D_1)} \leq C \|\etaa\circ \UU\|_{L^1 (\partial D_1)}\, .
\]
In particular we conclude easily
\[
\|\etaa\circ \HH\|_{L^\infty (B_{1/2})} + \|\etaa\circ \HH\|_{L^1 (B_1\setminus B_{1/2})} \leq C \int_{\partial B_1} |\etaa\circ \bar\NN|\, ,
\]
because the change of variables $\bW^{-1}$ is smooth on $B_1\setminus B_{1/2}$. The integrability of $|z|^{\gamma_0-1}$ on $B_1$ gives then 
\begin{align*}
\int_{B_1} |z|^{\gamma_0-1} |\etaa\circ \HH (z,w)|\, dz\leq & C \|\etaa\circ \HH\|_{L^\infty (B_{1/2})} + C\|\etaa\circ \HH\|_{L^1 (B_1\setminus B_{1/2})}\, , 
\end{align*}
which in turn completes the proof of \eqref{e:armonica3}. 

A similar argument proves \eqref{e:armonica5}. Using the classical theory of single valued harmonic functions we see indeed that 
$\|\zeta_j\|_{L^2 (B_1)} + \|\zeta_j\|_{L^\infty (B_{1/2})} \leq C \|\gamma_j\|_{L^2 (\partial B_1)}$ and thus, using the fact that $\bW$ is smooth on $B_1\setminus B_{1/2}$, we conclude that
\[
\|\HH\|^2_{L^\infty (B_{1/2})} + \|\HH\|^2_{L^2 (B_1\setminus B_{1/2})} \leq C \int_{\partial B_1} |\bar \NN|^2\, .
\]
From this we easily conclude \eqref{e:armonica5}.

\subsection{Proof of Proposition \ref{p:harmonic}: Step 2} We keep the notation of the previous paragraphs and assume that $\bar \NN$ is defined in $B_1$, after scaling. The specific scaling that we are using is the one which preserves the Lipschitz constant and is given by
\[
\bar\NN (z,w) \mapsto r^{-1} \bar \NN \big(r z, r^{\sfrac{1}{\bar{Q}}} w\big)\, 
\]
and by abuse of notation we keep the symbols $\bar \NN$, $\bar \LL$, etc. for all the rescaled maps.

Under this scaling we then have the estimates $\|\bar \NN\|_{L^\infty}\leq C \bmo^{\sfrac{1}{4}} r^{\gamma_0/2}$ and $\Lip (\bar \NN) \leq \bLambda (r)^{\eta_0}$ and we want to show that we can modify $\HH$ to a competitor $\bar \LL$ with $\Lip (\bar\LL)\leq C_{\ref{p:amin}}^{-1}$, satisfying 
\begin{align}
&\bar \LL|_{\partial B_1} = \bar \NN|_{\partial B_1} \qquad \mbox{and}\qquad \|\bar \LL\|_{L^\infty (B_1)} \leq C \|\bar \NN\|_{L^\infty (\partial B_1)}\label{e:app_arm_1}\\
&\int_{B_1} |D\bar \LL|^2 \leq C(1+r^\sigma) \int_{B_1} |D\HH|^2 + C \bLambda (r)^\sigma \int_{\partial B_1} |D\bar\NN|^2 \label{e:app_arm_2}\\
&\int_{B_r} |z|^{\gamma_0-1} |\bar\LL|^2 \leq C \int_{\partial B_1} |\bar \NN|^2\, \label{e:app_arm_5}\\
&\int_{B_1} |z|^{\gamma_0-1} |\etaa\circ \bar\LL| \leq C \int_{\partial B_1} |\etaa\circ \bar \NN|\, .\label{e:app_arm_3}
\end{align}
Observe that the harmonic functions $\zeta_j$ defined in \eqref{e:zeta_j} are Lipschitz in every ball $D_{1-t}$ for $0<t<1$ with an estimate of the form
\begin{equation}\label{e:Lipschitz_sciocca}
\|D \zeta_j\|_{L^\infty (D_{1-t})} \leq \frac{C}{t} \Lip (\gamma_j) \leq \frac{C}{t} \Lip (\bar \NN) \leq \frac{C \bLambda (r)^{\eta_0}}{t}\, .
\end{equation}
They are not Lipschitz up to the boundary $\partial D_1$ because the Dirichlet to Neumann map $\gamma_j \to \frac{\partial \zeta_j}{\rho} (1, \cdot)$ does not map $L^\infty$ into $L^\infty$. However we have the estimate
\[
\|D \zeta_j \|_{L^p (D_1)} \leq C_p \|\gamma_j\|_{W^{1,p} (\partial D_1)} \leq C_p\bLambda (r)^{\eta_0}\, 
\]
for every $p<\infty$. In particular, we can bound
\[
\|\zeta_j (1-t, \cdot)-\gamma_j \|_{W^{1,1} (\partial D_1)} \leq C_2 t^{\sfrac{1}{2}}\bLambda (r)^{\eta_0} \, ,
\]
which in turn implies
\begin{equation}\label{e:Lipschitz_scema}
\max |\zeta_j (1-t, \theta) - \gamma_j (\theta)| \leq C_2 t^{\sfrac{1}{2}}\bLambda (r)^{\eta_0}\, .
\end{equation}
Choose $t:= \bLambda (r)^{\sfrac{\eta_0}{2}}$ and define a new map $\xi_j$ as
\[
\xi_j (\rho, \theta) := \left\{
\begin{array}{ll}
\zeta_j (\rho, \theta) \quad & \mbox{for $\rho\leq 1-t$}\\ \\
\textstyle{\frac{1-\rho}{t}} \zeta_j (1-t, \theta) + \textstyle{\frac{\rho - (1-t)}{t}} \gamma_j (\theta) \quad &\mbox{for $1-t \leq \rho \leq 1$.}
\end{array}
\right.
\]
Now, \eqref{e:Lipschitz_sciocca} and \eqref{e:Lipschitz_scema} imply that $\|D\zeta_j\|\leq C \bLambda (r)^{\sfrac{\eta_0}{2}}$. Moreover we obviously have
\begin{align}
\int_{D_1} |D\xi_j|^2 \leq & \int_{D_1} |D\zeta_j|^2 + C \bLambda (r)^{\eta_0} \Big( \int_{\partial D_{1-t}} |D \zeta_j|^2 + \int_{\partial D_1} |D\gamma_j|^2\Big)\nonumber\\
\leq & \int_{D_1} |D\zeta_j|^2 + C r\bLambda (r)^{\eta_0} \int_{\partial B_1} |D\gamma_j|^2\, .\label{e:int_en}
\end{align}
We can now define two ``intermediate'' maps 
\[
\VV^0 (\rho, \theta) := \sum_{j=1}^J\sum_{i=1}^{Q_j}\a{\xi_j\left(\rho^{\sfrac{1}{Q_j}}, \frac{\theta+2\pi i}{Q_j}\right)}
\]
and $\LL^0 := \VV^0\circ \bW^{-1}$. It is then immediate to see that $\LL^0$ enjoys the bound $\Lip (\LL^0) \leq C \bLambda (r)^{\sfrac{\eta_0}{2}}$ on the 
domain $B_1\setminus B_{1/4}$ and that all the estimates \eqref{e:app_arm_1}, \eqref{e:app_arm_2} and \eqref{e:app_arm_3}. On the other hand the differential $D \LL^0$ is singular in the origin and in fact it is rather easy to see that we have the bound
\begin{equation}\label{e:higher}
|D\LL^0 (z,w)|^2 \leq C |z|^{2-\sfrac{2}{(Q\bar Q)}} \int_{B_1} |D\LL^0|^2\, .
\end{equation}
In order to produce $\bar \LL$ we need to smooth the singularity of $\LL^0$ at the origin. There are several ways to do this and we present here one possibility. First of all we fix $2<p< 2 Q \bar Q/(2 Q \bar Q -2)$ and observe that \eqref{e:higher} yields the estimate
\begin{equation}\label{e:higher2}
\int_{B_{3/4}} |D\LL^0 (z,w)|^p \leq C \Big( \int_{B_1} |D\LL^0|^2 \Big)^{\sfrac{p}{2}}\, .
\end{equation}
Next we define
\[
M |D\LL^0 (z,w)|:= \sup_{\rho< 1/4} \frac{1}{\rho^2} \int_{B_\rho (z,w)} |D\LL^0 (z,w)|
\]
and let
\[
A := \{(z,w) : M |D\LL^0 (z,w)| \geq c_0\}
\]
where $c_0$ is a constant to be chosen later. Observe that, given the Lipschitz bound for $\LL^0$ outside the origin, for $r$ sufficiently small the set $A$ is contained in $B_{1/2}$. Arguing as in the proof of \cite[Proposition 4.4]{DS1} we have the Lipschitz estimate $\Lip (\LL^0) \leq C c_0$ on $B_1\setminus A$, where $C$ is a dimensional constant. We can then use the Lipschitz extension of \cite[Theorem 1.7]{DS1} to extend $\LL^0$ to $\bar \LL$ on $A$ so that $\Lip (\LL) \leq C c_0$. Choosing $c_0$ accordingly we achieve the desired Lipschitz bound on $B_1$. As for \eqref{e:app_arm_1} and \eqref{e:app_arm_5} observe that the extension satisfies 
\[
\|\bar \LL\|^2_{L^\infty (B_{1/2})} \leq C \|\HH\|^2_{L^\infty (B_{3/4})}
\]
and coincides with $\LL_0$ on $B_1\setminus B_{1/2}$. As for \eqref{e:app_arm_3}, it would suffice to show that
$|\etaa\circ \bar\LL|\leq C |\etaa\circ \bar \NN|$. This can be easily achieved in the following way: we make a Lipschitz extension of $\LL^0$, subtract from each sheet the average and then sum back to each sheet a Lipschitz extension of $\etaa\circ \LL^0$. 

As for \eqref{e:app_arm_2} we compute
\begin{align}
\int |D\bar \LL|^2 \leq & \int |D\LL^0|^2 + C c_0^2 |A| \leq  \int |D\LL^0|^2 + C c_0^{2-p} \int_{B_{3/4}} |D\LL^0|^p\notag\\
\leq & \int |D\LL^0|^2 \Big(1+ C c_0^{2-p} \Big(\int |D\LL^0|^2\Big)^{\sfrac{p}{2}-1}\Big)\, .\label{e:int_en_2}
\end{align}
Observe that $p/2 -1 > 0$ and that by \eqref{e:int_en} and \eqref{e:armonica2}
\[
\int |D\LL^0|^2 \leq \int |D\HH|^2 + C \bLambda (r)^{\sfrac{\sigma}{2}} \int_{\partial B_1} |D \bar \NN|^2 \leq C \int_{\partial B_1} |D \bar \NN|^2 \leq C r^{\sigma}\, .
\]
so that
\[
\int_{B_1} |D\bar \LL|^2 \leq (1+C\, r^\sigma) \int_{B_1} |D\HH|^2 + C r^\sigma \int_{\partial B_1} |D\bar\NN|^2
\stackrel{\eqref{e:armonica2}}{\leq}  \int_{B_1} |D\HH|^2 + C r^\sigma \int_{\partial B_1} |D\bar\NN|^2\, .
\]

\section{Outer variations and the poincar\'e inequality}\label{s:poincare}

In this section we begin to exploit the first variations of the area functional on $T$ in conjunction with the estimates of the previous section. 
The main conclusion will be the following Poincar\'e inequality: 

\begin{theorem}[Poincar\'e inequality]\label{t:poincare}
There exists a constant $C_{\ref{t:poincare}}>0$ such that if $r$ is sufficiently small, then
\begin{equation}\label{e:poincare}
\bH(r) \leq C_{\ref{t:poincare}}\, r\, \bD(r)\,.
\end{equation}
\end{theorem}

We record however the two main tools used to prove Theorem \ref{t:poincare}, since they will be useful in the future.
The first one is an elementary computation. In order to state it we introduce the quantity
\begin{equation}\label{e:E}
\bE (r) := \int_{\de B_r} \sum_{j=1}^{Q} \langle \NN\,_j, D_\nu \NN\,_j\rangle\, .
\end{equation}

\begin{lemma}\label{l:H'}
$\bH$ is a Lipschitz function and  the following identity holds for a.e. $r \in (0,1)$ 
\begin{equation}\label{e:H'}
\bH' (r) = \frac{\bH(r)}{r} + 2\,\bE(r)\,.
\end{equation}
\end{lemma}

The second identity is a consequence of the first variations of $T$ under specific vector fields, which we call ``outer variations'': such variations ``stretch'' the normal bundle of $\cM$ suitably and they are defined using the map $\NN$. In the case of semicalibrated currents it is convenient to modify the Dirichlet energy suitably to gain a new quantity which enjoys better estimates. Thus, from now on $\bOmega$ will denote $\bD$ in the cases (a) and (c) of Definition \ref{d:semicalibrated}, whereas in the case (b) it will be given by
\begin{align*}
\bOmega (r) := &\bD (r) + \bL (r) \\
:= &\bD (r) + \int_{\Psii(B_r)}\sum_{i=1}^{Q}\, \langle \xi_1 (p) \wedge D_{\xi_2} N_i (p) \wedge N_i (p) + D_{\xi_1} N_i (p)
\wedge \xi_2 (p) \wedge  N_i (p),d\omega (p)\rangle\, dp\, .
\end{align*}

\begin{proposition}[Outer variations]\label{p:OV}
There exist constants $C_{\ref{p:OV}}>0$ and $\kappa >0$ such that,
if $r>0$ is small enough, then the inequality
\begin{equation}\label{e:OV}
\left|\bOmega (r)- \bE (r)\right| \leq C_{\ref{p:OV}}\,\mathcal{E}_{OV}(r)\,
\end{equation}
holds with
\begin{align}\label{e:stima_Eov_1}
\mathcal{E}_{OV}(r)& = \bLambda(r)^{\kappa} \Big(\bD(r)+ \frac{\bH(r)}{r} + r \bD' (r)\Big) + \bF(r)  +
 r^{1+\gamma_0} \frac{d}{dr}\|T-\bT_F\|(\p^{-1}(\Psii(B_r)))   \,.
\end{align}
Moreover
\begin{align}\label{e:L0}
|\bL(r)| \leq & C\,\bmo^{\sfrac12}\,r^{2-\gamma_0} \bD(r) + C\,\bmo^{\sfrac12}\,\bF(r).
\end{align}
\end{proposition}

\subsection{Proof of Lemma \ref{l:H'}} The Lipschitz regularity of $\bH$ follows from the Lipschitz regularity of $\NN$. 
Consider next the map $i_r : \gira \to \gira$ given by
$i_r(z,w)=\left(rz, r^{\sfrac{1}{\bar Q}}w\right)$.
By a simple change of variables we compute
\[
\bH(r)=\int_{\de B_1} |\NN|^2(i_r(z',w'))\, r\, .
\]
The formula \eqref{e:H'} is then an elementary computation using the chain rule for multifunctions, cf. \cite{DS1}. 

\subsection{Proof of Proposition \ref{p:OV}}\label{ss:OV} The inequality \eqref{e:L0} is a simple consequence of 
\[
|\bL (r)| \leq C\bmo^{\sfrac{1}{2}} \int_{B_r} |D\NN| |\NN| \leq C\bmo^{\sfrac{1}{2}} \int_{B_r} |z|^{2-\gamma_0} |D\NN|^2 
+ C \bmo^{\sfrac{1}{2}} \int_{B_r} |z|^{\gamma_0-2} |\NN|^2\, .
\]

\medskip

In order to show \eqref{e:OV} we fix a test function $\phi\in C^\infty_c (\R)$, nonnegative, symmetric, with support in $]-1,1[$ and monotone decreasing on $[0,1]$. We then follow \cite[Section 3.3]{DS5} and, having fixed $r$, we define the vector field $X^o$ on $\bV_{u,a}$ via
\[
X^o (p) := \varphi (\p (p)) (p-\p (p))\qquad
\mbox{where}\qquad
\varphi (\Psii (z,w)) = \phi \left(\textstyle{\frac{|z|}{r}}\right)\, .
\]
For $r$ small enough, by \eqref{e:Lip_N} we can argue as in \cite[Section~3.3]{DS5} and 
deduce via the change of coordinates given by $\Psii$, that  
\begin{equation}\label{e:OV_0}
\delta \bT_F(X)= \int_{\gira}\phi \big( \textstyle{\frac{|z|}{r}}\big)\,|D\NN|^2 + r^{-1} \int_{\gira} \phi' \big( \textstyle{\frac{|z|}{r}}\big) \,\sum_{j=1}^{Q} \langle \NN\,_j, D_\nu \NN\,_j\rangle + \sum_{i=1}^3 \textup{Err}_i^o,
\end{equation}
with
\begin{align}\label{e:outer_resto_1}
{\rm Err}_1^o & = \Big\vert\int_\cM \varphi\, \langle H_\cM, \etaa\circ {N}\rangle\Big\vert \leq C\, 
\bmo^{\sfrac12}\int_{B_r} \,|z|^{\gamma_0 - 1}\, |\etaa\circ \NN|
\stackrel{\eqref{e:media_pesata}}{\leq} C\, \bLambda^{\eta_0}(r)\,\bD(r) + C\,\bF(r)\,,
\end{align}
\begin{equation}\label{e:outer_resto_2}
{\rm Err}_2^o \leq C \int_\cM |\varphi| \,|A_{\cM}|^2\,|N|^2
\leq C\, \bF(r)\,,
\end{equation}
\begin{align}\label{e:outer_resto_3}
{\rm Err}_3^o \leq & C \int_\cM \Big(|\varphi| \big(|D{N}|^2\, |{N}|\, |A_{\cM}|
+ |D{N}|^4\big) +
|D\varphi|\, \big(|DN|^3 \,|{N}| + |D{N}|\, |{N}|^2 \,|A_{\cM}|\big)\Big)\notag\\
\leq & C \int_{B_r}  \Big[\Big(\frac{|{\NN}|^2}{|z|^{2-2\,\gamma_0}}+ 
|D{\NN}|^4\Big)- r^{-1} \phi' ( \textstyle{\frac{|z|}{r}})\,r^{1+\gamma_0}\,|D \NN|^3 - r^{-1} \phi' ( \textstyle{\frac{|z|}{r}})\,|D 
\NN|\,\frac{|\NN|^2}{|z|^{1-\gamma_0}}\Big] \notag\\
\stackrel{\eqref{e:Lip_N}\&\eqref{e:Ndecay}}{\leq} &
C\,\bLambda^{\eta_0}(r)\,\bD(r)+C\bF(r)
- C \,\bLambda(r)^{\eta_0}\,\int_{B_r} r^{-1} \phi' ( \textstyle{\frac{|z|}{r}})\,\frac{|\NN|^2}{|z|^{1-\gamma_0}}\notag\\
&\qquad - C \,r^{1+\gamma_0}\,\bLambda^{\eta_0}\,\int_{B_r} r^{-1}  \phi' ( \textstyle{\frac{|z|}{r}})\,|D\NN|^2\,.
\end{align}
(We recall that $\phi' \leq 0$ on $[0,1])$). 

\medskip

We next drop the superscript from $X^o$ and we distinguish two situations:
\begin{itemize}
\item In the cases (a) and (c) of Definition \ref{d:semicalibrated}, we denote by $X^\perp$ and $X^T$
the projections of $X$ on the normal and the tangential bundle of $\Sigma$, respectively. Then $\delta T(X^T) 
=0$ and therefore
\begin{equation*}
|\delta \bT_F(X)| \leq \underbrace{|\delta \bT_F(X)-\delta T(X)|}_{\textup{Err}_4^o}+\underbrace{|\delta T(X^\perp)|}_{\textup{Err}^o_5};
\end{equation*}
\item In case (b), since $\delta T(X) = T(dw\ser X)$, we estimate
\begin{equation*}
\big|\delta \bT_F(X) - \bT_F(d\omega\ser X)\big|
\leq \underbrace{|\delta \bT_F(X)-\delta T(X)|+ |T(d\omega \ser X)-\bT_F(d\omega\ser X)|}_{\textup{Err}_4^o}\, .
\end{equation*}
\end{itemize}
In both cases we have
\begin{align*}
\textup{Err}^o_4
\leq & Q \int_{\supp (T)\setminus \im (F)}\left|\dv_{\vec T} X\right|\, d\|T\|
+ Q\int_{\im (F)\setminus \supp (T)} \left|\dv_{\vec \bT_F} X\right|\, d\|\bT_F\|\\
&\quad + Q \|d\omega\|_\infty \int |X| d\|T-\bT_F\|\, ,
\end{align*}
where we use the convention that $\omega=0$ in the cases (a) and (c). We then can estimate
\begin{align}
\textup{Err}^o_4&\leq C\,\int
\left(\textstyle{\ph'(\p(p)) \,|p-\p(p)| +\ph(\p(p))}\right)
\,d\|T-\bT_F\|\notag\\
& \stackrel{\eqref{e:Ndecay}\&\eqref{e:diff masse}}{\leq} C\,\bLambda^{\eta_0}(r)\,\bD(r) +C\, \bF(r) + C\, r^{1+\gamma_0}\underbrace{\int
|\nabla \ph(\p(p))| \,|p-\p(p)|\,d\|T-\bT_F\|}_{S (\ph)}\,.\label{e:Err4_o}
\end{align}
In case (b) we have that
\begin{align*}
\bT_F(d\omega\ser X) &= \int_{\cM} \ph\,\sum_{i=1}^{Q} \langle (\xi_1+D_{\xi_1} N_i)\wedge(\xi_2+D_{\xi_2} 
N_i\cdot\xi_2)\wedge  N_i\,,\,d\omega(p+ N_i(p))\, .
\end{align*}
Clearly
\begin{align*}
&\Big|\bT_F (d\omega \ser X) - \int_{\cM} \ph\,\sum_{i=1}^{Q} \langle (\xi_1+D_{\xi_1} N_i)\wedge(\xi_2+D_{\xi_2} 
N_i\cdot\xi_2)\wedge  N_i\,,\,d\omega(p)\rangle\Big|\\ 
\leq &\; C \|d\omega\|_1 \int \ph |N|^2 
\end{align*}
and we can therefore conclude
\begin{align*}
&\Big|\bT_F (d\omega \ser X)  - \int_{\cM} \ph\, \sum_{i=1}^{Q}\, \langle \xi_1 (p) \wedge D_{\xi_2} N_i (p) \wedge N_i (p) + D_{\xi_1} N_i (p) \wedge \xi_2 (p) \wedge  N_i (p),d\omega(p)\rangle\Big|\\
\leq & C\|d\omega\|_0 \int \ph |N| |DN|^2 + C \|d\omega\|_0 \int \ph |\etaa\circ N|+C \|d\omega\|_1 \int \ph |N|^2\, . 
\end{align*}
Letting $\phi$ converge to the characteristic function of the interval $[-1,1]$, we reach the conclusion \eqref{e:OV}. The only term which needs some care is the 
term $S (\ph)$ in \eqref{e:Err4_o}. Note that we can approximate the characterstic function of $[-1, 1]$ with an increasing sequence of functions $\phi_j$ with the property that $|\phi'_j|\leq C j$, $0\leq \phi_j \leq 1$ and $\phi_j \equiv 1$ on $[-1+1/j, 1-1/j]$. Then we would have 
\[
\limsup_j S (\ph_j) \leq C \limsup_j \frac{ j}{r} \|T - \bT_F\| (\Psii (B_r\setminus B_{r (1-1/j)})) \leq C \frac{d}{dr} \|T - \bT_F\| (\Psii (B_r))\, ,
\]
by the monotonicity of the function $r\mapsto \|T - \bT_F\| (\Psii (B_r))$. 

In the cases (a) and (c) we follow the same argument, but we need to bound the additional term $\textup{Err}^o_5$. In order to deal with the latter
term we argue as in \cite[Section 4.1]{DS5}. In particular we bound
\begin{align}
\textup{Err}^o_5 &\leq \left|\int \dv_{\vec T} X^\perp\, d\|T\| \right|\notag\\
&\leq \underbrace{\int_{\supp (T)\setminus \im (F)}\left|\dv_{\vec T} X\right|\, d\|T\|
+ \int_{\im (F)\setminus \supp (T)} \left|\dv_{\vec \bT_F} X\right|\, d\|\bT_F\|}_{I_1}\notag\\
&\qquad + \underbrace{\left|\int\langle X^\perp,h(\vec{\bT}_F(p))\rangle\,d\|\bT_F\|\right|}_{I_2}\, ,
\end{align}
where $h (v_1\wedge v_2) := \sum_{i=1}^2 A_\Sigma (v_i, v_i)$. 
Since the projection on the normal to $\Sigma$ is a $C^{2,\eps_0}$ map, $X^\perp$ enjoys the same $C^1$ bounds as $X$ and $I_1$ can be controlled as $\textup{Err}^o_4$. The term $I_2$ can be estimated using
\[
|X^{o\perp} (p)| = \varphi\, |\p_{T_p \Sigma^\perp} (p-\p (p))| \leq C \mathbf{c} (\Sigma)\,\ph\, |p-\p (p)|^2 \leq C 
\bmo^{\sfrac{1}{2}}\,\ph\, |p-\p (p)|^2\quad\forall\;p\in\Sigma.
\]  
In particular we achieve $I_2\leq C \bH (r)$, which concludes the proof. 

\subsection{Proof of Theorem \ref{t:poincare}} In order to prove the theorem we start estimating the error term
$\bF$.

\begin{lemma}\label{l:F} 
There exist a constant $C_{\ref{l:F}}>0$ (depending on $\gamma_0$) such that
\begin{equation}\label{e:F}
\bF(r)\leq C_{\ref{l:F}}\,r^{\gamma_0 -1}\,\bH(r)+C_{\ref{l:F}}\,r^{\gamma_0}\, \bD(r) \quad \forall\;r \in (0,1).
\end{equation}
\end{lemma}

\begin{proof}
Using \eqref{e:H'} and an integration by parts we infer that
\begin{align}\label{e:F1}
\gamma_0\int_0^r\frac{\bH(\rho)}{\rho^{2-\gamma_0}}\,d\rho
&= 
\frac{\bH(\rho)}{\rho^{1-\gamma_0}}\Big\vert_0^r-\int_0^r\frac{d}{d\rho}\left(\frac{\bH(\rho)}{\rho}\right)\,\rho^{
\gamma_0}\,d\rho
=\frac{\bH(r)}{r^{1-\gamma_0}}-\int_0^r\frac{2\,\bE(\rho)}{\rho^{1-\gamma_0}}\,d\rho.
\end{align}
The Cauchy--Schwarz inequality yields then the following bound for every $\eps$:
\begin{align}\label{e:CS per E}
|\bE(r)| &\leq  \frac{\eps}{r} \int_{\de B_r} |\NN|^2  + 
\frac{r}{4\eps} \int_{\de B_r} |D\NN|^2
= \eps\,\frac{\bH(r)}{r} + \frac{r\,\bD'(r)}{4\eps}.
\end{align}
Therefore, by choosing $\eps =\sfrac{\gamma_0}{2}$, we deduce
\eqref{e:F} from \eqref{e:F1} and \eqref{e:CS per E}.
\end{proof}

\begin{proof}[Proof of Theorem~\ref{t:poincare}]
In view of Lemma~\ref{l:F},
for $r$ sufficiently small, the almost minimizing condition \eqref{e:AM1 bis}
reads as
\[
 \bD(r)\leq C\,r\,\bD'(r)+C\,\frac{\bH(r)}{r^{1-\gamma_0}}
+C\,\bmo^{\sfrac12}\,r^{\gamma_0}\,\int_{\de B_r} |\etaa\circ \NN|\,.
\]
Dividing by the radius and integrating we get
\begin{align}\label{e:AM1 tris}
\int_0^r \frac{\bD(s)}{s}\,ds 
&\leq C \, \int_0^r \left( \bD'(\rho) +\,\frac{\bH(\rho)}{\rho^{2-\gamma_0}}+ \, \rho^{\gamma_0-1}\,\int_{\de B_\rho} |\etaa\circ \NN|\right)\,d\rho \notag\\
&\stackrel{\eqref{e:F}}{\leq}  C \,\bD(r) +C\, \bF (r) + C\,\bmo^{\sfrac12}\int_{B_r} \frac{|\etaa\circ \NN|}{|z|^{1-\gamma_0}} \notag\\
&\stackrel{\eqref{e:media_pesata}}{\leq} C\, \bD(r)+C\,(\bLambda^{\eta_0}(r)\, \bD(r) +\bF(r))\leq C\, \bD(r) + C\, r^{\gamma_0-1} \bH (r)\, .
\end{align}
Therefore, using Lemma~\ref{l:H'} we deduce that
\begin{align*}
\frac{\bH(r)}{r} & = \int_0^r \frac{2\,\bE(\rho)}{\rho}\, dt
\stackrel{\eqref{e:OV}}{\leq} C\,\int_0^r \frac{\bD(\rho)}{\rho}\, d\rho  \notag\\
&\qquad +C\,\int_0^r \left( \frac{\bH(\rho)}{\rho^{2-2\gamma_0}}+ \rho^{\gamma_0} \,\bD'(\rho)+ \rho^{\gamma_0} \frac{d}{d\rho}\|T-\bT_F\|(\p^{-1}(\Psii(B_\rho))) \right) \, d\rho\notag\\
& \stackrel{\eqref{e:AM1 tris}}{\leq} C\,\bD(r) + C\,\frac{\bH(r)}{r^{1-\gamma_0}}+C\,r^{\gamma_0}\bD(r)+C\,\bF(r)+C\,r^{\gamma_0}\|T-\bT_F\|(\p^{-1}(\Psii(B_r))) \notag\\
&\stackrel{\eqref{e:diff masse}\&\eqref{e:F}}{\leq} C\,\bD(r) + C\,\frac{\bH(r)}{r^{1-\gamma_0}}\, .
\end{align*}
For $r$ sufficiently small this concludes the proof. 
\end{proof}

\section{Inner variations and key estimates}\label{s:refined}

Using the Poincar\'e inequality in Theorem~\ref{t:poincare},
we can give very simple estimates of the error terms in the ``inner
variations'' of the current $T$. The latter corresponds to deformations of $T$ along appropriate vector fields which are tangent to $\cM$. In order to state our main conclusion we need to introduce yet another quantity
\begin{equation}\label{e:G} 
\bG (r) :=  \int_{\de B_r} \left|D_\nu \NN\right|^2\, .
\end{equation}

\begin{proposition}[Inner Variations]\label{p:IV}
There exist constants $C_{\ref{p:IV}}>0$ and $\eta>0$ such that,
 if $r>0$ is small enough, than the following holds
\begin{equation}\label{e:inn}
\left|\bD' (r)  - 2\,\bG (r)\right| \leq C\, \mathcal{E}_{IV}(r)\,,  
\end{equation}
where
\begin{align}\label{e:stima_Eiv_1}
\mathcal{E}_{IV}(r) 
&= r^{2\eta-1} \bD(r)+
\bD(r)^{\eta}\,\bD'(r)+\frac{\bmo^{\sfrac12}}{r^{1-\gamma_0}} \int_{\de B_r}|\etaa\circ 
\NN(z,w)|\notag\\
&+\frac{d}{dr}\|T-\bT_F\|(\p^{-1}(\Psii(B_r)))\,.
\end{align}
\end{proposition}

For further use we summarize in the next lemma a set of inequalities which will be used in the next sections and which are direct consequences
of all the conclusions derived so far

\begin{lemma}\label{l:tutte le stime}
There exist constant $C_{\ref{l:tutte le stime}}>0$ and $\eta>0$ such that for every $r$ sufficiently small the following holds:
\begin{align}
\bF(r) + r \bF' (r) \leq & C_{\ref{l:tutte le stime}}\,r^{\gamma_0}\bD(r)\label{e:F finale}\\
|\bL(r)|\leq & C_{\ref{l:tutte le stime}} \,r\,\bD(r)\label{e:L finale}\\
 |\bL'(r)| \leq & C_{\ref{l:tutte le stime}}\,\left(\bH(r)\,\bD'(r)\right)^{\sfrac12}\label{e:L' finale}\\
\mathcal{E}_{OV} \leq &C_{\ref{l:tutte le stime}}\, \bD^{1+\eta} (r) + C_{\ref{l:tutte le 
stime}}\,\bF(r)+C_{\ref{l:tutte le stime}} r \bD^\eta (r) \bD' (r) + C_{\ref{l:tutte le stime}}\, r\,\mathcal{E}_{BP}(r),\label{e:OV finale}\\
\mathcal{E}_{IV}(r)\leq & C_{\ref{l:tutte le stime}}\, r^{2\eta-1} \bD(r)+C_{\ref{l:tutte le 
stime}}\,\bD(r)^{\eta}\,\bD'(r)+C_{\ref{l:tutte le stime}}\,\mathcal{E}_{BP}(r),\label{e:IV finale}
\end{align}
where 
\[
\mathcal{E}_{BP}(r):=\frac{\bmo^{\sfrac12}}{r^{1-\gamma_0}} \int_{\de B_r}|\etaa\circ 
\NN|+\frac{d}{dr}\|T-\bT_F\|(\p^{-1}(\Psii(B_r)))
\]
Moreover, for every $a>0$ there exist constants $b_0(a),C(a)>0$ such that
\begin{equation}\label{e:AM finale 1}
\bD(r) \leq \frac{r\,\bD'(r)}{2(2\,a+b)} + \frac{a(a+b)\,\bH(r)}{r(2\,a+b)} +C(a)\,r\, \mathcal{E}_{IV}(r)
\quad \forall \; b<b_0(a).
\end{equation}
\end{lemma}

An important corollary of the previous lemma is the following

\begin{corollary}\label{c:integrability}
There exists a constant $C_{\ref{c:integrability}}>0$ such that, if $\eta$ is the constant of Lemma \ref{l:tutte le stime},
then for every $0 \leq \gamma<\eta$ and $r$ sufficiently small, the nonnegative functions $\frac{\mathcal{E}_{IV}(r)}{r^\gamma\,\bD(r)}$
$\frac{\mathcal{E}_{OV}(r)}{r^{1+\gamma}\bD(r)}$ are both integrable. Moreover, if we define the functions
\begin{align}
\bSigma_{IV} (r) := & \int_0^r \frac{\mathcal{E}_{IV}(s)}{s^\gamma\,\bD(s)}\, ds\, ,\label{e:def_SIV}\\
\bSigma_{OV} (r) := & \int_0^r \frac{\mathcal{E}_{OV}(s)}{s^\gamma\,\bD(s)}\, ds\, ,\label{e:def_SOV}\\
\bSigma (r) := & \bSigma_{IV} (r) + \bSigma_{OV} (r)\, ,
\end{align}
then
\begin{equation}\label{e:bSigma}
\bSigma(r) \leq C_{\ref{c:integrability}}\,r^{\eta-\gamma}\,.
\end{equation}
\end{corollary}

\subsection{Proof of Proposition \ref{p:IV}} We evaluate the first variation of $T$ along a suitably defined vector field $X$. To this aim we fix a function $\phi\in C^\infty_c (]-1,1[)$, symmetric, nonnegative and identically one on $]-1+1/j, 1-1/j[$ and with the property that $|\phi'|\leq C j$. Then we introduce the 
vector field $Y\colon \cM \to \R^{n+2}$ defined, for every $(z,w)\in \gira\setminus \{0\}$, by
\[
Y(\Psii(z,w)) := \textstyle{\frac{|z|}{r}}\,\phi(\textstyle{\frac{|z|}{r}}) \, D_\nu \Psii(z,w) \in T_{\Psii(z,w)} \cM\,,
\]
and extended to be $0$ at the origin.

Next we define the vector field $X_{i}\colon \bV_{a,u} \to \R^{n+2}$ by $X_i (p):= Y(\p(p))$. 
Note that $X_i$ is the infinitesimal generator of a one parameter family of diffeomorphisms $\Phi_\eps$ defined as
$\Phi_\eps (p):= \Gamma_\eps (\p (p)) + p - \p (p)$, where 
$\Gamma_\eps$ is the one-parameter family of biLipschitz homeomorphisms of $\cM$ generated by $Y$. In fact, since $\Gamma_\varepsilon$ fixes the origin, we can consider it as a $C^{2, \gamma_0}$ map of $\cM \setminus \{0\}$ onto itself. 
Note moreover that $X_i$ is Lipschitz on the entire $\gira$. 

Observe that, by Lemma \ref{l:F} and the Poincar\'e inequality, $\bF(r)\leq C\, r^{\gamma_0}\,\bD(r)$, so that $\bLambda(r)\leq C \,\bD(r)$. Moreover, 
\begin{equation}\label{e:DY_bound}
 |D_{\cM} Y|(\Psii(z,w))+|\dv_\cM\,  Y|(\Psii(z,w))\leq - C r^{-2} |z|\,\phi'(\textstyle{\frac{|z|}{r}})+ Cr^{-1} \,\phi(\textstyle{\frac{|z|}{r}}) \,,
\end{equation}
where we recall that $\phi'\leq 0$ on $[0,1]$. 

If $r$ is small enough, by \eqref{e:Lip_N} we can argue as in \cite[Section~3.3]{DS5} and deduce that
\[
 \frac{1}{2}\left|\int_{\cM}\left( |D{N}|^2 \,{\rm div}_{\cM}Y-2\sum_{i=1}^{Q} \langle D{N}_i\colon 
\left(D{N}_i\cdot D_{\cM}Y\right)\rangle\right)\right| \leq  \sum_{k=1}^5 \textup{Err}^i_k\,,
\]
where the error terms can be bounded in the following manner.

First of all,
\begin{align*}
{\rm Err}_1^i &=  Q \left| \int_\cM \big(\langle H_\cM, \etaa\circ {N}\rangle\,\dv_\cM Y+\langle 
D_Y H_\cM,\etaa\circ {N}\rangle \big) \right| \\
&\leq Cr^{-1} \,\bmo^{\sfrac12}\, \int_{\gira} \left(\phi \big( \textstyle{\frac{|z|}{r}}\big)\,|z|^{\gamma_0-1}\,|\etaa\circ 
\NN(z,w)| - \phi' \big( \textstyle{\frac{|z|}{r}}\big)\, |z|^{\gamma_0-1}\,|\etaa\circ 
\NN(z,w)|\right)\\
&\stackrel{\eqref{e:media_pesata}}{\leq} Cr^{-1}\, \bD^{1+\eta}(r) - C\,\bmo^{\sfrac12}\,r^{\gamma_0-1} \int_{B_r} r^{-1} \phi'  \big( \textstyle{\frac{|z|}{r}}\big)\,|\etaa\circ 
\NN(z,w)|\,,
\end{align*}
where in the first inequality we used \eqref{e:DY_bound} and the fact that
\[
\langle D_YH_\cM,\etaa\circ {N}\rangle\leq |Y|\,|D H_{\cM}|\,|\etaa\circ N|\leq C\, \textstyle{\frac{|z|}{r}} \phi ( \textstyle{\frac{|z|}{r}})\,|z|^{\gamma_0-2}\,|\etaa\circ \NN|\,.
\]

As for $\textup{Err}^i_2$ and $\textup{Err}_i^3$ we have
\begin{align*}
{\rm Err}_2^i 
&= C \int_\cM |A_{\cM}|^2\big(|DY|\,|N|^2+|Y|\,|N|\,|DN|\big)\\
&\leq C\, \bmo\, \int_{\gira} \left[r^{-1}\left(- \textstyle{\frac{|z|}{r}}\phi'(\big( \textstyle{\frac{|z|}{r}}\big))+ 
\phi \big( \textstyle{\frac{|z|}{r}}\big)\right) \frac{|\NN|^2}{|z|^{2-2\gamma_0}}+ {\textstyle{\frac{|z|}{r}}}\, \phi\big( {\textstyle{\frac{|z|}{r^2}}}\big)
\frac{|\NN|\,|D\NN|}{|z|^{2-2\gamma_0}}\right]\\
& \leq C \,\bmo\,r^{\gamma_0-1} \,\bD(r)- Cr^{-1} \, \int_{B_r}r^{-1} \phi' \big( {\textstyle{\frac{|z|}{r}}}\big) \,
\frac{|\NN|^2}{|z|^{1-\gamma_0}}\,,
\end{align*}
and
\begin{align*}
{\rm Err}_3^i 
&\leq  C \int_\cM \Big(|Y|\,|A_{\cM}|\,|DN|^2\big(|N| + |DN|\big) +
|DY| \big(|A_{\cM}| \,|DN|\, |N|^2 + |DN|^4\big)\Big)\\
&\leq C r^{\gamma_0-1} \bD(r)- C\,\bD(r)^{\eta}\,\int_{\gira} r^{-1}  \phi'\big( {\textstyle{\frac{|z|}{r}}}\big)\, |D\NN|^2
+C\,r^{-1}\,\bD(r)^{\eta}\,\int_{\gira}r^{-1}  \phi\big( {\textstyle{\frac{|z|}{r}}}\big)\, \frac{|\NN|^2}{|z|^{2-\gamma_0}}\, .
\end{align*}

The errors ${\rm Err}_4^i$ and ${\rm Err}_5^i$ are the same as ${\rm Err}_4^o$ and ${\rm Err}_5^o$ respectively, in Section \ref{ss:OV}, 
evaluated along a different vector field. Proceeding in the same way as in the estimate of ${\rm Err}_4^o$, we deduce
\begin{align*}
{\rm Err}_4^i
&=\int_{\supp (T)\setminus \im (F)}  \left|\dv_{\vec T} X_i\right|\, d\|T\|
+ \int_{\im (F)\setminus \supp (T)} \left|\dv_{\vec \bT_F} X_i \right|\, d\|\bT_F\|\\
&\leq C\, r^{\gamma_0-1}\,\bD(r)+ C \underbrace{\int \alpha\,d\|T-\bT_F\|}_{S(\phi)}\,.
\end{align*}
where $\alpha (p) = \varphi (\p (p))$ and $\varphi (\Psii (z,w)) =r^{-2} |z| \phi \big(r^{-1} |z|) - r^{-1} \phi' (r^{-1} |z|)$. In particular using 
\eqref{e:diff masse} and the fact that $- \phi' \leq C j$ on $[0,1]$, we infer
\[
S (\phi) \leq C r^{\gamma_0-1} \bD (r) + C \frac{j}{r} \|T - \bT_F\| (\p^{-1} (\Psii (B_r\setminus B_{r(1-1/j)}))\, .
\]
As for $\textup{Err}_i^5$, we observe that it only appears in the cases (a) and (c) and arguing as in Section \ref{ss:OV} we can bound it as
\[
\textup{Err}_i^5 \leq I_1 + \underbrace{\left|\int\langle X_i^\perp,h(\vec{\bT}_F(p))\rangle\,d\|\bT_F\|\right|}_{I_2}\, ,
\]
where $h (v_1\wedge v_2) := \sum_{i=1}^2 A_\Sigma (v_i, v_i)$ and $I_1$ enjoys the same bounds as $\textup{Err}_i^4$. 
Following the argument of \cite[Section 4.3]{DS5} we can see that $I_2$ enjoys the same bounds as $\textup{Err}_i^1$ and $\textup{Err}_i^2$.

\medskip

To conclude the proof notice that, with analogous computation as in \cite[Proposition 3.1]{DS1},
\begin{equation}\label{e:dir_var}
 \frac{d}{d\eps}\Big\vert_{\eps=0}\int_{\cM} |D(N\circ \Gamma_\eps)|^2=\int_{\cM}\left(2\sum_{i=1}^{Q} \langle 
DN_i\colon \left(DN_i\cdot D_{\cM}Y\right)\rangle-|DN|^2 \,{\rm div}_{\cM}Y \right)\,.
\end{equation}
However, by the conformal invariance of the Dirichlet energy, we have
\[
\int_{\cM} |D(N\circ \Gamma_\eps)|^2 = \int_\gira |D (\NN \circ \hat{\Gamma}_\eps)|^2\, ,
\]
where $\hat{\Gamma}_\eps$ is the one parameter family of diffeomorphisms generated by the vector field $\hat{Y}\colon 
\gira \to \gira$ defined by
\[
\hat{Y}(z,w):=  \frac{|z|}{r} \,\phi\left(\frac{|z|}{r}\right)\nu\,.
\]
Hence
\begin{equation}\label{e:dir_var1}
 \frac{d}{d\eps}\Big\vert_{\eps=0}\int_{\cM} |D(N\circ 
\Gamma_\eps)|^2= \int_{\gira}\left(2\sum_{i=1}^{Q} \langle D\NN\,_i\colon \left(D\NN\,_i\cdot 
D\hat{Y}\right)\rangle-|D\NN|^2 \,{\rm div}\, \hat{Y} \right)\, ,
\end{equation}
where the differentiation is taken with respect to the (local) flat structure of $\gira$. 

In particular we conclude
\begin{equation}\label{e:dir_var3} 
 \frac{d}{d\eps}\Big\vert_{\eps=0}\int_{\cM} |D(N\circ 
\Gamma_\eps)|^2 =  \int_{B_r} \frac{|z|}{r^2}\,\phi'\left(\frac{|z|}{r}\right) (2 |D_\nu \NN|^2 - |D\NN|^2)\,.
\end{equation}
Collecting together \eqref{e:dir_var}, \eqref{e:dir_var3} and the error estimates, and letting $\phi$ converge to the to the indicator function of $[-1,1]$
(namely letting $j\uparrow \infty$) we conclude the proof. 

\subsection{Proof of Lemma \ref{l:tutte le stime}} The lemma is a very simple corollary of the estimates proven so far.
\eqref{e:F finale} is a simple consequence of 
the Poincar\'e inequality \eqref{e:poincare} and of \eqref{e:F}.
Similarly, by Lemma~\ref{l:F}, we have that 
$\bLambda(r) \leq C\, \bD (r)$, and therefore \eqref{e:OV finale}
follows in view of \eqref{e:F finale}.
The same arguments hold for \eqref{e:IV finale}.
Next for \eqref{e:L finale} we can estimate as follows:
\begin{align}\label{e:L}
|\bL(r)| & \leq C\,\bmo^{\sfrac12}\,\int_{B_r} |\NN|\,|D\NN| \leq  C_{}\,\bmo^{\sfrac12}\,\left(\int_0^r \bH(t)\,dt\right)^{\frac12} 
\bD^{\frac12}(r)\notag \\
&\stackrel{\eqref{t:poincare}}{\leq} C_{}\,\bmo^{\sfrac12}\,\left(C_{\ref{t:poincare}}\int_0^r 
t\,\bD(t)\,dt\right)^{\frac12} \bD^{\frac12}(r)
\leq C\,\bmo^{\sfrac12}\,r\,\bD(r)\, .
\end{align}
Similarly
\begin{align}\label{e:L'}
|\bL'(r)| & \leq C\,\bmo^{\sfrac12}\,\int_{\de B_r} |\NN|\,|D\NN| \leq 
C\,\bmo^{\sfrac12}\,\left(\,\bD'(r)\,\bH(r)\right)^{\frac12}\,.
\end{align}
Finally, we notice that Proposition \ref{p:IV} implies
\[
\left|\frac{\bD'(r)}{2}- \int_{\de B_r} |D_\tau \NN|^2\right| \leq C\, 
\mathcal{E}_{IV}(r).
\]
Therefore, using the almost minimizing property in \eqref{e:AM2 bis} and the Poincar\'e
inequality we infer that
\begin{align*}
\bD(r) & \leq (1+C\,r)\left[\frac{r\,\bD'(r)}{2(2\,a+b)} + \frac{a(a+b)\,\bH(r)}{r(2\,a+b)}\right] +C(a)\,r\,\mathcal{E}_{IV}(r)+ \mathcal{E}_{QM}(r)+C\,r^{1+\sigma}\,\bD'(r)\,.
\end{align*}
Absorbing the error term $r^{1+\sigma}\,\bD'(r)$ and dividing by $(1+C\,r^\sigma)$ we get
\begin{align*}
\bD(r) & \leq \frac{r\,\bD'(r)}{2(2\,a+b)} + \frac{a(a+b)\,\bH(r)}{r(2\,a+b)} +C(a)\,r\,\mathcal{E}_{IV}(r)+ \mathcal{E}_{QM}(r)+C\,r^\sigma \bD(r)\, ,
\end{align*}
from which \eqref{e:AM finale 1} follows straightforwardly by noticing that $\mathcal{E}_{QM}(r)+r\,\bD(r)\leq C\, r\,\mathcal{E}_{IV}(r)$. 

\subsection{Proof of Corollary \ref{c:integrability}} Recall first thet $\eta<\gamma_0$. We start with $\mathcal{E}_{BP}(r)$. Notice that, using $\bH(t)\leq C \,t\, \bD(t)$ together with the definition of $\bF(r)$, we have
\[
\int_0^r\left(\frac{1}{t^{\gamma}\,\bD(t)}\right)' \bF(t)\,dt\leq C\,\frac{\bF(r)}{r^\gamma \,\bD(r)}+C\int_0^r \frac{1}{t^\gamma\,\bD(t)}\,\frac{\bH(t)}{t^{2-\gamma_0}}\,dt\leq C r^{\gamma_0-\gamma}
\]
Next, by a simple integration by parts and the fact that $\bD(r)\leq C r^2$, we deduce
\begin{align}
&\int_0^r\frac{1}{t^{\gamma}\bD(t)} 
\frac{d}{dt}\|T-\bT_F\|(\p^{-1}(\Psii(B_t)))\,dt=\frac{1}{r^{\gamma}\bD(r)} \|T-\bT_F\|(\p^{-1}(\Psii(B_r))) \notag\\
&\qquad\qquad+\int_0^r \left(\frac{1}{t^\gamma \bD(t)}\right)'\|T-\bT_F\|(\p^{-1}(\Psii(B_t)))\,dt\notag\\
\stackrel{\eqref{e:diff masse}}{\leq} & C\, \frac{\bD^{1+\eta}(r)+\bF(r)}{r^\gamma\,\bD(r)}+\int_0^r \left(\frac{1}{t^\gamma \bD(t)}\right)' \, \left(\bD(t)^{1+\eta}+\bF(t)\right)\,dt \leq C \,r^{\eta-\gamma}\,.
\end{align}
In a similar fashion we have
\begin{align}
\int_0^r\frac{\bmo^{\sfrac12}}{t^\gamma\,\bD(t)} 
&\int_{\de B_t}\frac{|\etaa\circ 
\NN(z,w)|}{t^{1-\gamma_0}}\,dt
\leq\frac{\bmo^{\sfrac12}}{r^{\gamma}\,\bD(r)}\int_{B_r}\frac{|\etaa\circ 
\NN(z,w)|}{|z|^{1-\gamma_0}}\notag\\
&+\int_0^r \left(\frac{1}{t^\gamma \bD(t)}\right)'\,\bmo^{\sfrac12}\int_{B_t}\frac{|\etaa\circ 
\NN(z,w)|}{|z|^{1-\gamma_0}}\notag\\
&\stackrel{\eqref{e:media_pesata}}{\leq} C\, \frac{\bD^{1+\eta}(r)+\bF(r)}{r^\gamma\,\bD(r)}+\int_0^r \left(\frac{1}{t^\gamma \bD(t)}\right)' \, \left(\bD(t)^{1+\eta}+\bF(t)\right)\,dt \leq C \,r^{\eta-\gamma}\,.
\end{align}
so that 
\[
\int_0^r\frac{\mathcal{E}_{BP}(t)}{t^{\gamma}\,\bD(t)}\,dt\leq C\, r^{\eta-\gamma}
\]
To conclude, we estimate separately the two nonnegative functions $\bSigma_{IV}$ and $\bSigma_{OV}$. In particular
\begin{align}
\bSigma_{IV} (r) = & \int_{0}^r 
\frac{\mathcal{E}_{IV}(t)}{t^\gamma\,\bD(t)}\,dt \stackrel{\eqref{e:IV finale}}{\leq} 2\,C_{\ref{l:tutte le stime}}\,
\int_0^r \left(t^{\gamma_0-\gamma-1}+t^{-\gamma}\bD(t)^{\eta-1}\,\bD'(t)+\frac{\mathcal{E}_{BP}(t)}{t^{\gamma}\,\bD(t)} \right)\,dt\notag\\
\leq &\, C\,r^{\eta-\gamma} \left(1+\bD(t)^{\sfrac{\eta}{2}}\right) \leq C\,r^{\eta-\gamma}\,,
\end{align}
where in the second inequality we used $\bD(t)\leq C\, t^{2}$. Finally
\begin{align}
\bSigma_{OV} (r) = & \int_{0}^r \frac{\mathcal{E}_{OV}(t)}{t^{1+\gamma}\bD(t)}\,dt \notag\\
\stackrel{\eqref{e:OV finale}}{\leq}
C_{\ref{l:tutte le 
stime}}& \,\int_{0}^r \left(\frac{\bD^{\eta}(t)}{t^{1+\gamma}} + \frac{\bF(t)}{t^{1+\gamma}\bD(t)}+t^{-\gamma} \bD^{\eta-1} (t)\bD' (t) +\frac{\mathcal{E}_{BP}(t)}{t^{\gamma}\,\bD(t)}\right)\,dt\notag\\
\leq&\, C\, r^{\eta-\gamma}\,.
\end{align}

\section{Almost monotonicity and decay of the frequency function}\label{s:decay}

In this section we study the asymptotic behaviour of the normal approximation $\NN$.
The first step consists in proving approximate monotonicity and decay estimates for the frequency function.

For every $r \in (0,1)$ such that $\bH(r)>0$,
we set $\bbI(r) := \frac{r\,\bOmega(r)}{\bH(r)}$ where we recall that
\[
\bOmega(r) := 
\left\{
\begin{array}{ll}
\bD(r) \quad & \text{in the cases (a) and (b) of Definition \ref{d:semicalibrated};}\\
 \bD(r) + \bL(r) \quad &\text{in case (c).}
\end{array}
\right.
\]
Furthermore we define $\bar \bK(r):=\bar \bI(r)^{-1}$ whenever $\bOmega (r)\neq 0$. By \eqref{e:L finale} there exists $r_0>0$ such that
\begin{equation}\label{e:r_0}
\frac{1}{2} \bD (r) \leq (1-C\,r)\bD(r)\leq  \bOmega (r)\leq (1+C\,r) \bD(r)\leq 2\,\bD(r) \qquad \forall r\leq  r_0\, .
\end{equation} 
Having fixed $r_0$, $\bar \bK (r)$ is well defined whenever $\bD (r) >0$ and hence, by the Poincar\'e inequality, whenever $\bbI (r)$ is defined. Moreover, if for some $\rho\leq r_0$ $\bar \bK (\rho)$ is not well defined, that is $\bOmega (\rho) =0$, then obviously $\bOmega (r) = \bD (r) = 0$ for every $r\leq \rho$. 

We are now ready to state the first important monotonicity estimate. From now on we assume of having fixed a $\gamma$ 

\begin{theorem}\label{t:monot}
There exists a constant $C_{\ref{t:monot}}>0$ with the following property: if $\bD (r) > 0$ for some $r\leq r_0$, 
then (setting $\gamma =0$ in \eqref{e:def_SIV} and \eqref{e:def_SOV}) the function 
\begin{equation}\label{e:monot}
\bar \bK(r)\,\exp ( -4r -4 \bSigma_{IV}(r))-4\,\bSigma_{OV}(r)
\end{equation}
is monotone non-increasing on any interval $[a,b]$ where $\bD$ is nowhere $0$. 
In particular, either there is $\bar{r} >0$ such that $\bD (\bar{r})=0$ or
$\bar \bK$ is well-defined on $]0, r_0[$ and the limit $K_0:=\lim_{r\to 0}\bar \bK(r)$ exists.
\end{theorem}

A fundamental consequence of Theorem~\ref{t:monot} is the following dichotomy.

\begin{corollary}\label{c:bdd_freq} There exists $\bar{r}>0$ such that
\begin{itemize}
\item[(A)] either $\bar \bK(r)$ is well-defined for every $r \in ]0,r_0[$, the limit
\begin{equation}\label{e:K_0}
K_0:=\lim_{r\downarrow0} \bar \bK(r)
\end{equation}
is positive and thus there is a constant $C$ and a radius $\bar{r}$ such that
\begin{equation}\label{e:H=rD}
C^{-1}\,r\,\bD(r) \leq \bH(r) \leq C\,r\,\bD(r) \qquad
\forall\; r \in ]0, \bar{r}[\, ;
\end{equation}
\item[(B)] or $T\res \p^{-1} (\Psii (B_{\bar{r}})) = Q \a{\Psii (B_{\bar r})}$ for some positive $\bar{r}$.
\end{itemize}
\end{corollary}

In turn, using the above dichotomy we will show 

\begin{theorem}\label{t:decay}
Assume that condition (i) in Theorem \ref{t:bu} fails. Then the frequency $\bbI(r)$ is well-defined for every sufficiently small $r$ and its limit $I_0= \lim_{r\to 0} \bbI (r) = K_0^{-1}$ exists and it is finite and positive.
Moreover there exist constants $\lambda, C_{\ref{t:decay}}, H_0, D_0>0$
such that, for every $r$ sufficiently small the following holds:
\begin{gather}\label{e:decay}
\big\vert \bI(r) - I_0 \big\vert
+\left\vert \frac{\bH(r)}{r^{2I_0+1}} - H_0 \right\vert+
\left\vert \frac{\bD(r)}{r^{2I_0}} - D_0 \right\vert
\leq C_{\ref{t:decay}}\,r^{\lambda}\,.
\end{gather}
\end{theorem}

\subsection{Proof of Theorem \ref{t:monot}} In the first step we claim the monotonicity of the function $\bar \bK(r)\,\exp ( -\bSigma_{IV}(r))-2\,\bSigma_{OV}(r)$ on any interval contained in $[a,b]$ on which $\bD$ is everywhere positive. 
Recalling that $\bOmega$ and $\bH$ are absolutely continuous functions,
we can compute the following derivative: for every $r \in [a,b]$
\begin{align}\label{e:I'1}
\bar \bK'(r)& =  \left(\frac{\bH(r)}{r}\right)'\,\frac{1}{\bOmega(r)}-\frac{\bH(r)}{r}\,\frac{\bOmega'(r)}{\bOmega^2(r)}\notag\\
&\stackrel{\eqref{e:H'}}{\leq}\frac{1}{r\bOmega^2(r)}\big(2\,\bE(r)\,\bOmega(r)-\bD'(r)\,\bH(r)+|\bL'(r)|\,\bH(r)\big)\,.
\end{align}
Then, either $\bar \bK'\leq 0$, or the RHS of the inequality above is positive, that is
\[
 \bD'(r)\,\bH(r)\leq 2\bE(r)\,\bOmega(r)+|\bL'(r)|\,\bH(r)\stackrel{\eqref{e:L' finale}}{\leq}2\bE(r)\,\bOmega(r)+r\,\bD'(r)\,\bH(r)+\frac{\bH^2(r)}{r}\, .
\]
In turn, using $\bH(r)\leq C\,r\,\bD(r)\leq C\,r\,\bOmega(r)$, the latter inequality implies
\begin{equation}\label{e:controllo_D'H}
\bD'(r)\,\bH(r)\leq C\,\bE(r)\,\bOmega(r)+C\,r\,\bOmega^2(r)\,.
\end{equation}
 From this we deduce
\[
\bE^2(r)\leq \bH(r)\,\bG(r)\leq \bH(r)\,\bD'(r)\leq C \bOmega^2(r)+\frac{\bE^2(r)}{2}\,
\]
which implies that $\bE(r)\leq C\bOmega(r)$ and so, by \eqref{e:L' finale},
\begin{equation}\label{e:fanghualin}
|\bL'(r)|\leq C \,\bmo^{\sfrac12}\,(\bD'(r)\,\bH(r))^{\sfrac12}\leq C\,\bmo^{\sfrac12}\,\bOmega(r)\,.
\end{equation}
Next using again the Cauchy-Schwarz inequality and \eqref{e:OV}, we have
\begin{align*}
\bOmega(r)\,\bE(r)
&\leq \,\bOmega(r)\,\bH(r)^{\sfrac12} \,\bG(r)^{\sfrac12}
\leq\, \frac{\bOmega(r)^2}{2}+\frac{\bH(r)\,\bG(r)}{2}\\
&\leq  \frac{\bOmega(r)\bE(r)}{2}+\frac{\bOmega(r)\,\mathcal{E}_{OV}(r)}{2}+ \frac{\bH(r)\,\bG(r)}{2}\,,
\end{align*}
which implies
\begin{equation}\label{e:mon1}
\bOmega(r)\,\bE(r) \leq \bH(r)\,\bG(r)+\bOmega(r)\,\mathcal{E}_{OV}(r)\,.
\end{equation}
Collecting all these estimates together and using \eqref{e:inn}, we conclude that, if $\bar \bK' (r)\geq 0$, then
\begin{align}
&\bar \bK'(r) \stackrel{\eqref{e:I'1}\&\eqref{e:mon1}}{\leq} \frac{1}{r\bOmega^2(r)}\big(2\,\bH(r)\,\bG(r)-\bD'(r)\,\bH(r)+|\bL'(r)|\,\bH(r)+2\,\bOmega(r)\,\mathcal{E}_{OV}(r)\big)\notag\\
\stackrel{\eqref{e:inn}\&\eqref{e:fanghualin}}{\leq} &\frac{1}{r\bOmega^2(r)}\big(2\,\bH(r)\,\bG(r)-2\,\bH(r)\,\bG(r)+\bOmega(r)\,\bH(r)+\bH(r)\,\mathcal{E}_{IV}(r)+2\,\bOmega(r)\,\mathcal{E}_{OV}(r)\big)\notag\\
\leq &\, 2\frac{\mathcal{E}_{OV}(r)}{r\bOmega(r)}+\bar \bK(r)\, \left(1+\frac{\mathcal{E}_{IV}(r)}{\bOmega(r)}\right) \leq 4\,\frac{\mathcal{E}_{OV}(r)}{r\bD(r)}+4\,\bar \bK(r)\, \left(1+\frac{\mathcal{E}_{IV}(r)}{\bD(r)}\right)\, .
\end{align}
On the other hand the final inequality
\[
\bar \bK' (r) \leq 4 \frac{\mathcal{E}_{OV}(r)}{r\bD(r)}+4\,\bar \bK(r)\, \left(1+\frac{\mathcal{E}_{IV}(r)}{\bD(r)}\right)
\]
is certainly correct when $\bar \bK' (r)\leq 0$, because the right hand side is positive. The monotonicity of the function in \eqref{e:monot} is then obvious. 

Next, as already observed, either $\bD$ is always positive, or it vanishes on some interval $]0, \bar{r}[$. If $\bD$ is always positive, then $\bar \bK$ is well defined on $]0, r_0[$ and the existence of the limit $K_0 := \lim_{r\downarrow 0} \bar \bK (r)$ is a direct consequence of \eqref{e:monot} and Corollary \ref{c:integrability}.

\subsection{Proof of Corollary \ref{c:bdd_freq}} First of all observe that, if $\bD (\bar r)$ vanishes, then $\NN \equiv Q\a{0}$ on $B_{\bar r}$. In particular by \eqref{e:diff masse} we conclude that we are in the alternative (B). We can thus assume, without loss of generality, that $\bD$ is positive on $]0, r_0[$. Assuming that $K_0$ vanishes we will then reach a contradiction.

Under the assumption $K_0=0$, consider the monotonicity of $\bar \bK(r)\,\exp ( -4 \bSigma_{IV}(r)) - 4\,\bSigma_{OV}(r)$ between two radii $0<s<r$ and let $s\to 0$ to get
\[
\bar \bK(r)\leq 4\,e^{4r + 4 \bSigma_{IV}(r)}\,\bSigma_{OV}(r)\leq C\,\bSigma_{OV}(r) \,,
\]
where the last inequality holds for $r$ sufficiently small, since $\bSigma_{IV}(r)\leq C r^{\eta}$ (recall that we have set $\gamma=0$).
Next observe that, since the function $\bSigma_{OV}(r)$ is non-decreasing,
\begin{align}\label{e:err1_K}
\frac{\bF(r)}{\bD(r)}
&\leq \frac{1}{\bD(r)}\int_0^r\frac{\bH(s)}{s^{2-\gamma_0}}\frac{\bD(s)}{\bD(s)}\,ds\leq C\,\int_0^r\frac{\bar \bK(s)}{s^{1-\gamma_0}}\,ds\leq C \,r^{\gamma_0}\,\bSigma_{OV}(r)\, .
\end{align}
Moreover, integrating by parts:
\begin{align}\label{e:err2_K}
&\int_0^r \frac{1}{\bD(s)}\frac{d}{ds}\|T-\bT_F\|(\p^{-1} (\Psii (B_s)))\,ds\notag\\
\stackrel{\eqref{e:diff masse}}{\leq} &C \frac{\bD^{1+\eta}(r)+\bF(r)}{\bD(r)}+C\int_0^r\left(\frac{1}{\bD(s)}\right)'\left(\bD^{1+\eta}(s)+\bF(s)\right)\,ds\notag\\
\leq & C\,\bD^\eta(r)+C\, r^{\gamma_0} \bSigma_{OV}(r)+C\,\frac{\bF(r)}{\bD(r)}+C\,\int_0^r \frac{\bF'(s)}{\bD(s)}\,ds\notag\\
\leq & C\,\bD^\eta(r)+C\, r^{\gamma_0} \bSigma_{OV}(r)+C\,\int_0^r \frac{\bar \bK(s)}{s^{1-{\gamma_0}}}\,ds
\leq  C\,\bD^\eta(r)+C\, r^{\gamma_0} \bSigma_{OV}(r)\, ,
\end{align}
where we have used repeatedly \eqref{e:r_0}.

Using the latter in the formula for $\mathcal{E}_{OV}$ we also conclude
\begin{align*}
&\;\bSigma_{OV}(r)\\
\leq &\; C\int_0^r\frac{1}{s\bD(s)}\left( \bD(s)^{1+\eta}+s \bD^\eta (s) \bD' (s) + \bF(s)+s  \frac{d}{ds}\|T-\bT_F\|(\p^{-1} (\Psii (B_s)))  \right)\,ds\\
\leq &\; C\,r^{\eta} \bD(r)^{\sfrac{\eta}{2}}+Cr^{\gamma_0}\,\bSigma_{OV}(r)\, .
\end{align*} 
Hence, for $r$ sufficiently small,
\begin{equation}\label{e:K_decay}
\bar \bK(r)\leq C \bSigma_{OV}(r)\leq C \bD(r)^{\sfrac{\eta}{2}}\,.
\end{equation}
In particular this implies that
\begin{equation}\label{e:H_decay}
\bH(r)\leq C \,r\,\bD(r)^{1+\sfrac{\eta}{2}}\,.
\end{equation}
Combining this with \eqref{e:OV} and the Cauchy-Schwarz inequality, we deduce
\begin{align*}
\frac{1}{2} \,\bD(r) 
&\leq \bOmega(r) \leq \frac{\bE(r)}{r}+\mathcal{E}_{OV}(r)
\leq \left(\frac{\bH(r)}{r\,\bD(r)^{\sfrac{\eta}{4}}}\right)^{\sfrac12}\,\left(r\,\bD'(r)\,\bD(r)^{\sfrac{\eta}{4}}\right)^{\sfrac12}+\mathcal{E}_{OV}(r)\\
&\stackrel{\eqref{e:H_decay}}{\leq} C \, \bD(r)^{1+\sfrac{\eta}{4}}+C\, r\, \bD(r)^{\sfrac{\eta}{4}} \bD'(r)+\mathcal{E}_{OV}(r)\,. 
\end{align*}
Dividing the expression above by $r\bD(r)$, integrating between two radii $0<s<r$ and using the bound $\bD(r)\leq\,C\, r^2$ we obtain
\begin{align*}
\log\left(\frac{r}{s}\right) 
\leq C  \int_s^r \left( \frac{\bD(\rho)^{\sfrac{\eta}{4}}}{\rho}+\bD(\rho)^{\sfrac{\eta}{4}-1}\, \bD'(\rho) +\frac{\mathcal{E}_{OV}(\rho)}{\rho\,\bD(\rho)}\right)\,d\rho \leq C\, r^{\sfrac{\eta}{2}}\, .
\end{align*}
Sending $s\to 0$ we get a contradiction. 

\subsection{Proof of Theorem \ref{t:decay}} Clearly, if (i) in Theorem \ref{t:bu} does not hold, then $\bD$ is always positive and we are in alternative (A) of Corollary \ref{c:bdd_freq}. Thus $K_0$ is positive and the first statement is obvious.

Let $\bK(r):=\bI(r)^{-1}$ and observe that by \eqref{e:r_0} we have
\[
(1-C\,r)\bI(r)\leq \bbI(r)\leq (1+C\,r)\bI(r)\,,\quad\forall0\leq r\leq r_0\,, 
\]
which implies
\[
(1-C\,r)\bar \bK(r)\leq \bK(r)\leq (1+C\,r)\bar \bK(r)\quad \forall \,0\leq r\leq r_0\,,
\]
so that in particular $\bK(r)\leq C\, \bar\bK(r)<\infty$ for every $0<r<r_0$ and $\bK(r)\to K_0$ as $r\to 0$.
Using the monotonicity formula of Theorem \ref{t:monot} together with Corollary \ref{c:integrability} we have 
\begin{equation}\label{e:K_alto}
\bar \bK(r)-K_0\leq K_0 (\exp (4r + 4 \bSigma_{OV} (r)) -1) + 4 \bSigma_{IV} (r) \exp (4r + 4 \bSigma_{OV} (r))\leq
C r^\eta\, .
\end{equation}
Therefore
\begin{equation}\label{e:uniq_ineq_1}
\bK(r)-K_0\leq C\,r^{\eta}+C\, \bK(r) \,r\leq C\,r^{\eta}\,.
\end{equation}
To control $\bK (r) - K_0$ from below we apply \eqref{e:AM finale 1} with $a=I_0=\frac{1}{K_0}$ and $b=\lambda \leq \min\{\sfrac{\eta}{2}, b_0(I_0)\}$ to infer, after dividing by $r \bD(r)$, that
\[
-\frac{\bD'(r)}{\bD(r)}\leq \frac{2}{r} \left(I_0(I_0+\lambda) \bK(r) - (2I_0+\lambda)\right)\,.
\]
Multiplying this expression by $\bK(r)>0$ and adding $\sfrac2r$, we get
\begin{align}\label{e:epip}
\frac{2}{r}-\frac{\bD'(r)}{\bD(r)}\bK(r)
&\leq \frac{2}{r}\left[ 1+I_0(I_0+\lambda) \bK^2(r) - (2I_0+\lambda)\bK(r)\right]+ \frac{C\,\mathcal{E}_{IV}(r)}{\bD(r)}\notag\\
&\leq \frac{2}{r}\,I_0\left( \bK(r)-\frac{1}{I_0}  \right)\,\left((I_0+\lambda) \bK(r)-1\right)+ \frac{C\,\mathcal{E}_{IV}(r)}{\bD(r)}
\end{align}
Since $(I_0+\lambda)\bK(r)$ converges to $1+\lambda K_0$, we easily deduce that for $r$ small enough $(I_0+\lambda)\bK(r)-1 \geq \frac{\lambda}{2} K_0$. Using this together with \eqref{e:uniq_ineq_1}, we deduce from \eqref{e:epip} that
\begin{equation}\label{e:epip2}
\frac{2}{r}-\frac{\bD'(r)}{\bD(r)}\leq \,\frac{\lambda}{r}\, \left(\bK(r)-\frac{1}{I_0}  \right)+ \frac{C\,\mathcal{E}_{IV}(r)}{\bD(r)}+C\frac{r^{\eta}}{r}.
\end{equation}
We next compute
\begin{align}\label{e:I' bis}
\bK'(r)
&=\left(\frac{\bH(r)}{r}\right)'\frac{1}{\bD(r)}-\frac{\bH(r)}{r\bD(r)}\,\frac{\bD'(r)}{\bD(r)}
\stackrel{\eqref{e:H'}\&\eqref{e:r_0}}{\leq} \frac{2\,\bE(r)}{r\,\bD(r)}-\frac{\bD'(r)}{\bD(r)} \bK(r)\notag\\
&\stackrel{\eqref{e:OV}\&\eqref{e:r_0}}{\leq} \frac{2}{r} + C -\frac{\bD'(r)}{\bD(r)} \bK(r)+C\,\frac{\mathcal{E}_{OV}(r)}{r\bD(r)}\notag\\
&\stackrel{\eqref{e:epip2}}{\leq} \,\frac{\lambda}{r}\, \left(\bK(r)-\frac{1}{I_0}  \right)+ \frac{C\,\mathcal{E}_{IV}(r)}{\bD(r)}+C\,\frac{\mathcal{E}_{OV}(r)}{r\bD(r)}+C\frac{r^{\eta}}{r} \,.
\end{align}
Recalling that $ \bK(r)\leq C$, we deduce
\begin{align}\label{e:I' differenziale 3}
\frac{d}{dr}\left[\frac{ \bK(r)-K_0}{r^{\lambda}}\right] \leq C\,\frac{\mathcal{E}_{OV}(r)}{r^{1+\lambda}\bD(r)}+C\, \frac{\mathcal{E}_{IV}(r)}{r^{\lambda}\,\bD(r)}+C\frac{1}{r^{1+\lambda-\eta}}\, .
\end{align}
Integrating \eqref{e:I' differenziale 3}  on the interval $]s,r[$ and using \eqref{e:bSigma}, we get
\[
\bK(r) - K_0 \leq \frac{r^\lambda}{s^\lambda}\left(\bK(s)-K_0\right) + C\,r^{\eta-\lambda}
\]
that is $\bK(s)-K_0\geq - C s^{\lambda}$.
The inequality $|\bK(r)-K_0|\leq C \,r^\lambda$ easily implies $|\bI(r)-I_0|\leq C \,r^\lambda$.

For what concerns the other inequalities we compute
\begin{align}\label{e:log H' 1}
\left[\log\left(\frac{\bH(r)}{r^{2I_0+1}}\right)\right]' & = \frac{\bH'(r)}{\bH(r)} - \frac{2\,I_0+1}{r}=
\frac{2\,\bE(r)}{r\,\bH(r)} - \frac{2\,I_0}{r} \leq \frac{2\,\bD(r)}{\bH(r)} - \frac{2\,I_0}{r} + C\frac{\mathcal{E}_{OV}(r)}{\bH(r)} \notag\\
&= \frac{2}{r}\left(\bI(r)-I_0\right) + C\, 
\frac{\mathcal{E}_{OV}(r)}{\bH(r)} \, ,
\end{align}
and similarly 
\begin{equation}\label{e:log H'2}
\left[\log\left(\frac{\bH(r)}{r^{2I_0+1}}\right)\right]' \geq \frac{2}{r}\left(\bI(r)-I_0\right) - C\, 
\frac{\mathcal{E}_{OV}(r)}{\bH(r)}\, .
\end{equation}
Using that $|\bI (r) - I_0|\leq C r^\lambda$, for $r$ small enough we have the bound $r \bD (r)\leq 2 I_0 \bH (r)$. 
Hence we can use \eqref{e:bSigma} in the integrals of \eqref{e:log H' 1} and \eqref{e:log H'2} to deduce the existence of the limit
\[
H_0:=\lim_{s\downarrow0}\frac{\bH(s)}{s^{2I_0+1}},
\quad\text{with}\quad
\left|\frac{\bH(r)}{r^{2I_0+1}}-H_0\right|\leq C\,r^\lambda.
\]
Moreover, from \eqref{e:log H' 1} we also infer that for $r$ sufficiently small
\[
H_0 \geq \frac{\bH(r)}{r^{2I_0+1}} 
e^{-C\,r^{\lambda}} >0.
\]
Finally the last assertion follows simply setting $D_0:= I_0 \cdot H_0$ and from 
\begin{align*}
\left\vert\frac{\bD(r)}{r^{2\,I_0}} - D_0\right\vert &= \left\vert\bI(r) \, \frac{\bH(r)}{r^{2\,I_0+1}} - I_0\, 
H_0\right\vert\notag\\
& \leq \left\vert\bI(r) - I_0\right\vert\, \frac{\bH(r)}{r^{2\,I_0+1}}
+ I_0\,\left\vert\frac{\bH(r)}{r^{2\,I_0+1}} - H_0\right\vert \leq C\,r^\lambda.\qedhere
\end{align*}

\section{Blow-up and proof of Theorem \ref{t:bu}}\label{s:bu}

As a consequence of the decay estimate in Theorem~\ref{t:decay}
we can show that suitable rescaling of the normal approximation $N$
converge to a unique limiting profile.
To this aim we consider for every $r \in (0,1)$ the functions $f_r : \de B_1 \to \I{Q_1}(\R^{2+n})$
given by 
\[
f_r(z,w) := \frac{\NN(i_r(z,w))}{r^{I_0}},
\]
where we recall that $i_r(z,w)=\left(rz, r^{\sfrac{1}{\bar Q}}w\right)$.
We recall also that $T_0\cM = \R^2 \times \{0\}$, and $T_0\Sigma = \R^{2} \times \R^{\bar n} \times \{0\}$. In the following, with a slight abuse of notation, we write $\R^{\bar n}$ for the subspace 
$\{0\} \times \R^{\bar n} \times \{0\}$.

The final step in the proof of Theorem \ref{t:bu} is then the following proposition.

\begin{proposition}\label{p:unique limit}
Assume alternative (i) in Theorem \ref{t:bu} fails and let $I_0$ and $\lambda$ be the positive numbers of Theorem \ref{t:decay}.
Then $I_0>1$ and there exists a function $f_0:\de B_1 \to \I{Q}(\R^{\bar n})$
such that
\begin{itemize}
\item[(i)] $\etaa \circ f_0 =0$ and $f_0 \not\equiv Q_1\a{0}$;
\item[(ii)] for every $r$ sufficiently small
\begin{equation}\label{e:unique limit}
\cG(f_r(z,w),f_0(z,w))\leq C\, r^{\sfrac{\lambda}{16}}\,\quad \forall\;(z,w)\in \de B_1\,;
\end{equation}
\item[(iii)] the $I_0$-homogeneous extension $g (z,w):=|z|^{I_0}f_0\left(\textstyle{\frac{z}{|z|},\frac{w}{|w|}}\right)$ is nontrivial and $\D$-minimizing.
\end{itemize}
In particular, by \textup{(iii)} $\im(g) \setminus \{0\} \subset\R^{2+n}$ is a real analytic submanifold.
\end{proposition}

Theorem \ref{t:bu} follows immediately from Proposition \ref{p:unique limit} and Theorem \ref{t:decay}.

\begin{proof}[Proof of Theorem \ref{t:bu}]
Since we have identitified $\R^{\bar n}$ with $\{0\}\times \R^{\bar n}\times \{0\}$, it is obvious that the map $g$ has all the properties claimed in (ii), namely it is $\D$-minimizing, $\etaa\circ g \equiv 0$ and it is nontrivial. \eqref{e:uniform convergence} is a corollary of \eqref{e:unique limit} provided $a_0 \leq \frac{\lambda}{16}$. Next note that \eqref{e:Poincare0} has been shown in Theorem \ref{t:poincare}. As for \eqref{e:energia N} observe that, if $4\rho\leq r < 1$, then,  by Theorem \ref{t:decay},
\[
D_0 (r- 2\rho)^{2I_0} - C (r-2\rho)^{2I_0 + \lambda} \leq \bD (r - 2\rho) \leq \bD (r+2\rho) \leq D_0 (r+2\rho)^{2 I_0} + C (r+2\rho)^{2I_0 + \lambda}\, .
\]
Since $2 I_0 >2$, \eqref{e:energia N} follows easily from 
\[
\int_{B_{r+2\rho}\setminus B_{r-2\rho}} |D\NN|^2 = \bD (r+2\rho) - \bD (r-2\rho)\, ,
\]
provided $a_0 \leq \lambda$.
\end{proof}

The rest of this final section of the note is devoted to the proof Proposition \ref{p:unique limit}, which is split in several steps. Before starting with it, let us however observe that the conclusion $I_0 >1$ is an obvious consequence of the decay estimates of Theorem \ref{t:decay} and the fact that $\bD (r) \leq C r^{2+2\gamma_0}$. 

\subsection{Step 1: uniqueness of the limit $f_0$} For $r$ sufficienly small and $s \in [\frac{r}{2}, r]$, we
start estimating the following quantity:
\begin{align}\label{e:stima differenza 1}
\int_{\de B_1}\cG(f_r, f_s)^2  \leq 
(r-s) \int_{\de B_1}\int_s^r\left\vert \frac{d}{dt}f_t(z,w)\right\vert^2\,dt\, .
\end{align}
Using the differentiability properties of Lipschitz multiple valued functions
and the $1$-dimensional theory in \cite[Section~1.1.2]{DS1} (note that $t\mapsto \NN(i_t(z,w))$ is a Lipschitz map),
we easily infer that
\begin{align*}
\left\vert \frac{d}{dt}f_t(z,w)\right\vert^2 & =
\sum_{j=1}^{Q}
\left\vert
\frac{D\NN\,_j(i_t(z,w))\cdot z}{t^{I_0}} - I_0\,\frac{\NN\,_j(i_t(z,w))}{t^{I_0+1}}
\right\vert^2\notag\\
& = \frac{|z|^2|\de_{\hat r}\NN|^2(i_t(z,w))}{t^{2I_0}} - 
2\,I_0\,\frac{|z|}{t^{2I_0+1}}\,\sum_{j=1}^{Q}\langle\de_{\hat r}\NN\,_j, \NN\,_j \rangle(i_t(z,w)) +
\frac{|\NN|^2(i_t(z,w))}{t^{2I_0+2}}.
\end{align*}
Therefore, by the change of variable $(z',w') = i_t(z,w)$ in \eqref{e:stima differenza 1} we infer that
\begin{align*}
\int_{\de B_1}&\cG(f_r, f_s)^2 \leq
\frac{r}{2}\int_{\sfrac{r}{2}}^r 
\left(\frac{\bG(t)}{t^{2I_0+1}} - 2\,I_0 \frac{\bE(t)}{t^{2I_0+2}} + I_0^2\,\frac{\bH(t)}{t^{2I_0+3}}\right)\,dt\notag\\
&\leq \frac{r}{2} \int_{\sfrac{r}{2}}^r 
\left(\frac{\bD'(t)}{2t^{2I_0+1}} - 2\,I_0 \frac{\bD(t)}{t^{2I_0+2}} + I_0^2\,\frac{\bH(t)}{t^{2I_0+3}}
+C\,\frac{\mathcal{E}_{IV}(t)}{t^{2I_0+1}} + C\,\frac{\mathcal{E}_{OV}(t)}{t^{2I_0+2}} 
\right)\,dt\notag\\
&= \frac{r}{2} \int_{\sfrac{r}{2}}^r
\left[\frac{1}{2t}\left(\frac{\bD(t)}{t^{2I_0}}\right)' + I_0 \frac{\bH(t)}{t^{2I_0+3}}\,\left(I_0-\bI(t)\right)
+C\,\frac{\mathcal{E}_{IV}(t)}{t^{2I_0+1}} + C\,\frac{\mathcal{E}_{OV}(t)}{t^{2I_0+2}}
\right]\,dt.
\end{align*}
Using Theorem~\ref{t:decay}, we can then conclude that
\begin{align}\label{e:stima differenza 3}
\int_{\de B_1}\cG(f_r, f_s)^2\,
&\leq C\,\left\vert\frac{\bD(r)}{r^{2I_0}}
- \frac{\bD\left(\frac{r}{2}\right)}{\left(\frac{r}{2}\right)^{2I_0}}\right\vert
+C\,\int_{\sfrac{r}{2}}^r\left[\frac{|I_0-\bI(t)|}{t}
+C\,\frac{\mathcal{E}_{IV}(t)}{\bD(t)} + C\,\frac{\mathcal{E}_{OV}(t)}{t\,\bD(t)} 
\right]\,dt\notag\\
&\leq C\, r^{\lambda}.
\end{align}
By an elementary dyadic argument analogous to that of \cite[Theorem 5.3]{DS1}, we then infer
the existence of $f_0:\de B_1 \to \I{Q}(\R^{2+n})$
such that, for $r$ sufficiently small,
\begin{equation}\label{e:unique limit 2}
\|\cG(f_r, f_0)\|_{L^2(\de B_1)}^2\leq C\,  r^{\lambda}.
\end{equation}

\subsection{Step 2: uniform convergence} Set next $h(z,w):=\cG\left(\frac{\NN(z,w)}{|z|^{I_0}},\frac{\NN(i_{1/2}(z,w))}{|\sfrac z2|^{I_0}}\right)$. It follows from \eqref{e:stima differenza 3} that for $r$ sufficiently small
\begin{equation}\label{e:diff_height}
\int_{B_r} h^2\leq \int_0^r\,\int_{\de B_1} \cG(f_t,f_{\sfrac t2})^2\,t\,dt\stackrel{\eqref{e:stima differenza 3}}{\leq} C \,r^{2+\lambda}\,,
\end{equation}
and from \eqref{e:Ndecay} and \eqref{e:Lip_N}
\begin{equation}\label{e:lip_diff}
\Lip(h|_{B_1\setminus B_s})\leq C\, s^{-I_0}.
\end{equation}
Moreover, for every $\rho < \sfrac{|z|}{4}$ we claim the estimate
\begin{gather}\label{e:en_diff}
\int_{B_{\rho}(z,w)}|Dh|^2\leq C\, \rho + C\,|z|^{\lambda}\, .
\end{gather}
Indeed $|Dh| \leq C\, \left|D\left(\frac{\NN}{|z|^{I_0}}\right)\right|$
and by Theorem~\ref{t:decay}
\begin{align*}
\int_{B_{\rho}(z,w)}\left|D\left(\frac{\NN}{|z|^{I_0}}\right)\right|^2
& \leq 2\, \int_{|z|-\rho}^{|z|+\rho} \int_{\de B_t} \left(\frac{|D\NN|^2}{t^{2I_0}}+I_0^2 \, \frac{|\NN|^2}{t^{2I_0+2}} \right)\,dt\\
&\leq 2\int_{|z|-\rho}^{|z|+\rho} \left(\left(\frac{\bD(t)}{t^{2I_0}}\right)'+2\,I_0\frac{\bD(t)}{t^{2I_0+1}}+I_0^2\, \frac{\bH(t)}{t^{2I_0+2}}\right)\,dt\\
&\leq C\, \left(|z|+\rho \right)^\lambda + C \log\left( \frac{|z|+\rho}{|z|-\rho} \right)\leq C \,|z|^\lambda + C\,\frac{\rho}{|z|}.
\end{align*}
In particular, applying \eqref{e:diff_height}, \eqref{e:lip_diff} and \eqref{e:en_diff} with $\rho = |z|^{1+\sfrac{\lambda}{4}}$,
we infer that for every point $p = (z,w)\in \gira_{\bar{Q}}$ with $|z|$ sufficiently small
\begin{align}
h(p) & \leq \Bigg\vert h(p) - \mint_{B_{\frac{|z|^{1+\sfrac{\lambda}{4}}}{2^k}} (p)} h\Bigg\vert
+ \sum_{i=0}^{k-1}\Bigg\vert \mint_{B_{\frac{|z|^{1+\sfrac{\lambda}{4}}}{2^i}} (p)} h- \mint_{B_{\frac{|z|^{1+\sfrac{\lambda}{4}}}{2^{i+1}}} (p)} h\Bigg\vert + \mint_{B_{|z|^{1+\sfrac{\lambda}{4}}}(p)} h\notag\\
&\leq \Lip(h|_{B_1 (p) \setminus B_{\sfrac{|z|}{2}} (p)}) \frac{|z|^{1+\sfrac{\lambda}{4}}}{2^k} + C\sum_{i=0}^{k-1}\frac{|z|^{1+\sfrac{\lambda}{4}}}{2^i} \mint_{B_{\frac{|z|^{1+\sfrac{\lambda}{4}}}{2^i}} (p)} |Dh| + \mint_{B_{|z|^{1+\sfrac{\lambda}{4}}} (p)} h\notag\\
&\stackrel{\eqref{e:lip_diff}}{\leq}
C\, |z|^{1+\sfrac{\lambda}{4}} + C\,\sum_{i=0}^{k-1}\,\left(\int_{B_{|z|^{1+\sfrac{\lambda}{4}}}} |Dh|^2\right)^{\frac12}
+ \frac{C}{|z|^{1+\sfrac{\lambda}{4}}}\left(\int_{B_{2|z|}} |h|^2\right)^{\frac12}\, ,
\end{align}
where we have used the standard Poincar\'e inequality
\[
\left\vert \mint_{B_r} f - \mint_{B_{\frac{r}{2}}} f\right\vert \leq C\, r\,\mint_{B_r}|Df|\,\quad f\in W^{1,2}.
\]
Now choose $k \in \N$ such that
$\frac{|z|^{1+\sfrac{\lambda}{4}}}{2^k} < |z|^{1+\sfrac{\lambda}{4}+I_0} \leq \frac{|z|^{1+\sfrac{\lambda}{4}}}{2^{k-1}}$ (in particular $k\leq |\log|z||$) and use \eqref{e:diff_height} together with \eqref{e:en_diff} to bound
\begin{equation}
h (z,w) \leq  C\,|z|^{1+\sfrac{\lambda}{4}} + C\,|\log |z||\,|z|^{\sfrac{\lambda}{8}} + C\,|z|^{\sfrac{\lambda}{4}}
\leq C\,|z|^{\sfrac{\lambda}{16}}\,,
\end{equation}
This gives that, for a sufficiently small $r$, 
\[
\max_{\partial B_1} \cG (f_r, f_{r/2}) \leq C r^{\sfrac{\lambda}{16}}\, . 
\]
Thus
\[
\max_{\partial B_1} \cG (f_r, f_0) \leq \sum_{k=0}^\infty \cG (f_{r2^{-k}}, f_{r2^{-k-1}}) \leq C r^{\sfrac{\lambda}{16}}\, .
\]

\subsection{Step 3: nontriviality of the limit and other properties} To show that $f_0 \neq Q\a{0}$ it is enough to observe that,
by Theorem \ref{t:decay},
\[
\int_{\de B_1} |f_0|^2=\lim_{r\to 0} \int_{\de B_1} |f_r|^2 =\lim_{r\to 0}\frac{\bH(r)}{r^{2I_0+1}}=H_0> 0.
\]
In order to show that $\etaa\circ f_0 \equiv 0$, we notice that by a simple
slicing argument combined with \eqref{e:media_pesata}
 there exists a sequence of radii $r_k \in [2^{-k-1}, 2^{-k}]$
such that
\begin{align}
\int_{\de B_{r_k}} |\etaa \circ \NN| &\leq 2^{k+1} \int_{B_{2^{-k}}\setminus B_{2^{-k-1}}}
|\etaa \circ \NN| \leq C\, r_k^{\gamma_0}\int_{B_{2^{-k}}}|z|^{\gamma_0-1}|\etaa \circ \NN|\notag\\
&\leq C\, r_k^{\gamma_0+2\eta}\bD(2r_k) \leq C\, r_k^{\gamma_0+2\eta + 2I_0},
\end{align}
from which
\begin{align*}
\int_{\de B_1} |\etaa\circ f_0|
&=\lim_{r_k\to0}\int_{\de B_1} |\etaa\circ f_{r_k}|=
\lim_{r_k\to 0}r_k^{-I_0-1}\int_{\de B_r}
|\etaa\circ \NN|\\
& \leq C\,\lim_{r_k\to 0} r_k^{\gamma_0+2\eta+I_0-1} =0.
\end{align*}
Next we show that $f_0$ takes values in $\R^{\bar n}$.
We start by showing that $f_0$ must take values in $T_0\Sigma=\R^{2+\bar{n}}\times \{0\}$. 
Indeed, if we set $f_r (z,w) := \bar\NN (i_r (z,w))$, using \eqref{e:P1} and $|\NN|(i_r(z,w)) \leq C\, r^{1+\sfrac{\gamma_0}{2}}$ we conclude
\[
\int_{\de B_1} \cG(f_r,\bar f_r)^2\leq \frac{Cr^2}{r^{2I_0+1}}\,\int_{\de B_r} |\NN|^2\leq C r^2 \,,
\]
which implies that $f_0(z,w)\in \I{Q}(T_0\Sigma)$.

Next observe that $f_r (z,w) = \sum_i \a{\NN\,_i (i_r (z,w))}$ has the property that each $\NN\,_i (i_r (z,w))$ is orthogonal to $T_{\Psii (i_r (z,w))} \cM$. In particular, if $|z|=1$ and $r\downarrow 0$, the tangent planes $T_{\Psii (i_r (z,w))} \cM$ converge to $\R^2\times \{0\}$: it follows, by the uniform convergence of $f_r$ to $f_0$, that $f_0 (z,w) = \sum_i \a{(f_0)_i (z,w)}$ for some $(f_0)_i (z,w)$ which are orthogonal to $\R^2\times \{0\}$. We thus conclude that each $(f_0)_i (z,w)$ belongs to $\{0\}\times \R^{\bar n}\times \{0\}$.

\subsection{Step 4: Minimality of $g$} In order to complete the proof of Proposition \ref{p:unique limit} we need to show that $g$ is $\D$-minimizing. Given the homogeneity of $g$ in the radial direction, it suffices to show that there is no $W^{1,2}$ multifunction $h: B_1 \to \Iq (\R^{\bar{n}})$ which has the same trace of $g$ on $\partial B_1$ and less energy on $B_1$. Assume thus by contradiction that there is an $h\in W^{1,2} (B_1, \Iq (\R^{\bar n}))$ such that $h|_{\partial B_1}$ and 
\begin{equation}
\int |Dh|^2 \leq \int |Dg|^2 - \delta
\end{equation}
for some positive $\delta>0$.
Recall the definition of $W^{1,2}$ according to Remark \ref{r:W12}: using the map $\bW$ in there and the functions $h\circ \bW$ and $g\circ \bW$ we can use the theory of \cite{DS1} and assume that $h\circ \bW$ is a $\D$-minimizer on the euclidean disk $D_1 \subset \R^2$.
Observe also that, since $\etaa\circ g \equiv 0$, we must have $\etaa\circ h \equiv 0$ as well. Indeed since $h\circ \bW=g\circ \bW$ on $\de D_1$, we have $\etaa\circ h\circ \bW =\etaa\circ g\circ \bW=0$ on the boundary and considering that
\[
\int_{D_1} \sum_i |D(h_i \circ \bW -\etaa\circ h\circ \bW)|^2\leq \int_{D_1}|D (h\circ \bW)|^2-Q\int_{D_1}|D (\etaa\circ h\circ \bW)|^2\, ,
\]
the minimality of $h\circ \bW$ forces the Dirichlet energy of $\etaa\circ h \circ \bW$ to vanish identically.

Using \eqref{e:media_pesata}, the decay $\bD(r)\leq C\,r^{2I_0}$ and a Fubini-type argument we can find a sequence of radii $s_j\to 0$ such that
\begin{equation}\label{e:bounds}
\int_{\de B_1}|Df_{0}|^2\leq \limsup_j \int_{\de B_1}|Df_{s_j}|^2\leq \limsup_j \frac{\bD'(s_j)}{s_j^{2I_0-1}}\leq C\,.
\end{equation}

We now wish to ``smooth'' $h$, i.e. to approximate it with a sequence of Lipschitz maps $h_\eps$ such that $\etaa\circ h_\eps \equiv 0$,
\begin{gather}
\int_{B_1} |Dh_\eps|^2-|Dh|^2\leq \eps^2 \label{e:est_B1}\\
\int_{\de B_1}\cG(f_0,h_{\eps})^2+\Bigg|\int_{\partial B_1} |Df_0|^2 - |Dh_{\eps}|^2\Bigg|\leq \eps^2 \label{e:est_dB1}\,.
\end{gather}
We would like to appeal to \cite[Lemma 3.5]{DS3}, but there is the slight technical complication that $\gira$ is not regular. We postpone this technical step and continue with the argument assuming the existence of the approximations $h_\eps$.

Next we would like to apply \cite[Lemma 3.6]{DS3} to $h_\eps$ and $\p_{T_0\Sigma }(f_{s_j})=:\bar f_{s_j}$, 
to get a family of competitor functions $(\hat f_{s_j})\subset W^{1,2}(B_1,\I{Q}(\R^{2+\bar n}))$, 
such that $\hat f_{s_j}|_{\de B_1}= \bar f_{s_j} |_{\de B_1})$ and
\begin{gather}
\int_{B_1} |D\hat f_{s_j}|^2\leq \int_{B_1} |Dh_\eps|^2+\eps \int_{\de B_1} \left(|D_\tau h_\eps|^2 +|D_\tau \bar f_{s_j}|^2\right)+\frac{C}{\eps} \int_{\de B_1} \cG(h_\eps,\bar f_{s_j})^2\label{e:est_en_B1}\,,\\
\Lip(\hat f_{s_j}) \leq C\left(\Lip(h_\eps)+\Lip(\bar f_{s_j})+\frac{1}{\eps}\sup_{\de B_1}\cG(\bar f_{s_j},h_\eps)\right)\label{e:est_Lip_B1}\\
\etaa\circ \hat{f}_{s_j} = \etaa\circ \bar{f}_{s_j}
\label{e:est_med_B1_2}\,.                                                  
\end{gather}  
Again, this is not straightforward because \cite[Lemma 3.6]{DS3} is stated for euclidean domains. We postpone this second technical
problem and continue with our argument assuming the existence of $\hat{f}_{s_j}$. 

We are now ready to define our competitor function. We set $\bar \LL_{s_j} (z,w):=s_j^{I_0} \hat f_{s_j} (i_{\sfrac{1}{s_j}}(z,w))$ and, observing that $\bar \LL_{s_j}$ takes value in $\Iq (T_0 \Sigma)$, we use \eqref{e:relazione} to define a corresponding $\LL_{s_j}$, which clearly is a competitor 
$\NN$ in $B_{s_j}$ according to Definition \ref{d:competitors}. Moreover 
\[
\Lip (\LL_{s_j})\leq C \, s_j^{I_0+1} \Lip(\hat f_{s_j}|_{B_1})\stackrel{\eqref{e:lip_diff}}{\leq} C\, {s_j}^{\eta}\,. 
\] 
Therefore we can apply Proposition \ref{p:amin} with $\bar \LL=\bar \LL_{s_j}$. In particular, taking into account Theorem \ref{t:decay} and \eqref{e:bounds}, we conclude that
\[
\int_{B_{s_j}} |D\bar \NN|^2 \leq (1+ Cs_j) \int_{B_{s_j}} |D\bar\LL_{s_j}|^2 + C \bmo^{\sfrac{1}{2}} \int_{B_{s_j}} |z|^{\gamma_0-1} |\etaa\circ \LL_{s_j}| + C s_j^{2I_0+\eta}\, .
\]
Next, recall the inequality \eqref{e:dettagliuccio}:
\[
\int_{B_{s_j}} |z|^{\gamma_0-1} |\etaa\circ \LL_{s_j}| \leq C \int_{B_{s_j}} |z|^{\gamma_0-1} |\etaa\circ \bar \LL_{s_j}|
+ C \int_{B_{s_j}} |z|^{\gamma_0-1} |\bar \LL_{s_j}|^2\, .
\]
By \eqref{e:est_med_B1_2} the first term in the right hand side equals indeed
\[
C \int_{B_{s_j}} |z|^{\gamma_0-1} |\etaa\circ \bar \NN| \leq C s_j^\eta \bD (s_j) \leq C s_j^{2I_0+\eta}\, .
\]
For the second term we use the Poincar\'e inequality
\begin{equation}\label{e:ancora_poinc}
\int_{B_{s_j}} |z|^{\gamma_0-1} |\bar \LL_{s_j}|^2 \leq C s_j^{1+\gamma_0}  \int_{B_{s_j}} |D\bar \LL_{s_j}|^2 
+ C s_j^{\gamma_0} \int_{\partial B_{s_j}} |\bar \LL_{s_j}|^2\, ,
\end{equation}
whose proof will be given in Lemma \ref{l:poinc}.

Using that 
\[
\int_{\partial B_{s_j}} |\bar \LL_{s_j}|^2 = \int_{\partial B_{s_j}} |\bar \NN|^2 =\bH (s_j) \leq C s_j^{2I_0+1}\, ,
\]
we easily conclude that
\begin{equation}\label{e:ci_siamo_quasi}
\int_{B_{s_j}} |D\bar \NN|^2 \leq (1+ Cs_j) \int_{B_{s_j}} |D\bar\LL_{s_j}|^2 + C s_j^{2I_0+\eta}\, .
\end{equation}
Changing variables and dividing by $s_j^{2I_0}$ we infer that
\begin{equation}\label{e:quasi_2}
\int_{B_1} |D \bar{f}_{s_j}| \leq \int_{B_1} |D \hat{f}_{s_j}|^2 + C s_j^\eta\, . 
\end{equation}
Using \eqref{e:est_B1}, \eqref{e:est_dB1} and \eqref{e:est_en_B1}, we conclude
\begin{align*}
\int_{B_1} |D \bar{f}_{s_j}|^2 \leq & \int_{B_1} |Dh|^2 + C s_j^\eta + C \varepsilon + \frac{C}{\varepsilon} \int_{\partial B_1} \cG (f_0, \bar f_{s_j})^2\\
\leq & \int_{B_1} |Dg|^2 -\delta + C s_j^\eta + C \varepsilon + \frac{C}{\varepsilon} \int_{\partial B_1} \cG (f_0, \bar f_{s_j})^2\, ,
\end{align*}
where the constant $C$ is independent of $\varepsilon$. In particular, if we fix $\varepsilon$ sufficiently small and we then let $s_j\downarrow 0$, by the uniform convergence of $f_{s_j}$ to $f_0$ on $\partial B_1$ we conclude 
\[
\limsup_{j\to\infty} \int_{B_1} |D\bar f_{s_j}|^2 \leq \int_{B_1} |Dg|^2 - \frac{\delta}{2}\, .
\]
Since however $f_{s_j} \to g$ in $B_1$, the latter inequality contradicts the semicontinuity of the Dirichlet energy.

\subsection{Step 5: Technical leftovers} First of all we show the existence of the map $h_\eps$ as in \eqref{e:est_B1} and \eqref{e:est_dB1}. We consider $h\circ \bW$, which is defined on the closed unit disk $\bar{D}_1\subset \R^2$. 
We then can apply \cite[Lemma 3.5]{DS3} to the latter map and generate approximations $\hat h_\varepsilon$ which satisfy
the bounds \eqref{e:est_B1} and \eqref{e:est_dB1} with $D_1$ in place of $B_1$ and $h\circ \bW$ in place of $h$. The maps
$h_\varepsilon:= \hat h_\eps \circ \bW$ would then satisfy the desired estimates because of the conformality of $\bW^{-1}$ (which keeps the Dirichlet energy invariant) and its regularity in $B_1\setminus \{0\}$ (which results into the loss of a constant factor in \eqref{e:est_dB1}). However the resulting map would not be Lipschitz because of the singularity of $\bW^{-1}$ in the origin. To overcome this difficulty it suffices to perturb slightly $\hat h_\varepsilon$ so that it is constant in a small neighborhood of the origin. As for the condition $\etaa\circ h_\eps\equiv 0$, this can easily be achieved subtracting the average to whichever extension satisfies \eqref{e:est_B1} and \eqref{e:est_dB1}.

Secondly we show the existence of $\hat{f}_{s_j}$. First of all we observe that the condition \eqref{e:est_med_B1_2} can be easily achieved after we prove the existence of a map which satisfies the other two conditions: as above it suffices to subtract the average
of this map and add back $\etaa\circ \bar{f}_{s_j}$. At this point we observe that it suffices, as above, to compose with the map $\bW$,
apply \cite[Lemma 2.14]{DS1} and \cite[Lemma 3.6]{DS2} and compose the resulting map with $\bW^{-1}$: indeed the latter would coincide with $h_\varepsilon \circ \bW$ on $D_{1-\eps}$ and on the complement $\bW^{-1}$ is regular.

\appendix

\section{Some useful lemmas.} \label{a:proiezione}

The first lemma is a simple version of the Poincar\'e inequality for $W^{1,2}$ functions.

\begin{lemma}\label{l:poinc} There exists a universal constant $C>0$ such that the following two inequalities hold
for every $f\in W^{1,2}(B_r, \Iq)$ with $B_r\subset \gira_{Q}$:
\begin{align}
\int_{B_r} |f|^2\leq C r^2 \,\int_{B_r} |Df|^2+C r\,\int_{\de B_r}|f|^2 \label{e:poinc1}\\
\int_{B_r} |z|^{\gamma_0-1} |f|^2 \leq C\,r^{1+\gamma_0} \int_{B_r} |D f|^2 + C\,r^{\gamma_0} \int_{\partial B_1} |f|^2\, .\label{e:poinc2}
\end{align}
\end{lemma}

\begin{proof}
By approximation we can assume, without loss of generality, that $f$ is Lipschitz and, by scaling, it suffices to show the inequalities \eqref{e:poinc1} and \eqref{e:poinc2} on the ball $B_1$. 
Fixing $|z|=1$ and integrating along rays
\[
|f (r z,r^{1/Q} w)|^2 \leq 2 |f (z,w)|^2 + 2 \int_r^1 |Df (tz, t^{1/Q} w)|^2\, dt\, .
\]
Using radial coordinates we then conclude
\[
\int_{B_1} |z|^{\gamma_0-1} |f|^2 \leq C \int_{\partial B_1} |f|^2 + \int_{\partial B_1} \int_0^1 r^\gamma_0 \int_r^1 
|Df (tz, t^{\sfrac{1}{Q}} w)|^2\, dt\, dr\, dz\, .
\]
Using Fubini the latter integral can be rewritten as
\[
\int_0^1 \int_{\partial B_1} |Df (tz, t^{\sfrac{1}{Q}} w)|^2 \int_0^t r^{\gamma_0}\, dt\, dz\, dr \leq
 \int_0^1 t \int_{\partial B_1}  |Df (tz, t^{\sfrac{1}{Q}} w)|^2\, dz\, dr\, .
\]
This completes the proof of \eqref{e:poinc2}. The proof of \eqref{e:poinc1} is a simple variation of this one and is left to the reader.
\end{proof}

\begin{lemma}\label{l:proiezione} Let $\bar \LL\colon \gira_{\bar Q} \to \I{Q}(\R^{2+\bar n})$ be Lipschitz and consider the map $\LL\colon \gira_{\bar Q} \to \I{Q}(\R^{2+n})$ defined by \eqref{e:relazione}.
Then there exists a constant $C:=C(\|\Psi_0\|_{C^3})>0$ such that
\begin{gather}
\cG(\LL,\bar \LL)(z,w)\leq C\,r\,|\bar{\LL}|(z,w)+ C\, |\bar{\LL}|^2(z,w)\,,\quad \forall (z,w)\in B_r\label{e:P1}\\
 \int_{B_r}|D\LL|^2\leq (1+Cr)\int_{B_r} |D\bar{\LL}|^2 +C\,r\,\int_{\de B_r}|\bar{\LL}|^2\label{e:P2}\, .
\end{gather}
\end{lemma}

\begin{proof}
For what concerns \eqref{e:P1}, observe that $D\Psi(0)=0$ implies $\|D\Psi\|_{L^\infty(B_r)}\leq C r$. Therefore, by the $C^{3}$ regularity of $\Psi$, we get 
\begin{align*}
\cG(\LL,\bar \LL)(z,w)
&= \sum_{j=1}^{Q} |\Psi(\p_0(\Psii)+\bar{\LL}_j)-\Psi(\p_0(\Psii))| (z,w)\\
&\leq \|D\Psi\|(\Psii(z,w))\,|\bar \LL|(z,w)+ \|A_{\Sigma}\|\,|\bar \LL|^2(z,w)\\
&\leq C\,r\,|\bar \LL|(z,w)+C\,|\bar \LL|^2\,.
\end{align*}
An analogous computation gives
\[
\int_{B_r}|D\LL|^2\leq (1+Cr)\int_{B_r} |D\bar{\LL}|^2 +C\,\int_{B_r}|\bar{\LL}|^2
\]
and we conclude \eqref{e:P2} using Lemma \ref{l:poinc}.
\end{proof}

\bibliographystyle{plain}
\bibliography{references-Cal}

\end{document}